\documentclass[11pt, oneside]{article}
\RequirePackage[OT1]{fontenc}
\usepackage{graphicx}                           
\RequirePackage{amsthm,amsmath, amssymb, bm, enumerate, pgf, amscd,amsfonts}

\usepackage{epic,graphpap,graphics,float,tabularx}
\usepackage{url, epsfig, graphicx, color, times}
\usepackage{pgf,tikz}
\usetikzlibrary{arrows}

\numberwithin{equation}{section}
\theoremstyle{plain}
\newtheorem{thm}{Theorem}
\numberwithin{thm}{section}
\newtheorem{prop}[thm]{Proposition}
\newtheorem{lm}[thm]{Lemma}

\newtheorem{claim}[thm]{Claim}
\newtheorem{cor}[thm]{Corollary} 

\theoremstyle{definition} 
\newtheorem{definition}[thm]{Definition} 
 
\numberwithin{example}{section}

\numberwithin{conjecture}{section}
\newtheorem{condition}[thm]{Condition}
\theoremstyle{remark} 
\newtheorem{remark}[thm]{Remark}
\theoremstyle{ass}
\newtheorem{ass}[thm]{Assumption}
\theoremstyle{prob} 
\newtheorem{prob}[thm]{Problem}
\def\R{{\mathbb R}}

\def\N{{\mathbb N}}

\title{
\textsc{
 Semimartingales on Rays, Walsh Diffusions,  
 and   Related  Problems of Control and Stopping} \thanks{~~  Research   supported in part by  the National Science Foundation under  grant   NSF-DMS-14-05210. We are indebted to   Tomoyuki Ichiba  for joint work on related matters and for the initial suggestion on the control problems studied here; and  to   Mykhaylo Shkolnikov for discussions on an earlier work   which inspired part of the present one. }
} 

\vspace{2mm}

\author{IOANNIS KARATZAS
\thanks{~
Department of Mathematics,  Columbia University, New York, NY 10027 (E-mail: {\it ik1@columbia.edu}), and       \textsc{Intech} Investment Management,  One Palmer Square, Suite 441, Princeton, NJ 08542, USA  (E-mail: {\it ikaratzas@intechjanus.com}). 
	}
\and	MINGHAN YAN
\thanks{~
Department of Mathematics,  Columbia University, 
New York, NY 10027 (E-mail: {\it  my2379@math.columbia.edu}).
	} 
}	

\begin{document}
\maketitle 

\begin{abstract}  
\noindent \small  
We introduce a class of continuous planar processes, called ``semimartingales on rays", and develop for them a change-of-variable formula involving   quite general  classes of test functions. Special cases of such planar processes are diffusions     which choose, once they reach the origin, the rays for their subsequent voyage according to a fixed probability measure in the manner of Walsh (1978). We develop existence and uniqueness results   up to an explosion time for these ``Walsh diffusions", study their asymptotic behavior, and develop tests for explosions in finite time. We use these results to find an optimal strategy, in a problem of stochastic control with discretionary  stopping involving  Walsh diffusions.
\end{abstract}


\section{Introduction and Summary}

A pathwise construction was given recently in \cite{IKPY} for   so-called {\it \textsc{Walsh} semimartingales} on the plane. A typical such process is a two-dimensional continuous semimartingale, whose motion away from the origin follows a scalar ``driver semimartingale" $\, U(\cdot)\,$  along rays emanating from the origin. Once at the origin, the process chooses a new ray for its voyage randomly, according to a given probability measure on angles.  When the driver $\,U (\cdot) \,$ is a Brownian motion, this  \textsc{Walsh} semimartingale  becomes the renowned {\it \textsc{Walsh}  Brownian motion,} a process introduced by \textsc{Walsh} (1978) in the epilogue of \cite{W}; it was studied by \textsc{Barlow, Pitman \& Yor} (1989) in \cite{BPY} and by many other authors after them (the introduction of \cite{IKPY} contains a comprehensive   literature survey).  
The recent work \cite{IKPY} established  stochastic integral equations     that these     \textsc{Walsh} semimartingales  satisfy, as well as additional features of their singular nature  
at the origin. Taken together, these equations and properties  gave a \textsc{Freidlin-Sheu-}type  
change-of-variable formula for any processes satisfying them.  
Previous results in this regard include also \cite{FS} for diffusion processes on   graphs, \cite{P} for semimartingales on trees, and \cite{FK}, \cite{HT} for \textsc{Walsh}'s Brownian motion. Local martingale problems for {\it \textsc{Walsh} diffusions,} where the driver   
is an \textsc{It\^o} diffusion, were also considered in \cite{IKPY}.      

In \cite{IKPY} we constructed \textsc{Walsh} semimartingales and showed their many properties, but 
 as we realize now,  we
  did not provide the best definition for them. We present in Section 2 a novel 
  view of \textsc{Walsh} semimartingales, 
  as a {\it subclass} of the newly introduced   
  {\it semimartingales on rays} (Definitions \ref{def: semirays} and \ref{def: Walshsemi}). We extract three characterizing and non-overlapping properties --- continuity in the tree-topology, radial semimartingale property, and angular measure --- of \textsc{Walsh} semimartingales, with semimartingales on rays defined by the first two. This is by far the most stripped-down and portable definition of \textsc{Walsh} semimartingales we can offer; and though semimartingales on rays are initially introduced as a primitive version of \textsc{Walsh} semimartingales, this more general class  may well deserve   further study. A more subtle 
  advantage of the current approach is that  the ``thinned process" $\, R^{X}_{A}(\cdot) \,$ in \eqref{eq: RA}, assumed to be a semimartingale in Theorem 4.1 of \cite{IKPY}, is {\it shown} here to be a semimartingale (Theorem \ref{Thm: Gen}(ii)). 

Our second purpose is {\it relaxing the boundedness requirements, from  ``locally" to ``near the origin"}, for both derivatives of test functions in the resulting stochastic calculus  and for the coefficients of \textsc{Walsh}  diffusions with angular dependence. This relaxation is prompted by our interest in studying related control problems.  Not only does it achieve  our goal, but it also corresponds  naturally  to a simple fact about these ``\textsc{Walsh}-type" processes: except at the origin, their motions are one-dimensional; thus boundedness should only be posed near the origin and along the rays, rather than locally on the plane  as   assumed in \cite{IKPY}. Following this thread, we also define \textsc{Walsh} diffusions with state-space open   only in the tree-topology, then generalize to this setting many   one-dimensional results (Section 5.5 of \cite{KS}; see also \cite{ES1}-\cite{ES3}) to \textsc{Walsh} diffusions: existence, uniqueness, asymptotic behavior, and explosion tests 
(Theorems \ref{Thm: Walshnodrift}, \ref{Thm: Walshgen}, \ref{Thm: XS} and \ref{Thm: Sfinite}). The latter two results are especially interesting in how similar they are to their one-dimensional analogues --- and at the same time simpler and more revealing.

Finally, the power of the approach and of the calculus developed here,  is further illustrated in Section 5. We study there an   optimization problem  involving both {\it control}   and {\it stopping} of a \textsc{Walsh} diffusion on the unit disc   with an absorbing boundary, and 
  for which a ``reward" function is specified.  This can be seen as the analogue in the \textsc{Walsh} setting of the problem studied in \cite{KS1}. By handling interesting new aspects arising  from the  roundhouse singularity  at the origin, we explicitly solve the ``pure" optimal stopping problem  and, quite a bit more surprisingly,  the ``mixed" stochastic   control problem with discretionary stopping, in Theorems \ref{thm: solstop}, \ref{thm: ControlStop} 
  and under very  mild assumptions. We also specify the underlying dynamic programming equations, which are not given rigorous meaning but used as guidelines in finding our optimal strategy.
  
 To summarize: Section 2 gives succinct definitions for semimartingales on rays and for \textsc{Walsh} semimartingales, and extends the stochastic calculus in \cite{IKPY} along two directions, namely, to semimartingales on rays and to   functions with relaxed conditions (Theorems \ref{Thm: Gen} and \ref{prop: itoformula}). Sections 3 and 4 define and study \textsc{Walsh} diffusions  as described in the second main purpose above. Section 5 deals with stochastic control and stopping problems. Proofs of selected technical results can be found in the Appendix, Section 6.

\section{A Stochastic Calculus for Semimartingales on Rays} \label{sec2}

Whenever   a function $\, f \, $ is defined on a subset of $\,\R^{2}\,$, we will write ``$f(r, \theta)$" (or sometimes ``$f_\theta(r  )$") 
to mean its expression in polar coordinates;  we have for example $\, f(r, \theta) = f (x) \,,$ where $\, x = (r \cos \theta, r \sin \theta) \,$ in Euclidean coordinates. We   write $\, \text{arg}(x) \in [0, 2\pi) \,$ for the  argument of a generic vector $\, x \in \mathbb R^{2} \setminus \{ {\bm 0} \}\,$. 

The polar coordinates $\, ( 0, \theta), \,\, \theta \in [0, 2\pi)\,$ are identical and identified with $\, \bm 0 \in \R^{2}\,$. Thus, whenever we define a function $\, f\,$ via polar coordinates as $\, f(r, \theta) \,$, we must make sure $\, f(0, \theta) \equiv f({\bm 0})\,$ is constant.

\subsection{Semimartingales on Rays}
\label{sec2.1}

We shall introduce in this section a class of processes called ``semimartingales on rays"; this class includes the \textsc{Walsh} semimartingales that will be studied later.

Indispensable in the study of such semimartingales is the so-called ``tree-metric"  on the Euclidean plane.  

\begin{definition}\label{def: tree}
We define the {\it tree-metric} (cf. \cite{FK}, \cite{HT})  on the plane as follows:
\begin{equation} 
\label{eq: tree}
\varrho (x_{1},x_{2}) \,:=\, (r_1 + r_2) \,  \mathbf{ 1}_{ \{ \theta_1 \neq \theta_2 \} } + |r_1 - r_2| \,  \mathbf{ 1}_{ \{ \theta_1 = \theta_2 \} }\,, \qquad  x_{1}, \,  x_{2} \in \mathbb{R}^{2},
\end{equation} 
where $\, (r_1, \theta_1)\,$, $\,(r_2, \theta_2)\,$ are the expressions in polar coordinates of  $\, x_{1} \,$ and $\, x_{2}\,$, respectively. 

We shall call {\it tree-topology}   the topology on the plane induced by this   metric.
\end{definition}

\begin{remark}\label{rmk: tree}
It is   checked    that the  
recipe of (\ref{eq: tree}) defines   a metric on the plane. The distance in the tree-metric between two points on the plane, is the shortest distance of going from one point to the other along rays emanating from the origin. Thus, {\it the tree-topology is stronger than the usual topology on the plane.}   
\end{remark}

\begin{prop}\label{prop: semirays}
Assume that a function $\, x: [0, \infty) \rightarrow \R^{2}\,$ is continuous in the tree-topology. Then, whenever $\, \Vert x(t) \Vert \neq 0 \,$ holds for all $\, t  \in  [t_{1}, t_{2}]\,$, the mapping $\, t \mapsto \text{arg}(x(t)) \,$ is constant on $\, [t_{1}, t_{2}]\,$. 
\end{prop}

\begin{proof}
Clearly,  showing that $\, t \mapsto \text{arg}(x(t)) \,$ is constant, is equivalent to showing that $\, t \mapsto \frac{x(t)}{\Vert x(t)\Vert}\,$ is constant. By way of contradiction, let us assume that $\, \Vert x(t) \Vert \neq 0 \,$ holds for all $\, t \in [t_{1}, t_{2}]\,$ but  the ratio $\, \frac{x(t)}{\Vert x(t)\Vert} \,$ is {\it not} constant on the interval $\, [t_{1}, t_{2}]\,$. 
From Remark \ref{rmk: tree}, the function $\, x: [0, \infty) \rightarrow \R^{2}\,$ is also continuous in the usual sense. Thus the mapping $\, t  \mapsto \frac{x(t)}{\Vert x(t)\Vert}\,$ is continuous on $\, [t_{1}, t_{2}]\,$ in the usual sense, so we have
$$\,\frac{x(t_{3})}{\Vert x(t_{3})\Vert} = \frac{x(t_{1})}{\Vert x(t_{1})\Vert}\,\qquad \text{for} \qquad \, t_{3} :=\inf \left\{ t \geq t_{1}: \frac{x(t)}{\Vert x(t)\Vert} \neq \frac{x(t_{1})}{\Vert x(t_{1})\Vert} \right\} < t_{2}\, .$$ 
It follows that  there exists a sequence $\, \{ t_{(n)} \}_{n=1}^{\infty} \subseteq (t_{3}, t_{2}]\,$ with $\, t_{(n)} \downarrow t_{3}\,$ and $\,\frac{x(t_{(n)})}{\Vert x(t_{(n)})\Vert} \neq \frac{x(t_{1})}{\Vert x(t_{1})\Vert} = \frac{x(t_{3})}{\Vert x(t_{3})\Vert}\,,$ therefore also $\, \text{arg}(x(t_{(n)})) \neq \text{arg}(x(t_{3}))\,$. We have then  $$\, \varrho (x(t_{(n)}),x(t_{3})) \,= \,\Vert x(t_{(n)}) \Vert + \Vert x(t_{3})\Vert \geq \Vert x(t_{3})\Vert > 0 \, , \qquad \forall \,\,\, n \in \N\,,$$ contradicting the continuity of $\, x(\cdot)\,$ in the tree-topology.
\end{proof}

Proposition \ref{prop: semirays} shows that  any  process,  which is continuous in the tree-topology, does not change  the ray along which it travels  when away from the origin; any such change  to a new ray can happen only when the process is  at the origin.

\begin{definition} {\bf Semimartingales on Rays:} 
\label{def: semirays}
We place ourselves on a filtered probability space $\,(\Omega, \mathcal{F}, \mathbb{P}), \,$ $\, \mathbb{F} = \{ \mathcal{F}(t) \}_{0 \leq t < \infty} \,$ that satisfies the ``usual conditions", i.e., $\, \mathbb{F}\,$ is right-continuous and $\,\mathcal{F}(0)\,$ contains every $\,\mathbb{P}-$negligible event. On this space, we are given   a continuous scalar    semimartigale $\, U(\cdot)\,$. 

We say that a two-dimensional process $\, X(\cdot)  \,$ is a {\it semimartingale on rays driven by} $\, U(\cdot)\,$, if: 

\smallskip
\noindent
(i) ~~ It is adapted, and is continuous in the tree-topology. 

\noindent
(ii) ~ Its radial part $\, \Vert X(\cdot)\Vert\,$ is the \textsc{Skorokhod} reflection (cf. Section 3.6.C in \cite{KS}) of $\,U(\cdot)\,$, i.e., 
 \begin{equation}
 \label{eq: ||X||}
\Vert X( t)\Vert \, =\, U  ( t) +  \Lambda ( t)\,, \qquad \text{where} \quad 
 \Lambda ( t)= \max_{0 \le s \le \, t} \big( - U(s)\big)^{+}\,, \qquad 0 \le t < \infty\,  .
\end{equation}
\end{definition}

\begin{remark} {\it Terminology:} 
\label{rmk: semirays}
We do not assume explicitly in Definition \ref{def: semirays}, that  $\, X(\cdot)\,$ is a two-dimensional semimartingale;   only its radial part $\, \Vert X(\cdot)\Vert\,$ is clearly seen from (\ref{eq: ||X||}) to be a semimartingale. 

But  $\, X(\cdot)\,$ will indeed turn  out to be a semimartingale, thanks to the  assumption that it is  continuous in the tree-topology.  This fact is implied by the general result of Theorem \ref{Thm: Gen} below. In light of this theorem we   use the terminology ``semimartingale on rays" here, leaving it somewhat unjustified for  the moment. 
\end{remark}

\subsection{A Generalized Change-of-Variable Formula}

The property of ``moving along rays" given by Proposition \ref{prop: semirays},  suggests   considering test functions on the plane that have  good properties only along every ray emanating from the origin.  
In this vein, we  develop a   generalized \textsc{Freidlin-Sheu}-type change of variable formula  in Theorem \ref{Thm: Gen}. 

\begin{definition}
\label{def: D}
Let $\, \mathfrak D\,$ be the class of \textsc{Borel}-measurable  functions $ \, g:    \mathbb R^{2}   \rightarrow \R\,,$ such that 

\smallskip
\noindent
(i) ~~\,for every $\,\theta \in [0, 2 \pi)\,$, the function $\,  r \mapsto g_{\theta}(r)   :=  g(r, \theta)    \,$ is differentiable on $\, [0, \infty),$   and the derivative $\, r \mapsto g_{\theta}^{\prime}(r) \,$ is absolutely continuous on $\, [0, \infty)$;\newpage

\noindent
(ii) ~the function $\, \theta \mapsto g_{\theta}^{\prime}(0+) \,$ is bounded; and

\noindent
(iii) there exist  a real number $\, \eta > 0\,$ and a \textsc{Lebesgue}-integrable function $\, c : (0, \eta] \rightarrow [0, \infty)$  such that $\, | g^{\prime \prime}_\theta (r) | \leq c(r)\,$ holds for all $\, \theta \in [0, 2\pi)\,$ and $\, r \in (0, \eta]  $. 
\end{definition}

\begin{definition}
\label{def: G}
For every given function  $ \, g :    \mathbb R^{2}   
\rightarrow \R\,$ in the class  $\,\mathfrak{D} \,,$ we    set  for every $\, x \in \R^{2} \setminus \{ \bf 0 \}\,$: 
 $$\, g^{\prime} (x)  :=  g^{\prime}_\theta (r),\,\quad \, g^{\prime \prime} (x)  :=  g^{\prime \prime}_\theta (r)\,, \quad \text{where} \,\, (r, \theta)\,\, \text{is the expression of}  \,\, x\, \,  \text{in polar coordinates.}$$
\end{definition}

\begin{prop}\label{prop: D}
Every function $\, g \in \mathfrak{D}\,$ has the following properties:

\smallskip
\noindent
{\bf (i)} ~The mappings $\, (r, \theta) \mapsto g_{\theta}^{\prime}(r) \,$ and $\, (r, \theta) \mapsto g_{\theta}^{\prime\prime}(r) \,$ are \textsc{Borel}-measurable on $\, \R^{2} \setminus \{ \bm 0 \}\,$, and the mapping $\, \theta \mapsto g_{\theta}^{\prime}(0+) \,$ is \textsc{Borel}-measurable on $\, [0,2\pi)$. Also, for the  constant $\, \eta >0 \,$  
in Definition \ref{def: D}\,(iii) for $ g\,$, we have $$\, \sup_{\, 0 <r\leq\eta  \atop \theta \in [0, 2 \pi)} |g_{\theta}^{\prime}(r)| \,< \,\infty \,  .$$   
{\bf (ii)} ~The function $\, g \,$ is continuous in the tree-topology.

\end{prop}
\noindent
{\it Proof:} {\bf (i)} The first claim is a consequence of the measurability of the function $\, g \,$ and of Definition \ref{def: D}(i), because of the definition of derivatives as limits. The second comes from the fact $\, g_{\theta}^{\prime}(r) -g_{\theta}^{\prime}(0+) = \int_{0}^{r} g^{\prime \prime}_\theta (y)\,{\rm d}y\,$ and the requirements (ii) and (iii) in  Definition \ref{def: D}.

\smallskip
\noindent
{\bf (ii)}    By the fact $\, g_{\theta}(r) - g({\bm 0}) = \int_{0}^{r} g^{\prime }_\theta (y)\,{\rm d}y \,$ and the second claim of  {\bf (i)}, the function $\, g \,$ is continuous at the origin in the tree-topology. The continuity at other points is equivalent to the continuity of the functions $\, r \mapsto g_{\theta}(r)\,$ for all $\,\theta \in [0,2\pi)\,$, and this is implied by (i) of Definition \ref{def: D}.\qed

\smallskip
The class $\, \mathfrak D\,$ includes the functions in Definition 4.1 of \cite{IKPY}.  In contrast to that definition, which assumes the derivatives to be locally bounded, {\it here we     only assume  some boundedness near the origin.} The reason why this will suffice for the development of a stochastic calculus for semimartingales on rays, is provided by the following two lemmas; these  supply the keys to the main result of this section, Theorem \ref{Thm: Gen}. 

\begin{lm}\label{lm: finite}
Let $\, f: \mathbb{R}^{2} \rightarrow [0, \infty) \,$ be a \textsc{Borel}-measurable function with the following properties:

\smallskip
\noindent
(i) \, For every $\, \theta \in [0, 2 \pi) \,$, the function $\, r \mapsto f(r, \theta)\,$ is locally integrable on $\, [0, \infty)\,$.

\noindent
(ii) \,There exist  a real number $\, \eta > 0\,$ and a \textsc{Lebesgue}-integrable function $\, c : [0, \eta] \rightarrow [0, \infty)$  such that $\, f(r, \theta)  \leq c(r)\,$ holds  for all $\, \theta \in [0, 2\pi)\,$ and $\, r \in [0, \eta] \,$. 

\smallskip
Then for any semimartingale on rays $\, X(\cdot)\,$ in the context of Definition \ref{def: semirays}, we have 
\[
\mathbb{P} \bigg( \int_{0}^{T} f(X(t)) \, {\rm d} \langle U \rangle (t) < \infty , \quad \forall \,\, 0\leq T < \infty \bigg) \, = \, 1 \,.
\]
\end{lm}

\begin{lm}\label{lm: gprime}
In the context of Definitions \ref{def: semirays} and \ref{def: D}  we have, for every semimartingale  $X(\cdot)$ on rays  and for every function $\, g \in \mathfrak{D} ,$ the properties 
\[
\textit{(i)} \quad \mathbb{P} \bigg( \sup_{0\leq t \leq T , \,\, X(t) \neq {\bm 0}} \big| g{^\prime}\big(X(t)\big)\big| < \infty , \quad \forall \,\, 0\leq T < \infty \bigg) \, = \, 1 \,, \qquad\qquad\qquad\qquad\qquad\qquad\qquad\qquad\qquad\qquad\qquad
\]
\[
\textit{(ii)} \,\,\,\, \mathbb{P} \left( \int_{0}^{T} \mathbf{ 1}_{\{ X(t) \neq \mathbf{ 0}\}} \, \big\vert g^{\prime\prime}\big( X(t) \big)\big\vert \, {\rm d} \langle U \rangle (t) < \infty , \quad \forall \,\, 0\leq T < \infty\right) \, = \, 1 \,.\qquad\qquad\qquad\qquad\qquad\qquad\qquad
\]
\end{lm}

 \medskip
To prove Lemma \ref{lm: finite}, we recall the   {\it (right) local time} $\,L^{\Xi}(T, a)\,$ accumulated at the site $\, a \in \R\,$ during the time-interval $[0,T]$ by  a generic one-dimensional continuous semimartingale $\, \Xi (\cdot)\,$,  namely 
\begin{equation} 
\label{LT}
L^{\Xi}(T, a) \, :=\, \lim_{\varepsilon \downarrow 0} \frac{1}{\, 2\,\varepsilon\, }\int^{T}_{0} {\bf 1}_{\{ a \le \Xi(t) < a+\varepsilon \}} \,  {\mathrm d} \langle \Xi   \rangle (t)\,, \qquad 0 \le T < \infty  \,,\,\, \,\,\,a \in \R \, .  
\end{equation}
From the theory of semimartingale local time (e.g., section 3.7 in \cite{KS}),  the identity\newpage
\begin{equation}
\label{eq: LDens}
\int_{0}^{T} k \big(\Xi (t)\big)\, {\rm d} \langle \Xi   \rangle (t) \, = \, 2 \int_{- \infty}^{\infty} k (a)\, L^{\Xi}(T, a)\, {\rm d} a \, , \qquad 0\leq T < \infty \, 
\end{equation}
holds a.e.\,on the underlying probability space for every \textsc{Borel}-measurable function $\, k : \R \rightarrow [0, \infty)\,$. If, in addition, the continuous semimartingale $\, \Xi (\cdot)\,$ is {\it nonnegative}, then its local time admits the representation 
\begin{equation}
\label{eq: L0}
 L^{\Xi}(\cdot \,, 0) \, =\,  \int^{\cdot}_{0} {\bf 1}_{\{\Xi(t) \, =\,  0\}}\, {\mathrm d}  \Xi  (t) \, .  \qquad
\end{equation}
From now on, we will always write ``$L^{\Xi}(\cdot)$" to denote the semimartingale local time $L^{\Xi}(\cdot \,, 0)$ at the origin. 

\medskip
\noindent
{\it Proof of Lemma \ref{lm: finite}:} By condition (ii) of Lemma \ref{lm: finite}, we have $\, f(x) \leq c (\Vert x\Vert ) \,$ whenever $\, x \in \R^{2}\,$ and $\, \Vert x\Vert \leq \eta\,$. In conjunction with  (\ref{eq: LDens}) and (\ref{eq: ||X||}), we have on the one hand
$$ \int_{0}^{T} f\big(X(t)\big) \, {\bf 1}_{\{ \Vert X(t)\Vert \leq \eta\}}\, {\rm d} \langle U \rangle (t) \leq \int_{0}^{T} c(\Vert X(t)\Vert ) \, {\bf 1}_{\{ \Vert X(t)\Vert \leq \eta\}}\, {\rm d} \langle \Vert X\Vert \rangle (t) = 2 \int_{0}^{\eta} c (r) L^{\Vert X\Vert}(T, r) {\rm d} r \, ;$$  whereas, by   the theory of semimartingale local time (e.g., section 3.7 in \cite{KS}),   the mapping $\, r \mapsto L^{\Vert X\Vert }(T, r, \omega) \,$ is RCLL (right-continuous with left-limits), hence bounded on $\, [0, \eta]$, for $\, \mathbb P-$a.e.$\, \,\omega \in \Omega\,$. Thus,  the integrability of $\, c \,$ gives the $\, \mathbb P-$a.e. finiteness of the last expression above.

\smallskip
Let us  define now for every $\, \varepsilon > 0\,$ the stopping times $\,\tau_{-1}^{\varepsilon} \equiv 0 \,$, $\,\tau_{0}^{\varepsilon} : = \inf \big\{ t\geq 0 : \Vert X(t) \Vert = 0 \big\}$  and 
\begin{equation} 
\label{eq: tau rec}
\tau^{\varepsilon}_{2\ell+1} \, :=\,  \inf \big\{ t > \tau^{\varepsilon}_{2\ell} : \Vert X(t) \Vert \geq \varepsilon \big\} \, , \qquad \tau^{\varepsilon}_{2\ell+2} \, :=\,  \inf \big\{ t > \tau^{\varepsilon}_{2\ell+1} : \Vert X(t) \Vert \, =\, 0 \big\}  
\end{equation}
recursively, for $\,\ell \in \mathbb N_{0}\,.$ With $\, \Theta(\cdot) := \text{arg}(X(\cdot)) \,$, we have on the other hand
\[
\int_{0}^{T} f\big(X(t)\big) \, {\bf 1}_{\{ \Vert X(t)\Vert > \eta\}}\, {\rm d} \langle U \rangle (t) \, \leq \, \int_{0}^{T} f\big(X(t)\big) \, \Big( \sum_{\ell \in \mathbb N_{0}\cup \{ -1\}}  \, {\bf 1}_{(\tau^\eta_{2 \ell + 1} , \, \tau^\eta_{2\ell+2}) }(t)\Big) \, {\rm d} \langle U \rangle (t)
\]
\[
=  \sum_{\{ \ell: \,\, \tau^\eta_{2 \ell + 1} < T\}} \int_{T \wedge \tau^\eta_{2 \ell + 1}}^{T \wedge \tau^\eta_{2 \ell + 2}} f\big(\Vert X(t)\Vert , \Theta(t) \big)\, {\rm d} \langle U \rangle (t)  
=   \sum_{\{ \ell: \,\, \tau^\eta_{2 \ell + 1} < T\}} \int_{T \wedge \tau^\eta_{2 \ell + 1}}^{T \wedge \tau^\eta_{2 \ell + 2}} f\big(\Vert X(t)\Vert , \Theta(\tau^\eta_{2 \ell + 1})\big) \, {\rm d} \langle \Vert X\Vert \rangle (t)
\]
\[
\leq \sum_{\{ \ell: \,\, \tau^\eta_{2 \ell + 1} < T\}} \int_{0}^{T} f\big(\Vert X(t)\Vert , \Theta( \tau^\eta_{2 \ell + 1})\big) \, {\rm d} \langle \Vert X\Vert \rangle (t)\, =\, 2 \sum_{\{ \ell: \,\, \tau^\eta_{2 \ell + 1} < T\}} \int_{0}^{\infty} f\big(r, \Theta( \tau^\eta_{2 \ell + 1})\big) L^{\Vert X\Vert }(T, r)\, {\rm d} r 
\]
\[
= \, 2 \sum_{\{ \ell: \,\, \tau^\eta_{2 \ell + 1} < T\}} \int_{0}^{M(T)} f\big(r, \Theta( \tau^\eta_{2 \ell + 1})\big) L^{\Vert X\Vert}(T, r)\, {\rm d} r \, , \qquad \text{where} \quad M(T): = \max_{0\leq t\leq T} \Vert X(t)\Vert \, .
\]
We have  used (\ref{eq: ||X||}) and Proposition \ref{prop: semirays} for the second equality, and (\ref{eq: LDens}) for the third. The last equality follows from the theory of semimartingale local time.

\smallskip
{\it We claim that the last expression above is a.e.\,\,finite.} Indeed, $\, r \mapsto L^{\Vert X\Vert}(T, r, \omega) \,$ is a.e.\,\,bounded on $\, [0, M(T, \omega)]\,,$ just as before; thus, by condition (i) of Lemma \ref{lm: finite}, each integral in the last expression is a.e.\,finite. Moreover, the set $\, \{ \ell:   \tau^\eta_{2 \ell + 1} < T\} \,$ is a.e.\,\,finite; for otherwise the continuity of the path $\, t \mapsto \Vert X(t, \omega) \Vert \,$ would be violated. The validity of this finiteness   claim follows.

\smallskip
With all the considerations above, Lemma \ref{lm: finite} is   seen to have been established. \qed

\begin{remark}
\label{rmk: 01law}
Lemma \ref{lm: finite} can   be thought of as an analogue of the \textsc{Engelbert-Schmidt} $0$-$1$ law (cf.$\,$\cite{ES1} and Section 3.6.E of \cite{KS}), as it gives a condition   guaranteeing the finiteness of some integral functional of the process $\, X(\cdot)$. In contrast to the necessary and sufficient condition of local integrability considered in the \textsc{Engelbert-Schmidt} $0$-$1$ law, the condition here is only sufficient, due to the difficulty in dealing with the ``roundhouse singularity" at the origin.  
\end{remark}

\noindent
{\it Proof of Lemma \ref{lm: gprime}:}   The claim (ii) is a direct consequence of Lemma \ref{lm: finite} and of Definition \ref{def: D}. For the claim (i), we observe by Definition \ref{def: D} and Proposition \ref{prop: semirays} that  \newpage
\[
\sup_{0\leq t \leq T , \,\, 0< \Vert X(t) \Vert \leq \eta } \big| g^{\prime} \big( X(t) \big) \big| \, \leq  \, \Big(\, \sup_{\theta \in [0, 2\pi) }  g_{\theta}^{\prime} (0+) \Big)  \, + \, \int_{0}^{\eta} c(r) \, {\rm d}r \, < \, \infty\, , \quad \text{a.e.,} \quad \text{and}
\]
\[
\sup_{0\leq t \leq T , \,\, \Vert X(t) \Vert > \eta } \big| g^{\prime} \big( X(t) \big) \big| \, \leq  \, \sup_{\{ \ell: \,\, \tau^\eta_{2 \ell + 1} < T\}} \bigg( \,\sup_{ t \in [ \tau^\eta_{2 \ell + 1},\, \tau^\eta_{2 \ell + 2} \wedge T ]} \big| g_{\Theta (\tau^\eta_{2 \ell + 1}) }^{\prime} \big( \Vert X(t)\Vert \big)\big|  \bigg) \, < \, \infty\, ,   \quad \text{a.e.,}
\]
thanks also to the a.e. finiteness of the set $\, \big\{ \ell:    \tau^\eta_{2 \ell + 1} < T\big\}\,$. The claim (i) follows then.\qed

\bigskip
Now we can state and prove the main result of this section, a generalized \textsc{Freidlin-Sheu}-type identity for semimartingales on rays; it extends 
Theorem 4.1 in \cite{IKPY}.

\begin{thm}
\label{Thm: Gen}
{\bf A Generalized Change-of-Variable Formula:}   Let $\, X(\cdot)  \,$ be a  semimartingale on rays with driver $\, U(\cdot)\,$, in the context of Definition \ref{def: semirays}.

\smallskip
\noindent
{\bf (i)}  
\, Then for every function $\, g \in \mathfrak{D}\,$, the process  $\, g(X(\cdot)) \,$ is a continuous semimartingale  and satisfies 
\begin{equation}
\label{eq: gen}
g \big(X(\cdot)\big) = g \big(X(0)\big) \,+ \int_0^{\, \cdot} \mathbf{ 1}_{\{ X(t) \neq \mathbf{ 0}\}} \left( g^{\prime} \big(X(t)\big) \,    \mathrm{d} U(t)  + \frac{1}{\,2\,} \, g^{\prime \prime} \big(X(t)\big) \, \mathrm{d} \langle U \rangle (t) \right) + V_{g}^{X}(\cdot)\, .
\end{equation}
Here $\, V_{g}^{X}(\cdot)\,$ is a continuous process of finite variation on compact intervals, with
\begin{equation}
\label{VgL}
\big| V_{g}^{X}(t_{2}) - V_{g}^{X}(t_{1})\big| \,\leq \,\, \Big( \sup_{\theta \in [0, 2\pi)} |g_{\theta}^{\prime}(0+)| \Big) \big( L^{\Vert X\Vert }(t_{2}) - L^{\Vert X\Vert }(t_{1}) \big) \, , \qquad \forall \,\, 0 \leq t_{1} < t_{2} < \infty \, .
\end{equation}
{\bf (ii)} \, In particular, for every set $\, A \in \mathcal{B}([0, 2\pi))\,$ and with the recipe   $\, g^{A} (r, \theta) : = r \, {\bf 1}_{A}(\theta) \,,$ we have $\, g^{A}  \in \mathfrak{D}\,;$  therefore, the ``thinned process" $\,R_{A}^{X}(\cdot)\,$ below is a continuous semimartingale:
\begin{equation} 
\label{eq: RA}
\,R_{A}^{X}(\cdot) :=  g^{A} \big(X(\cdot)\big) =  \lVert X (\cdot) \rVert \cdot {\bf 1}_A \big(\text{arg}\big(X(\cdot)\big) \big)  \,.
\end{equation}
{\bf (iii)} \,   Assume   that  there exists a probability measure $\, \bm \nu\,$ on   $\,   \mathcal{B}([0, 2\pi)) $  such that, for every set $\, A \in \mathcal{B}([0, 2\pi))\,,$  the semimartingale local time at the origin for the process $\, R_{A}^{X}(\cdot)\,$ in \eqref{eq: RA}  has   the ``partition property" 
\begin{equation}
\label{eq: LTA}
L^{R_{A}^{X}}(\cdot) \, \equiv \, {\bm \nu} (A) \, L^{ \,\lVert X \rVert}(\cdot) \,.  
\end{equation}
Then for every function $\, g \in \mathfrak{D} \,$,  the decomposition  (\ref{eq: gen}) holds with
\begin{equation}
\label{eq: VXg}
V_{g}^{X}(\cdot) = \, \Big( \int_{0}^{2\pi} g^{\prime}_\theta (0+) \,  {\bm \nu}({\mathrm d} \theta) \Big)\, L^{\Vert X\Vert }(\cdot) \, .
\end{equation}
\end{thm}

\noindent
{\it Proof:} {\bf (i)} We employ a method   similar  to that used in the proof of Theorem 4.1 in \cite{IKPY}. With $\, \mathbb N_{-1} :=  \mathbb N_{0} \cup \{-1\}\,$ and the sequence of stopping times $\,\{\tau_{k}^{\varepsilon}\}_{k \in \mathbb N_{-1}}\,$   defined as in (\ref{eq: tau rec}), we have the decomposition
$$
g\big(X(T)\big) \,= \,g\big(X(0)\big) + \sum_{\ell \in \mathbb N_{-1}} \Big( g\big(X(T \wedge \tau^\varepsilon_{2\ell+2})\big) - g\big(X(T \wedge \tau^\varepsilon_{2\ell+1})\big) 
\Big) 
$$
\begin{equation} 
\label{eq: gXe}
~~~~~~~~~~~~+ \sum_{\ell \in \mathbb N_{0}}   \Big( g\big(X(T \wedge \tau^\varepsilon_{2\ell+1})\big) - g\big(X(T \wedge \tau^\varepsilon_{2\ell})\big) \Big) \, .
\end{equation}

 \noindent
 $\bullet~$ 
 Recalling the notation $\,\Theta(\cdot)   =  \text{arg} (X(\cdot)) \,$, we   write the first summand above as 
\[
\sum_{\ell \in \mathbb N_{-1}}  \Big( g\big(X(T \wedge \tau^\varepsilon_{2\ell+2})\big) - g\big(X(T \wedge \tau^\varepsilon_{2\ell+1})\big)  \Big)  = \sum_{\ell \in \mathbb N_{-1}}  \Big( g_{\theta}(\Vert X(T \wedge \tau^\varepsilon_{2\ell+2})\Vert ) - g_{\theta}(\Vert X(T \wedge \tau^\varepsilon_{2\ell+1})\Vert ) \Big)\Big|_{\theta \, =\, \Theta 
(T \wedge \tau^\varepsilon_{2 \ell +1})} \,  
\]
\begin{equation*} 
~~~~~~~~~~\, =\, \sum_{\ell \in \mathbb N_{-1}} \int^{T\wedge \tau^{\varepsilon}_{2\ell+2}}_{T\wedge \tau^{\varepsilon}_{2\ell+1}} \Big( g^{\prime}_{\theta}(\Vert X(t) \Vert) \,  {\mathrm d} \Vert X(t) \Vert + \frac{1}{\, 2\, } \, g^{\prime\prime}_{\theta}(\Vert X(t) \Vert) \, {\mathrm d} \langle \Vert X \Vert\rangle (t) \Big) \Big|_{\theta \, =\, \Theta 
(t)}  
\end{equation*}
\[
~~~~~\, =\, \int^{T}_{0} \Big( \sum_{\ell \in \mathbb N_{-1}} {\bf 1}_{(\tau^\varepsilon_{2 \ell + 1}, \, \tau^\varepsilon_{2\ell+2}) }(t)\Big) \Big( g^{\prime}\big(X(t)\big) \,{\mathrm d} U(t) + \frac{1}{\, 2\, } \, g^{\prime\prime} \big(X(t)\big) \, {\mathrm d}   \langle U \rangle(t) \Big) .
\]\newpage
For the second equality of this string,  we have used Proposition \ref{prop: semirays} and the generalized \textsc{It\^o}'s rule (cf. Problem 3.7.3 in \cite{KS}; although $\, \theta = \Theta 
(T \wedge \tau^\varepsilon_{2 \ell +1})\,$ is a random variable, a careful look into the proof of the generalized \textsc{It\^o}'s rule will justify its application here). The third equality     is valid  because of (\ref{eq: ||X||}), and of the fact that the process $\, \Lambda(\cdot)\,$ that appears there is flat off the set $\, \{ 0 \leq t < \infty : \Vert X(t) \Vert = 0 \}\,$.  Now with the help of Lemma \ref{lm: gprime}, we let $\,\varepsilon \downarrow 0\,$ and obtain the convergence in probability
\begin{equation}\label{limit1}
\sum_{\ell \in \mathbb N_{-1}}  \Big ( g\big(X(T \wedge \tau^\varepsilon_{2\ell+2})\big) - g\big(X(T \wedge \tau^\varepsilon_{2\ell+1})\big)  \Big) \xrightarrow[\varepsilon \downarrow 0]{} \int_{0}^{T} \mathbf{ 1}_{\{ X(t) \neq \mathbf{ 0}\}} \left( g^{\prime} \big(X(t)\big) \,    \mathrm{d} U(t)  + \frac{1}{\,2\,} \, g^{\prime \prime} \big(X(t)\big) \, \mathrm{d} \langle U \rangle (t) \right) .
\end{equation}
 $\bullet~$  
By Definition \ref{def: D} and Proposition \ref{prop: D}, the  process $\, g(X(\cdot))\,$ is   adapted and continuous. Thus the process
\[
V_{g}^{X}(\cdot): = \, g\big(X(\cdot)\big) - g \big(X(0)\big) - \int_{0}^{\,\cdot} \mathbf{ 1}_{\{ X(t) \neq \mathbf{ 0}\}} \Big( g^{\prime} \big(X(t)\big) \,    \mathrm{d} U(t)  + \frac{1}{\,2\,} \, g^{\prime \prime} \big(X(t)\big) \, \mathrm{d} \langle U \rangle (t) \Big)
\]
is also adapted and continuous, and we have from (\ref{eq: gXe}), (\ref{limit1}) the convergence in probability 
\begin{equation}
\label{limit2}
\, \sum_{\ell \in \mathbb N_{0}}   \Big ( g\big(X(T \wedge \tau^\varepsilon_{2\ell+1})\big) - g\big(X(T \wedge \tau^\varepsilon_{2\ell})\big) \Big) \xrightarrow[\varepsilon \downarrow 0]{} V_{g}^{X}(T)\, .
\end{equation}
From (\ref{eq: gXe})-(\ref{limit2}), the representation (\ref{eq: gen}) follows.

\medskip
\noindent
 $\bullet~$
Let us concentrate now on the summand on the left-hand side of the above display (\ref{limit2}). We recall     the constant $\, \eta > 0 \,$ and the function $\, c \,$ in Definition \ref{def: D} (iii), and note for  $\, 0 < \varepsilon \leq \eta\,$ the decompositions 
\[
\sum_{\ell \in \mathbb N_{0}}  \Big ( g\big(X(T \wedge \tau^\varepsilon_{2\ell+1})\big) - g\big(X(T \wedge \tau^\varepsilon_{2\ell}) \big) \Big) =   \sum_{\{\ell \, : \, \tau^{\varepsilon}_{2\ell + 1} < T\}} \big( g_{\theta} ( \varepsilon)  - g_{\theta} (0) \big)\Big\vert_{\theta \, =\, \Theta(\tau^{\varepsilon}_{2\ell + 1})} + O(\varepsilon)
\]
\begin{equation} 
\label{eq: gXe2}
=\, \sum_{\{\ell \, : \, \tau^{\varepsilon}_{2\ell + 1} < T\}} \Big( \varepsilon g^{\prime}_{\theta}(0+) +  \int^{\varepsilon}_{0} (\varepsilon - r) g_{\theta}^{\prime\prime}(r) {\mathrm d} r  \Big) \Big \vert_{\theta \, =\, \Theta(\tau^{\varepsilon}_{2\ell + 1})} + O(\varepsilon) ,
\end{equation}
where we have used Proposition \ref{prop: D}(i) to obtain the term $\, O(\varepsilon) $. We also have
\[
\bigg \lvert \sum_{\{\ell  : \, \tau^{\varepsilon}_{2\ell + 1} < T\}}\Big(  \int^{\varepsilon}_{0} (\varepsilon - r) \,g_{\theta}^{\prime\prime}(r) {\mathrm d} r \Big) \Big \vert_{\theta  = \Theta 
(\tau^{\varepsilon}_{2\ell + 1})}   \bigg| \, \le  \sum_{ \{\ell  : \, \tau^{\varepsilon}_{2\ell + 1} < T\}} \varepsilon \int_{0}^{\varepsilon} c(r)  {\rm d} r =  \big( \varepsilon N(T, \varepsilon) \big) \int_{0}^{\varepsilon} c(r)  {\rm d} r  \xrightarrow[\varepsilon \downarrow 0]{} 0
\]
in probability, where we   set 
$$\,
N(T, \varepsilon) \,  :=  \,\sharp \, \big\{ \ell \in \mathbb{N}_{-1}: \tau^{\varepsilon}_{2 \ell+1 } < T \big\}
$$ 
and use  Theorem VI.1.10 in \cite{RY} for the convergence $\,\, \varepsilon N(T, \varepsilon) \xrightarrow[\varepsilon \downarrow 0]{} L^{\Vert X\Vert} (T) \,\,$ in probability (the ``downcrossings" representation of local time). This, in conjunction with (\ref{limit2}) and (\ref{eq: gXe2}), gives the convergence 
\begin{equation}
\label{Vg2}
\sum_{\{\ell \, : \, t_{1} \leq \tau^{\varepsilon}_{2\ell + 1} < t_{2}\}} \Big( \varepsilon g^{\prime}_{\theta}(0+)  \Big) \bigg\vert_{\theta \, =\, \Theta(\tau^{\varepsilon}_{2\ell + 1})} \,\, \xrightarrow[\varepsilon \downarrow 0]{} \,\,\, V_{g}^{X}(t_{2}) - V_{g}^{X}(t_{1}) \qquad \text{in probability,}
\end{equation}
for fixed $\, 0 \le t_1 < t_2 < \infty\,.$ On the other hand, with $\, \, C_{g}:= \sup_{\theta \in [0, 2\pi)} |g_{\theta}^{\prime}(0+)|\,$, we have again
\[
\sum_{\{\ell \, : \, t_{1} \leq \tau^{\varepsilon}_{2\ell + 1} < t_{2}\}} \Big( \varepsilon \vert g^{\prime}_{\theta}(0+) \vert \Big) \Big\vert_{\theta \, =\, \Theta(\tau^{\varepsilon}_{2\ell + 1})} \, \leq \, C_{g} \cdot \varepsilon \big( N(t_{2}, \varepsilon) - N(t_{1}, \varepsilon) \big) \xrightarrow[\varepsilon \downarrow 0]{} \, C_{g} \big( L^{\Vert X\Vert }(t_{2}) - L^{\Vert X\Vert }(t_{1}) \big)\, ,
\]
in probability. Together with (\ref{Vg2}),  this last convergence in probability leads to the estimate (\ref{VgL}), which in turn implies that the process $\, V_{g}^{X}(\cdot)\,$ is of finite variation on compact intervals. Thus the process $\, g(X(\cdot))\,$ is a continuous semimartingale, and the last claim of {\bf (i)} is justified. The claim {\bf (ii)} follows readily. 

\smallskip
\noindent
{\bf (iii)} Finally, we   need to argue that the ``partition of local time" property (\ref{eq: LTA}) leads to the representation \eqref{eq: VXg}. By virtue of (\ref{Vg2}), it suffices to show\newpage
\[
\sum_{\{\ell \, : \, \tau^{\varepsilon}_{2\ell + 1} < T\}} \Big( \varepsilon g^{\prime}_{\theta}(0+)  \Big) \Big\vert_{\theta \, =\, \Theta(\tau^{\varepsilon}_{2\ell + 1})}\,\, \xrightarrow[\varepsilon \downarrow 0]{} \,\,\bigg( \int_{0}^{2\pi} g^{\prime}_\theta (0+) \, {\bm \nu}({\mathrm d} \theta) \bigg) \,L^{\Vert X\Vert }(\cdot) \, .
\]
This can be done in exactly the same manner as in the last part of the proof of Theorem 4.1 in \cite{IKPY}, so we refer to that proof for this part. \qed

\begin{definition}{\bf \textsc{Walsh} Semimartingales:} 
\label{def: Walshsemi}
A given semimartingale on rays $\, X(\cdot)\,$ as in   Definition \ref{def: semirays}, which satisfies the ``partition of local time" property (\ref{eq: LTA}) for some probability measure $\, \bm \nu\,$ on    $\,   \mathcal{B}([0, 2\pi)) ,$ will be called   {\it \textsc{Walsh} semimartingale with  driver  $\, U(\cdot)\,$ and  angular measure  $\,\bm\nu\,$.}  
\end{definition}

Let us stress that      \textsc{Walsh} semimartingales    
are also   semimartingales on rays, but the converse need not be  true; cf.\,\,Remark 9.4 in \cite{IKPY}. In fact, Definition \ref{def: semirays}   makes no provision regarding the behavior of $\, X(\cdot)\,$ at the origin --- i.e., about the manner in which $\, X(\cdot)\,$ chooses the next ray    when  it ``tries to  extricate itself  from the origin".

\begin{remark} {\it The Planar Semimartingale Property;   
 \textsc{Walsh} Semimartingales.}  
\label{rmk: gen1}
Theorem \ref{Thm: Gen}   generalizes Theorem 4.1 of \cite{IKPY} to a larger  class   of functions, namely, the class $\,\mathfrak D\,$ of Definition \ref{def: D}.  Its part (i)   provides     results on a larger class of processes, for which the ``partition of local time" property (\ref{eq: LTA}) may not hold. 

With $\, g_{1}(r, \theta) = r \cos \theta, \,\, g_{2}(r, \theta) = r \sin\theta\,$, we deduce from Theorem \ref{Thm: Gen}(i) that a process $\, X(\cdot)\,$ as in  Definition \ref{def: semirays} is indeed a two-dimensional semimartingale. If this semimartingale $\, X(\cdot)\,$ satisfies also the ``partition of local time" property  (\ref{eq: LTA})  for some probability measure $\, \bm \nu\,$ on   $\,   \mathcal{B}([0, 2\pi)) $, then 
\begin{equation}
\label{eq: X1}
g_{1}\big(X(\cdot)\big) = g_{1} \big(X(0)\big) \,+ \int_0^{\, \cdot} \mathbf{ 1}_{\{ X(t) \neq \mathbf{ 0}\}}\,  \cos \big( \text{arg}(X(t)) \big) \,    \mathrm{d} U(t)   \, + \, \gamma_{1}\, L^{\Vert X \Vert} (\cdot) \, ,
\end{equation}
\begin{equation}
\label{eq: X2}
g_{2}\big(X(\cdot)\big) = g_{2} \big(X(0)\big) \,+ \int_0^{\, \cdot} \mathbf{ 1}_{\{ X(t) \neq \mathbf{ 0}\}} \, \sin \big( \text{arg}(X(t)) \big) \,    \mathrm{d} U(t)   \, + \, \gamma_{2} \, L^{\Vert X \Vert} (\cdot) \, ,
\end{equation}
hold by virtue of Theorem \ref{Thm: Gen}(iii), where 
$$\, 
\gamma_{1}: = \int_{0}^{2 \pi} \cos (\theta) \, {\bm \nu}({\rm d} \theta) \,\quad \text{and} \quad   \, \gamma_{2}: = \int_{0}^{2 \pi} \sin (\theta) \, {\bm \nu}({\rm d} \theta) \, .
$$ 
The equations (\ref{eq: X1}), (\ref{eq: X2}) are equivalent to the stochastic integral equations in Theorem 2.1 of \cite{IKPY}, after some slight adjusting of notation. Thus,    Definition \ref{def: Walshsemi} above is consistent with the terminology in \cite{IKPY} for \textsc{Walsh} semimartingales.

By virtue of (\ref{eq: LTA}), the probability measure $\,\bm\nu\,$ captures the ``intensity  of excursions of $\, X(\cdot)\,$ away from the origin"     along the rays in a  given set of angles. Thus, under the property (\ref{eq: LTA}),     when   the process $\, X(\cdot)\,$ finds itself at the origin,  it chooses the next ray for its voyage  according to this ``angular measure" $\,\bm\nu$.    
\end{remark}

\subsection{A Refined Stochastic Calculus}
\label{sec: refine}

A further refinement  of the change-of-variable formula \eqref{eq: gen}, \eqref{eq: VXg} is possible   for a class of test-functions that    extends  the class  $\, \mathfrak{D} \,$ of Definition \ref{def: D}, as follows. This will not be used until Section 5. 


\begin{definition}
\label{def: D3}
Let $\, \mathfrak{C} \,$   be the class of \textsc{Borel}-measurable  
functions $ \, g:    \R^{2} \rightarrow \R\,$,  such that: 

\smallskip
\noindent
(i) ~~\,for every $\,\theta \in [0, 2 \pi)\,$, the function $\,  r \mapsto g_{\theta}(r)   :=  g(r, \theta)    \,$ is the difference of two convex and continuous functions on $\, [0, \infty )\,$, and thus the left- and right-derivatives $\, r \mapsto D^{\pm} g_\theta (r)\,$ exist and are of finite variation on compact subintervals of $\, (0, \infty)\,$;

\noindent
(ii) ~the function $\, \theta \mapsto D^{+} g_\theta(0) \,$ is well-defined and bounded;

\noindent
(iii) ~~there exist a real number $\,\eta >0\,$ and a finite measure $\, \mu \,$ on $\, \big( (0, \eta), \, \mathcal{B}( (0, \eta)) \big)\,$, such that for all $\, \theta \in [0, 2\pi)\,$ and $\, 0 < r_{1} < r_{2} \leq \eta \,$, we have $\, | D^{2} g_\theta( [r_{1}, r_{2}) ) | \leq \mu \big( [r_{1}, r_{2}) \big)\,$. Here we denote by $\, D^{2} g_\theta\,$ the ``second-derivative" measure of $\, g_\theta \,$, i.e.,\newpage $$\, D^{2} g_\theta \big( [r_{1}, r_{2}) \big) = D^{-} g_\theta(r_{2})-D^{-} g_\theta(r_{1})\, \qquad \forall \, 0 < r_{1} < r_{2} <\infty \, .$$
\end{definition}

 For this larger class of functions, we have the following  extension of the \textsc{Freidlin-Sheu}-type change of variable formula  developed in Theorem \ref{Thm: Gen}; its proof is  in the Appendix, Section 6.  The summation that appears in \eqref{eq: ito} right below makes sense, because the summand is nonzero only for countably many $\,\theta$'s; indeed, $\, \Theta (\cdot)\,$ is constant on each excursion interval of $\, \Vert X(\cdot) \Vert \,$ away from the origin, and on each generic path  there are at most countably-many such intervals.

\begin{thm}
\label{prop: itoformula}
We let $\, X(\cdot)\,$ be a semimartingale on rays with angular measure $\,\bm \nu\,$, and recall the notation $\, \Theta (\cdot)   = \text{arg} \big( X(\cdot) \big) $. Then, for any function $\, g \in \mathfrak{C}\,$ as in Definition \ref{def: D3}, the process $\, g(X(\cdot)) \,$ is a continuous semimartingale   and satisfies the \textsc{Freidlin-Sheu}-type identity
\begin{equation}
\label{eq: ito}
g \big(X(\cdot)\big) = g \big(X(0)\big) \,+ \int_0^{\, \cdot} \mathbf{ 1}_{\{ X(t) \neq \mathbf{ 0}\}}  \,\,D^{-} g_{\Theta (t)} \big( \Vert X(t)\Vert  \big) \,  {\mathrm d} \Vert X(t) \Vert \,~~~~~~~~~~~~~~~~~~~~~~~~~~~~~~~~~~~~~~~~~~~~~~~~~~~~~
\end{equation}
\[
~~~~~~~~~~~~~~~~~~ +   \sum_{\theta \in [0, 2\pi) } \int_{0}^{\,\cdot}  \int_{0}^{\infty}  {\bf 1}_{ \{ X(t) \neq {\bm 0}, \,\, \Theta (t) = \theta \} }   \, D^{2} g_\theta ({\rm d} r)\, L^{\Vert X \Vert} ({\rm d}t , r) \, + \,V_{g}^{X}(\cdot) ,
\]
with the process $\, V_{g}^{X}(\cdot)\,$ satisfying (\ref{VgL}). Furthermore, for any function $\, f \,$ as in   Lemma \ref{lm: finite}, we have 
\begin{equation}
\label{eq: density}
\int_{0}^{\,\cdot} f(X(t)) \, {\rm d} \langle \Vert X \Vert \rangle (t) \, = \, 2\, \sum_{\theta \in [0, 2\pi) } \int_{0}^{\,\cdot} \int_{0}^{\infty} {\bf 1}_{ \{ X(t) \neq {\bm 0}, \,\, \Theta (t) = \theta \} } \,    f(r,\theta)   \, {\rm d}r   \, L^{\Vert X \Vert} ({\rm d}t , r)  .
\end{equation}
Finally, if $\, X(\cdot)\,$ is a \textsc{Walsh} semimartingale with angular measure $\,\bm \nu\,$, then (\ref{eq: ito}) holds with
\begin{equation}
\label{eq: VXg2}
V_{g}^{X}(\cdot) = \, \Big( \int_{0}^{2\pi} D^{+} g_\theta(0) \,  {\bm \nu}({\mathrm d} \theta) \Big)\, L^{\Vert X\Vert }(\cdot)\, \, .
\end{equation}
\end{thm}


\section{A Study of Walsh Diffusions with Angular Dependence} 
\label{sec3}

In this section  we  provide conditions   under which existence and uniqueness in distribution hold, up to an explosion time, for processes   we call  \textsc{Walsh} {\it diffusions  with angular dependence.}   In the three subsections that follow we  discuss, respectively,    the basic setting,   the driftless case,  and   the case with drift. 

In the rest of this work, we consider an arbitrary but fixed probability measure $\, \bm \nu\,$ on the space $\, ( [0, 2 \pi),$ $ \mathcal{B}([0, 2\pi))) $, which will always be the ``angular measure" of our \textsc{Walsh} diffusions. 

\subsection{Walsh Diffusions with Angular Dependence; Explosions}

We will consider \textsc{Walsh} diffusions with values in  a \textsc{Borel}  
set  {\it which is open in the tree-topology and contains the origin.} More precisely, we fix  a measurable function $\, \ell : [0, 2\pi) \rightarrow (0, \infty] \,$ which is bounded away from zero, and  consider the set 
\begin{equation}
\label{eq:set_I}
I\,  : = \, \big\{ x \in \R^{2} : \, 0 < \Vert x \Vert < \ell(\text{arg}(x)) \,\,\text{or} \,\,  x = {\bm 0}   \big\} \, = \big\{ (r, \theta): \, 0 \leq r < \ell(\theta), \,\, 0 \le \theta < 2 \pi \big\}
\end{equation}
expressed in Euclidean and   polar coordinates, respectively. We consider also the punctured set $\, \check{I} := \, I \setminus \{ \bm 0 \}, $ as well as the closure $\, \overline{I}\,$   of $\, I\,$ under the tree-topology in the collection of all the ``extended rays"; that is, even when $\, \ell(\theta) = \infty \,$ holds for some $\,\theta$'s, we set  $\, \overline{I}\,= \,\big\{ (r, \theta): \, 0 \leq r \leq \ell(\theta), \,\, 0 \le \theta < 2 \pi \big\}\, . $
  
  \smallskip
Finally, we consider   a strictly increasing  sequence of measurable functions $\,\{ \ell_{n}\}_{n=1}^{\infty}\,,$ where each $\,\ell_{n}: [0, 2\pi) \rightarrow (0, \infty)\,$ is bounded away from zero, and  such that $\, \ell_{n} (\theta) \uparrow \ell (\theta)\,$ as $\, n \uparrow \infty\,$, $\, \forall \, \theta \in [0, 2\pi)\,$. We set $$\, I_{n}\,: = \, \big\{ (r, \theta): \, 0 \leq r < \ell_{n}(\theta) , \,\, 0 \le \theta < 2 \pi\big\}, \,\, \,\,\,\,\forall \,\,\, n \in \mathbb{N}. $$ 
By the generalized \textsc{It\^o} rule (Theorem 3.7.1 in \cite{KS}), we see that (\ref{eq: ||X||}) implies \newpage
\begin{equation}
\label{eq: ||X||2}
\Vert X(\cdot)\Vert = \Vert X(0) \Vert + \int_{0}^{\,\cdot} {\bf 1}_{\{ \Vert X(t) \Vert >0 \} } \, {\rm d}U(t) \, + L^{\Vert X \Vert}(\cdot) \, . \qquad
\end{equation}

To introduce \textsc{Walsh} diffusions, let us fix \textsc{Borel}-measurable functions $\, {\bm b}: \, \check{I} \rightarrow \R \,$ and $\, {\bm s}: \, \check{I} \rightarrow \R \,$  and consider finding a \textsc{Walsh} semimartingale $\, X (\cdot) ,$ with angular measure $\,{\bm \nu}\,$ and 
driven by an \textsc{It\^o} process $\, U (\cdot) \,$, whose instantaneous drift and dispersion depend at any given time $\,t\,$ on the current position   $\, X ( t) \,$ through the functions $\, {\bm b}\,$ and $\, {\bm s}\,$. With this dispensation, the equation  \eqref{eq: ||X||2} becomes  
\begin{equation}
\label{eq: WalshDiff}
\Vert X(\cdot)\Vert = \Vert X(0) \Vert + \int_{0}^{\,\cdot} {\bf 1}_{\{ \Vert X(t) \Vert >0 \} } \big[ {\bm b} \big( X(t) \big) \, {\rm d}t + {\bm s}\big( X(t) \big) \, {\rm d}W(t)\big] \, + L^{\Vert X \Vert}(\cdot)\, .
\end{equation}
It is important to note that  (\ref{eq: WalshDiff})   represents the radial part $\, \Vert X(\cdot)\Vert\,$ as a reflected \textsc{It\^o} process, whose local characteristics depend at each time $t$ on the full position $\, X( t)\,$ and not just the radial part $\Vert X(t) \Vert$. This is the so-called ``angular dependence", as introduced in Section 8 of \cite{IKPY}.

Furthermore, for the sake of uniqueness,  
we impose also the requirement 
\begin{equation}
\label{conditionWalsh}
\, \int^{\cdot}_{0}{\bf 1}_{\{ X(t) \, =\,  {\bm 0} \}} \, {\mathrm d} t \,\equiv \,0 \,, \qquad ~~ 
\end{equation}
since  the property (\ref{eq: WalshDiff}) does permit an arbitrary amount of time to be spent by $\, X (\cdot)\,$ at the origin.    

\smallskip
Following Section 5.5 of \cite{KS}, we define \textsc{Walsh} Diffusions with explosion times, as follows.

\begin{definition} {\bf \textsc{Walsh} Diffusion:} 
\label{def: WalshDiff}
 A \textsc{Walsh} diffusion {\it with state-space  $\, I  ,$ associated with the triple $\,(\bm b ,  \bm s , \bm \nu )\,$ and defined up to an explosion time,}  is a triple $\, (X, W), \, (\Omega, \mathcal{F}, \mathbb{P}), \, \mathbb{F} = \{ \mathcal{F}(t) \}_{0 \leq t < \infty} \,$, such that: 

\smallskip
\noindent
(i) ~~ $\, (\Omega, \mathcal{F}, \mathbb{P}), \, \mathbb{F}= \{ \mathcal{F}(t) \}_{0 \leq t < \infty}  \,$ is a filtered probability space satisfying the usual conditions.

\smallskip
\noindent
(ii) ~~The process $ \, \{ X(t), \, \mathcal{F}(t); \, 0 \leq t < \infty \}\,$ is  adapted, $\, \overline{I}-$valued, and continuous in the tree-topology with $\, X(0) \in I \, $ a.s.; and $\, \{ W(t), \, \mathcal{F}(t); \, 0 \leq t < \infty \}\,$ is a standard one-dimensional Brownian motion. 

\smallskip
\noindent
(iii) ~~With $\,S_{n}:= \inf \big\{ t\geq 0: \, \Vert X(t) \Vert \geq \ell_{n} \big(\text{arg}(X(t)\big) \big\} = \inf \big\{ t \geq 0: \, X(t) \notin I_{n}\,\big\}$, we have
\[
\mathbb{P} \bigg( \int_{0}^{T\wedge S_{n}} {\bf 1}_{\{ \Vert X(t) \Vert >0 \} }\Big( \big|{\bm b}\big(X(t)\big)\big| + {\bm s}^{2}\big(X(t)\big) \Big)\, {\rm d}t < \infty \, , \quad 0\leq T < \infty \,\bigg) \, = \, 1 \, .  
\]
(iv) ~~For every $\, n \in \mathbb{N}\,$, the process $ \, \Vert X(\cdot \wedge S_{n})\Vert\,$ is a semimartingale that satisfies
\[
\mathbb{P} \bigg( \Vert X(T\wedge S_{n} )\Vert = \Vert X(0) \Vert + \int_{0}^{T\wedge S_{n}} {\bf 1}_{\{ \Vert X(t) \Vert >0 \} } \big[ {\bm b} \big( X(t) \big) \, {\rm d}t \, +\, {\bm s}\big( X(t) \big) \, {\rm d}W(t)\big] \, \qquad\qquad\qquad\qquad
\]
\[
\qquad\qquad\qquad \qquad\qquad \qquad\qquad \qquad + \, L^{\Vert X (\cdot \wedge S_{n}) \Vert}(T) , \quad 0 \leq T < \infty \bigg) = \, 1 \, .
\]
(v) ~~~We have   $\, \int^{\cdot}_{0}{\bf 1}_{\{ X(t) \, =\,  {\bm 0} \}} \, {\mathrm d} t \,\equiv \,0 \, $, and for every $\, n \in \mathbb{N}\,$ the ``partition of local time property" 
\[
~~~~~~~~~~~~~~~~~~~~~~~~~~~   L^{R_{A}^{X}(\cdot \wedge S_{n})} (T) \, = \, {\bm \nu}(A) \, L^{ \lVert X(\cdot \wedge S_{n}) \rVert}(T) \,, \qquad \forall \,\, A \in \mathcal B ([0, 2\pi))\, , \,\,  
   \,\, \mathbb{P-}\text{a.s.} ~~~~~~~~~ ~~~~~~~~~~~~~\qed
\]

\smallskip
Abusing terminology slightly, we shall  also call the state-process $\, X (\cdot)\,$ a {\it \textsc{Walsh} diffusion,} omitting the underlying probability space and Brownian motion.   We shall refer to   
\begin{equation}
\label{eq: S}
S : = \, \lim_{n \rightarrow \infty} S_{n} 
\end{equation}
as the {\it explosion time}  of $\, X (\cdot)$ from $\, I  $, and stipulate that $\, X(t) = X(S)\, , \,\, S \leq t < \infty\,$. We note that the assumption of continuity of $\, X (\cdot) \,$ on $\, \overline{I} ,$ in the topology induced by the tree-metric, implies that
\begin{equation}
\label{eq: explos}
S = \inf \{ t \geq 0: \, X(t) \notin I \}\,, 
\qquad X(S) \in  \big \{ (r, \theta): r = \ell (\theta) , \,\, 0 \le \theta < 2 \pi\big \} \,\,\, \text{a.e. on} \,\,\, \{ S<\infty\} .
\end{equation}
\end{definition}

\begin{remark}
By Theorem \ref{Thm: Gen}(ii), the processes $\, R_{A}^{X}(\cdot \wedge S_{n})\, , \, A \in \mathcal B ([0, 2\pi)), \, n \in \mathbb{N} \,$ are   continuous semimartingales. Moreover, the sets $\, I_{n}, \, n \in \mathbb{N}\,$ and $\, I\,$ are open in the topology induced by the tree-metric;  thus, the continuity of $\, X(\cdot)\,$ in the above topology  implies that $\, S_{n}, \, n \in \mathbb{N}\,$ and $\, S\,$ are stopping times. 
We do not assume 
continuity up to time $\, \infty\,$, thus $\, X(S) \,$ may not be defined on the event $\, \{ S= \infty\}\,$. 
\end{remark}

\subsection{The Driftless Case: Method of Time-Change}
\label{sec: timechange}

We study here \textsc{Walsh} diffusions with drift  $\, {\bm b} \equiv 0 \,$ and state-space $\, I =  \R^{2} = \{ (r, \theta): \, 0 \leq r < \infty , \,\, 0 \le \theta < 2 \pi\}\,$. To employ the method of time-change, we shall first establish in our setting   results  analogous to the \textsc{Dambis-Dubins-Schwarz} representation for local  martingales, and to the non-explosion property (Problem 5.5.3 in \cite{KS}).

\begin{definition} {\bf \textsc{Walsh} Brownian Motion:} 
\label{def: WalshBM}
  A    \textsc{Walsh} 
  semimartingale     $\, X(\cdot) $  will be called   {\it \textsc{Walsh} Brownian motion}, if its driver   $\, U(\cdot)\equiv B(\cdot)\,$  
    is  a Brownian motion;  see Definitions \ref{def: Walshsemi} and \ref{def: semirays}. 

  \end{definition}

  This terminology is  consistent with the   construction of the \textsc{Walsh} Brownian motion in \cite{BPY}; this is  thanks to Proposition 7.2 in \cite{IKPY},  and to Remark \ref{rmk: gen1} here. We note at this point that a \textsc{Walsh} Brownian motion is also a \textsc{Walsh} diffusion  with   state-space $\, I = \R^{2} \,,$ $\,{\bm b} \equiv 0\,,$   $\,{\bm \sigma} \equiv 1\,,$ and $\, \mathbb{P} (S = \infty) =1\,.$

\begin{prop}
\label{prop: DDS}
{\bf A \textsc{Dambis-Dubins-Schwarz}-Type Representation:}~~ Let $\, X(\cdot)\,$ be a \textsc{Walsh} semimartingale driven by a continuous local martingale $\, U(\cdot)\,,$ and with angular measure $\, \bm \nu \,$. 

There exists then, on a possibly extended probability space,  a \textsc{Walsh} Brownian motion $\, Z(\cdot)\,$ with the same angular measure  and with the property  $\, X(\cdot) = Z\big(\langle U \rangle (\cdot) \big)\,$.
\end{prop}

\begin{proof}
Let us  assume first that $\, \langle U \rangle (\infty) = \infty \,$. Define
\begin{equation}
\label{eq: timechange}
T(s): = \,\inf \{ t \geq 0 : \langle U \rangle (t) > s \} \, , \quad Z(s) :=\, X\big(T(s)\big)\, , \quad \mathcal{G}(s) : = \, \mathcal{F}\big(T(s)\big)\, , \quad 0 \leq s < \infty \, .
\end{equation}
Recall that $\, U(\cdot)\,$ is a continuous local martingale. Thus, by the proof of Theorem 3.4.6 in \cite{KS}, we have:

\smallskip
\noindent
(i) ~~ With $\, B(\cdot) := U\big(T(\cdot)\big) \,$, the process $\{ B(s) , \, \mathcal{G}(s) , \, 0 \leq s < \infty \}\,$ is   Brownian motion, and $\, U(t) = B\big(\langle U \rangle (t)\big)\, , \,\, 0\leq t < \infty \,$. 

\noindent
(ii) ~~There exists $\, \Omega^{\ast}  
\in \cal F\,$ with $\, \mathbb{P} (\Omega^{\ast}) =1 \,$, such that for every $\,   \omega \in \Omega^{\ast} \,$, we have
\begin{equation}
\label{Uconstant}
\langle U \rangle (t_{1}, \omega ) = \langle U \rangle (t_{2} ,\omega) \,\,\,\,\text{for some} \,\, 0\leq t_{1} < t_{2}< \infty \quad \Rightarrow \quad t \mapsto U(t, \omega) \,\,\text{ is constant on} \,\, [t_{1} , t_{2}] .
\end{equation}
Since $\, X(\cdot)\,$ is continuous in the tree-topology, we see from Proposition \ref{prop: semirays} that the constancy of $\, X(\cdot)\,$ on some interval $\, [t_{1},t_{2}]\,$ is implied by the constancy of $\, \Vert X(\cdot)\Vert\,$ on $\, [t_{1},t_{2}]\,$, which by (\ref{eq: ||X||}) can be implied by the same constancy of $\, U(\cdot)\,$. Thus the above property (ii) is still true if we replace (\ref{Uconstant}) by 
\begin{equation}
\label{Xconstant}
\langle U \rangle (t_{1}, \omega ) = \langle U \rangle (t_{2} ,\omega) \,\,\,\,\text{for some} \,\, 0\leq t_{1} < t_{2}< \infty \quad \Rightarrow \quad t \mapsto X(t, \omega) \,\,\text{ is constant on} \,\, [t_{1} , t_{2}] .
\end{equation}
In the spirit of Problem 3.4.5(iv) in \cite{KS}, this implies the continuity in the tree-topology of the process $\, Z(\cdot) :=\, X(T(\cdot))\,$. Moreover, we observe from (\ref{eq: ||X||}) that 
\[
\Vert Z(s) \Vert = \, U\big(T(s)\big) + \max_{0 \le t \le \, T(s)} \big( - U(t)\big)^{+} = U\big(T(s)\big) + \max_{0 \le t \le \, s} \Big( - U\big(T(t)\big)\Big)^{+}\, = B(s) + \max_{0 \le t \le \, s} \big( - B(t) \big)^{+} .
\]
The second equality uses the fact,   implied by (\ref{Uconstant}),  that $\, U(\cdot)\,$ is constant on $\, [ T(t-) , T(t)]\,$ for every $\, t \,$.

Finally, we claim that  {\it the    ``partition of local time"  
property (\ref{eq: LTA}) of $\, X(\cdot)\,$ is inherited by $\,Z(\cdot)\,$.} Indeed,   
\[
\, R_{A}^{Z}(\cdot) =  \lVert Z (\cdot) \rVert \cdot {\bf 1}_A \big(\text{arg}\big(Z(\cdot)\big) \big) = R_{A}^{X}\big( T(\cdot)\big) 
\]
 is continuous, in the same way   
$  Z(\cdot) = X(T(\cdot)) $ is. 
Then by Theorem \ref{Thm: Gen} and time-change (Proposition 3.4.8 in \cite{KS}), we obtain that $\, R_{A}^{Z}(\cdot)\,$ is a continuous semimartingale  of the filtration $\, \{ \mathcal{G}(s) \}\,$, and that $\, \langle R_{A}^{Z}\rangle (\cdot) \equiv \langle R_{A}^{X} \rangle \big( T(\cdot)\big) \,$. 
Now it is easy to use (\ref{LT}) to obtain $\, L^{R_{A}^{Z}}(\cdot) \equiv L^{R_{A}^{X}}\big(T(\cdot)\big)\,;$ in particular, $\, L^{\Vert Z \Vert}(\cdot) \equiv L^{\Vert X \Vert}\big(T(\cdot)\big)\,$. Thus (\ref{eq: LTA}) implies $\, L^{R_{A}^{Z}}(\cdot) \equiv {\bm \nu}(A)\, L^{\Vert Z \Vert}(\cdot) \,$, which is the claim.\newpage

It is clear at  this point that $\, Z(\cdot)\,$ is a \textsc{Walsh} Brownian motion with the same angular measure $\, \bm \nu\,$ as $\, X(\cdot)$, and that $\, X(\cdot) = Z\big(\langle U \rangle (\cdot) \big)\,$ holds, thanks to (\ref{Xconstant}).

\smallskip
\noindent
$\bullet~$ Next, we consider the case $\, \mathbb{P} \big( \langle U \rangle (\infty) < \infty\big) > 0 \,$. We shall argue   this case heuristically, as a rigorous argument is straightforward but  laborious. On the event $\,\{\langle U \rangle (\infty) < \infty\}\,$,   the limit $\, \lim_{t \rightarrow \infty} U(t) \,$ exists; therefore, so do the limits $\, \lim_{t \rightarrow \infty} \Vert X(t)\Vert \,$ and   $\, \lim_{t \rightarrow \infty}  X(t) \,$, thanks to (\ref{eq: ||X||}) and the continuity of $\, X(\cdot)\,$ in the tree-topology. It follows that  (\ref{eq: timechange}) is still well-defined; the only problem is that $\, Z(\cdot)\,$ need not be a \textsc{Walsh} Brownian motion anymore: it ``runs out of gas" from the time $\, \langle U \rangle (\infty)\,$ onwards, as does $\, B(\cdot) \,$.

\smallskip
We deal with this issue as follows: 
On the event $\,\{\langle U \rangle (\infty) < \infty\}\,$, we   keep $\, Z(\cdot)\,$ running on the time interval $\,[\langle U \rangle (\infty) , \infty )\,$, by first redefining $\, B(\cdot)\,$ on $\,[\langle U \rangle (\infty) , \infty )\,$ to make it a Brownian motion, as described in  
Problem 3.4.7 of \cite{KS}; then   following  the ``folding and unfolding" scheme in the proof of Theorem 2.1 in \cite{IKPY}, to construct pathwise a \textsc{Walsh} Brownian motion $\, Z(\langle U \rangle (\infty) + \cdot)\,$ with angular measure $\, \bm \nu\,$ and driven by $\, B(\langle U \rangle (\infty)+\cdot )\,$. This ``continued" process $\, Z(\cdot)\,$   satisfies all the required properties.
\end{proof}

We also have the following   result, regarding the absence of explosions  for \textsc{Walsh} diffusions with  $\, {\bm b} \equiv 0 \,$ and state-space $\, I = \R^{2} \,$. Its proof is in the Appendix, Section 6.

\begin{prop}
\label{nonexplo}
Suppose $\, X (\cdot) \,$ is a \textsc{Walsh} diffusion associated with the triple $\, ({\bf 0}, {\bm s}, {\bm \nu}) \,$ on the Euclidean plane $\,  \R^{2} = \{ (r, \theta): \, 0 \leq r < \infty , \,\, 0 \le \theta < 2 \pi\}\,$ and defined up to an explosion time $\, S$. Then   $\, S= \infty $ a.e. 
\end{prop}

Now we can state the existence-and-uniqueness result for a   \textsc{Walsh} diffusion without drift. As in the scalar case, we recall Remark \ref{rmk: 01law} and define the   sets 
\begin{equation}
\label{sets: I,Z}
\mathcal{I}({\bm s}) : = \Big\{ (r, \theta)\in \check{\R}^{2} : \int_{-\varepsilon}^{\varepsilon} \frac{{\rm d} y}{{\bm s}^{2}(r+y , \theta) } \, = \infty, \, \, \forall \, \varepsilon \in (0,r) \Big\} \,, \quad \mathcal{Z}({\bm s}) : = \big\{ x \in \check{\R}^{2} :   {\bm s}(x) = 0 \big\} .  
\end{equation}
 Since the \textsc{Engelbert-Schmidt} $0$-$1$ law is critical for establishing the one-dimensional existence-and-uniqueness result, we need to impose the following additional condition, in order to ensure that the above two sets are both bounded away from the origin, and that the integral process $\, T(\cdot)\,$ in the proof Theorem \ref{Thm: Walshnodrift} does not explode when the \textsc{Walsh} Brownian motion considered there is near the origin.

\begin{condition}\label{sigma}
{\it 
There exist an $\, \eta > 0\,$ and an  
integrable function $\, c : (0, \eta] \rightarrow [0, \infty)$, such that 
$$\,
 \qquad  \qquad  \qquad 
 \frac{1}{{\bm s}^{2}(r, \theta)}  \leq c(r)\, \qquad \text{holds for all} ~~\,\,\,\theta \in [0, 2\pi)\,,\,\,\,r \in (0, \eta] \,. \qquad   \qquad  \qquad  \qed
 $$
 }
\end{condition}
 
  Under this condition,   the following existence-and-uniqueness result, for a   \textsc{Walsh} diffusion without drift, is a two-dimensional analogue of  Theorem 5.5.4 in \cite{KS}; its proof is also in the Appendix.

\begin{thm}\label{Thm: Walshnodrift}
Suppose the function $\, {\bm  s}: \check{\R}^{2} \rightarrow \R \,$ satisfies  Condition \ref{sigma}. Then, for any given initial distribution $\, \bm\mu\,$ on $ \mathcal{B} (\R^{2})   ,$  there exists a non-explosive and unique-in-distribution \textsc{Walsh} diffusion $\, X (\cdot)\,$ with values in $\, \R^{2} \,$ and associated with the triple $\, ({\bf 0}, {\bm  s}, {\bm \nu})  ,$   if and only if $\, \, \mathcal{I}({\bm s}) = \mathcal{Z}({\bm s})\,$.
\end{thm}

\begin{remark}\label{rmk: Walshnodrift}
Assuming $\,  \mathcal{I}({\bm s}) = \mathcal{Z}({\bm s})\,$ and Condition \ref{sigma}, the \textsc{Walsh} diffusion in Theorem \ref{Thm: Walshnodrift} becomes motionless once it hits the set $\,  \mathcal{I}({\bm s})\,,$ but keeps   moving before that time. This can be seen in the same spirit as in Remarks 5.5.6, 5.5.8 of \cite{KS}.
\end{remark}

\subsection{The General Case: Removal of Drift by Change of Scale}
 
 Let us move now on to the study of \textsc{Walsh} diffusions with  drift, via the   method of ``removal of drift" followed by reduction to the driftless case of the previous subsection.  
 We  recall the set $\, I\,$ 
 from  \eqref{eq:set_I}, with a function $\, \ell : [0, 2 \pi) \rightarrow (0, \infty] \,$ which is measurable and bounded away from zero. We recall also the class $\, \mathfrak{D}\,$ of functions in Definition \ref{def: D}, and   adjust it presently to ``fit" the domain $\, I $.
 
\begin{definition}\label{def: D2}
Let $\, \mathfrak D_{I} \,$ be the class of \textsc{Borel}-measurable  functions $ \, g:    I   \rightarrow \R\,$ satisfying: \newpage

\noindent
(i) ~for every $\,\theta \in [0, 2 \pi)\,$, the function $\,  r \mapsto g_{\theta}(r)   :=  g(r, \theta)    \,$ is differentiable on $\, [0, \ell(\theta) )\,$  and the derivative $\, r \mapsto g_{\theta}^{\prime}(r) \,$ is absolutely continuous on $\, [0, \ell(\theta) )\,$;

\noindent
(ii) ~the function $\, \theta \mapsto g_{\theta}^{\prime}(0+) \,$ is bounded;

\noindent
(iii)  there exist a constant $\,\eta \,$ with $\, 0 < \eta < \inf_{\theta \in [0, 2\pi)} \ell(\theta) \,$ and a \textsc{Lebesgue-}integrable function $\, c : (0, \eta] \rightarrow [0, \infty)$, such that for all $\, \theta \in [0, 2\pi)\,$ and $\, r \in (0, \eta] \,$, we have $\, | g^{\prime \prime}_\theta (r) | \leq c(r)\,$. 
\end{definition}

\begin{remark}\label{rmk: DIextend}
The class $\,\mathfrak{D}_{I}\,$ of functions in Definition \ref{def: D2} can be generalized in the same manner as in Definition \ref{def: D3},  to an extended class that we shall denote by $\,\mathfrak{C}_{I}\,$. Then it is also easy to generalize Theorems \ref{Thm: Gen} and \ref{prop: itoformula} to \textsc{Walsh} semimartingales with values in $\, I \,,  $ and to functions in $\,\mathfrak{D}_{I}\,$ and $\,\mathfrak{C}_{I}\,$, respectively. We will later apply these adjusted versions, still under the names of Theorems \ref{Thm: Gen} and \ref{prop: itoformula}.  
\end{remark}

We shall work throughout this subsection in the most general setting of Definition \ref{def: WalshDiff} for \textsc{Walsh} diffusions, and impose the following condition on the functions $\, {\bm b}: \, \check{I} \rightarrow \R \,$ and $\, {\bm s}: \, \check{I} \rightarrow \R \,$. 

\begin{condition}
\label{bsigma}
{\it
(i) ~  We have $\, {\bm s} (x)  \neq  0 \,$, $\, \forall \, x \in \check{I} \,$.

\noindent
(ii)~~ For every fixed $\, \theta \in [0, 2 \pi)\,$, both functions below are   locally integrable on $\, (0, \ell (\theta)):$ $$\, r \, \longmapsto \, \frac{{\bm b}(r, \theta)}{{\bm s}^{2}(r, \theta)} \qquad \text{and} \qquad r \, \longmapsto  \,  \frac{1}{{\bm s}^{2}(r, \theta)} \,.
$$  

\noindent
(iii) ~ There exists a constant $\,\eta \,$ with $\, 0 < \eta < \inf_{\theta \in [0, 2\pi)} \ell(\theta) \,$, such that $$\, \sup_{\, 0 <r \leq\eta  \atop \theta \in [0, 2 \pi)} \left( \frac{1+ \vert {\bm b}(r, \theta)\vert}{{\bm s}^{2}(r, \theta)}\right) \, < \infty \, .$$  }
\end{condition}

Under this Condition \ref{bsigma}, we define the {\it radial scale function} $\, p : \, I \rightarrow [0, \infty) \,$ by
\begin{equation}
\label{eq: p}
p(r, \theta) =p_{\theta}(r) \, :=\,  \int^{r}_{0} \exp \Big( - 2 \int^{y}_{0} \frac{ {\bm b}(z, \theta)\, }{\,   {{\bm s}}^{2}(z, \theta)\, } {\mathrm d} z \Big) {\mathrm d} y \, ,  \qquad (r, \theta) \in I \,, 
\end{equation}
as well as the {\it scale mapping} $\, \mathcal{P}: \, I \rightarrow J ,$ where 
\begin{equation}
\label{eq: P}
J : = \big\{ (r, \theta): \, 0 \leq r < p_{\theta}(\ell(\theta) -) , \,\, 0 \le \theta < 2 \pi\big\} \qquad \text{and} \qquad 
 \mathcal{P}(r, \theta) : = \, \big( p(r, \theta) , \theta \big) , \quad (r, \theta) \in I \, .
\end{equation}
These are well-defined, as $\, p(0, \theta) \equiv 0 \,$ and $\,\mathcal{P}(0, \theta) = (0, \theta) \equiv \bm 0 \,$. Moreover, since the mapping $\, r \mapsto p(r, \theta) \,$ is strictly increasing on $\, [0, \ell(\theta)) \,$ for every $\, \theta \in [0, 2\pi)\,$, we see that the mapping $\, \mathcal{P} \,$ is invertible; we denote by $\,\mathcal{Q} : J \rightarrow I \,$  its inverse. From (\ref{eq: P}), we have the representation 
\begin{equation}
\label{eq: Qmap}
\mathcal{Q}(r, \theta) = \, \big( q(r, \theta) , \theta \big) , \qquad (r, \theta) \in J
\end{equation}
where  $\, q : J \rightarrow [0, \infty) \,$ is a function with the property  that, for every $\, \theta \in [0, 2\pi) \,$, the mappings $\, r \mapsto p_{\theta}(r) \,$ and $\, r \mapsto q_{\theta}(r) := q(r, \theta) \,$ are inverses of each other. 

We   extend $\,  \mathcal{P}\,$ to $\, \overline{I}\,$ and $\,\mathcal{Q}\,$ to $\, \overline{J}\,$ continuously, with the aid of Proposition \ref{prop: pq}(iii) below; here $\,I\,$ and $\,J\,$ are equipped with the tree-topology, and closures are as described at the beginning of Subsection 3.1.  


The following fact can be checked in a very direct manner; its proof is omitted.  

\begin{prop}\label{prop: pq}
Assume Condition \ref{bsigma} holds for $\, {\bm b}: \, \check{I} \rightarrow \R \,$ and $\, {\bm s}: \, \check{I} \rightarrow \R \,$. Then:

\noindent
{\bf (i)} ~~ The mapping $\, \theta \mapsto p_{\theta}(\ell(\theta)-)\,$ is bounded away from zero, thus $\, J\,$ is open in the tree-topology.

\noindent
{\bf (ii)} ~ We have $\, p \in \mathfrak{D}_{I}\,$, $\, q \in \mathfrak{D}_{J} \,$, $\, p_{\theta}(0) \equiv 0 \equiv q_{\theta}(0)\,$, $\, p_{\theta}^{\prime}(0+) \equiv 1 \equiv q_{\theta}^{\prime}(0+)\,$, and that 
\[
p_{\theta}^{\prime\prime} (r) \, = \, -\frac{2\, {\bm b} (r, \theta)}{{\bm s}^{2}(r, \theta)}  p_{\theta}^{\prime} (r)\, , \quad q_{\theta}^{\prime}(r) \, =\, \frac{1}{ \,  p_{\theta}^{\prime}\big(q_{\theta}(r)\big)\,} \, , \quad  q_{\theta}^{\prime\prime}(r) \, =\, \frac{ 2\, {\bm b}( q_{\theta}(r), \theta)}{ \,     \,{\bm  s}^{2} \big( q_{\theta}(r), \theta\big) \cdot \big( p_{\theta}^{\prime}(q_{\theta}(r)) \big)^2} \,\, 
\]
hold for every $\,\theta \in [0, 2\pi)\,$ and a.e. $\, r \in (0, \ell(\theta) ) \,$.

\noindent
{\bf (iii)} ~The mappings $\, \mathcal{P}: \, I \rightarrow J \,$ and $\,\mathcal{Q} : J \rightarrow I \,$ are both continuous in the tree-topology. 
\end{prop}

We have then the following ``removal-of-drift" result.\newpage

\begin{prop}
\label{prop: scaling}
Assume that Condition \ref{bsigma} holds, and consider 
the function $\, \widetilde{{\bm s}} : \, \check{J} \rightarrow \R\,$   given by 
\begin{equation}
\label{eq: sigmatilde}
\widetilde{{\bm s}} ( r, \theta) := \,   p^{\prime}_{\theta} \big( q_{\theta}( r)\big) \, {\bm s} \big( q_{\theta}(r) , \, \theta \big) \,, \qquad  (r, \theta) \in \check{J} .
\end{equation}
 If $\, X (\cdot) \,$ is a \textsc{Walsh} diffusion with state-space $\, I  ,$ associated with the triple $\, ({\bm b}, {\bm s}, {\bm \nu}) \,$ and defined up to an explosion time $\, S\,,$  then $ \, Y(\cdot): = \mathcal{P} (X(\cdot)) \,   $ in the notation of \eqref{eq: P} is a \textsc{Walsh} diffusion  associated with the triple $\, ({\bm 0},  \widetilde{\bm s}, {\bm \nu}) \,$ and defined up to the same explosion time $\, S \,$, with   state-space  $\, J   $ and the same underlying probability space and Brownian motion as $\, X (\cdot) \,;$ and vice versa.    
\end{prop}

\begin{proof}
We prove only the first claim, as the converse part can be established in the same way. Assume that  $\, X (\cdot) \,$   is a \textsc{Walsh} diffusion  with state-space  $\, I \,$ associated with the triple $\, ({\bm b}, {\bm s}, {\bm \nu}) \,$ and up to an explosion time $\, S\,$, and let $  Y(\cdot) = \mathcal{P} (X(\cdot))    $. It follows that $  Y(\cdot) $ is $\, \overline{J} \,$-valued and continuous in the tree-topology. 
We recall Definition \ref{def: G}. 
By Definition \ref{def: WalshDiff}, Theorem \ref{Thm: Gen}(iii) 
 and Proposition \ref{prop: pq}, direct calculation gives 
\begin{equation}
\label{eq: ||Y||}
\Vert Y(\cdot \wedge S_{n}) \Vert = p \big( X(\cdot \wedge S_{n}) \big) = \Vert Y(0) \Vert + \int_{0}^{\,\cdot \wedge S_{n}} \, {\bf 1}_{ \{ X(t) \neq {\bm 0} \} } \, \widetilde{\bm s} \big( Y(t)\big) {\rm d} W(t) \, + \, L^{ \Vert X(\cdot \wedge S_{n} ) \Vert}( \cdot ).
\end{equation}
From $ \, Y(\cdot) = \mathcal{P} (X(\cdot)) \,   $ it is clear that the equality $\, \{ t: \, \Vert Y(t) \Vert = 0 \} = \{ t: \, \Vert X(t) \Vert = 0 \} \,$ holds pathwise, so  by (\ref{eq: ||Y||}) and (\ref{eq: L0}) we have 
\begin{equation}
\label{eq: LY}
L^{ \Vert Y(\cdot \wedge S_{n} ) \Vert}( \cdot ) = \int^{\cdot}_{0} {\bf 1}_{\{\Vert Y(t \wedge S_{n} ) \Vert \, =\,  0\}}\, {\mathrm d}  \Vert Y(t \wedge S_{n} )\Vert \, = \, L^{ \Vert X(\cdot \wedge S_{n} ) \Vert}( \cdot )  
\end{equation}
\noindent
and   (\ref{eq: ||Y||}) turns into 
\[
\Vert Y(\cdot \wedge S_{n}) \Vert  = \Vert Y(0) \Vert + \int_{0}^{\,\cdot \wedge S_{n}} \, {\bf 1}_{ \{ Y(t) \neq {\bm 0} \} } \, \widetilde{\bm s} \big( Y(t)\big) {\rm d} W(t) \, + \, L^{ \Vert Y(\cdot \wedge S_{n} ) \Vert}( \cdot ) \, .
\]
Therefore, it suffices to verify that (v) of Definition \ref{def: WalshDiff} holds for $\, Y(\cdot)\,$. 

It  is apparent that $\, \int^{\cdot}_{0}{\bf 1}_{\{ Y(t) \, =\,  {\bm 0} \}} \, {\mathrm d} t \, =\,\int^{\cdot}_{0}{\bf 1}_{\{ X(t) \, =\,  {\bm 0} \}} \, {\mathrm d} t \equiv \,0 \, $ holds. On the other hand, since $$\, R_{A}^{Y}(\cdot) = \lVert Y (\cdot) \rVert \cdot {\bf 1}_A \big(\text{arg}\big(Y(\cdot)\big) \big) = p \big( X(t) \big) \cdot {\bf 1}_A \big(\text{arg}\big(X(\cdot)\big) \big) \, ,$$ we obtain the following, in the same way as in the derivation of (\ref{eq: ||Y||}): 
\[
R_{A}^{Y}(\cdot \wedge S_{n}) = R_{A}^{Y}(0) + \int_{0}^{\,\cdot \wedge S_{n}} \, {\bf 1}_{ \{ X(t) \neq {\bm 0} \} } \cdot {\bf 1}_{A}\big(\text{arg}\big(X(\cdot)\big) \big)\, \widetilde{\bm s} \big( Y(t)\big) {\rm d} W(t) \, + \,{\bm \nu} (A)\, L^{ \Vert X(\cdot \wedge S_{n} ) \Vert}( \cdot ) \, .
\]
Moreover, we observe $\, \{ t: \, R_{A}^{Y}(t) = 0 \} = \{ t:  X(t) = {\bm 0} \,\,\, \text{or} \,\,\, {\bf 1}_{A}\big(\text{arg}\big(X(t)\big) \big) = 0 \}\,$, thus
\[
L^{R_{A}^{Y}(\cdot \wedge S_{n})}(\cdot) = \int^{\cdot}_{0} {\bf 1}_{\{R_{A}^{Y}(t \wedge S_{n}) \, =\,  0\}}\, {\mathrm d}  R_{A}^{Y}(t \wedge S_{n}) \, = \, {\bm \nu} (A)\, L^{ \Vert X(\cdot \wedge S_{n} ) \Vert}( \cdot ) = {\bm \nu} (A)\, L^{ \Vert Y(\cdot \wedge S_{n} ) \Vert}( \cdot ) \, ;
\]
we have used (\ref{eq: LY}) for the last equality. Now (v) of Definition \ref{def: WalshDiff} is seen to hold for $\, Y(\cdot)\,$.
\end{proof}

We obtain the following result regarding   existence and uniqueness of a general \textsc{Walsh} diffusion.
\begin{thm}\label{Thm: Walshgen}
Assume  Condition \ref{bsigma} holds for   $\, {\bm b}: \, \check{I} \rightarrow \R \,$ and $\, {\bm  s}: \, \check{I} \rightarrow \R \,$. Then, for every   initial distribution $\, \bm\mu\,$ on the \textsc{Borel}  subsets of $\, I\,$, there exists a unique-in-distribution \textsc{Walsh} diffusion $\, X (\cdot)\,$ with state-space $\, I \,,$ associated with the triple $\, ({\bm b}, {\bm s}, {\bm \nu}) \,$ and defined up  to an explosion time $\, S\,$. 
\end{thm}

\begin{proof}
In light of Proposition \ref{prop: scaling}, it suffices to show   existence and uniqueness for the \textsc{Walsh} diffusion $\, Y (\cdot)\,$ in $\, J \,$ associated with the triple $\, ({\bm 0}, \widetilde{\bm s}, {\bm \nu}) \,$  up to an explosion time $\, S\,$, given any initial distribution $\, \bm\mu\,$. \newpage 

We shall reduce this to Theorem \ref{Thm: Walshnodrift}, which considers the full state space $\, \R^{2}\,$, not $\, J \,$. In addition to (\ref{eq: sigmatilde}), let us define $\,\widetilde{\bm s}(r, \theta) : = 0\,$ for $\, (r, \theta) \in J^{c} := \{ (r, \theta): \, r \geq p_{\theta} \big(\ell(\theta) \big) , \,\, 0 \le \theta < 2 \pi\}\,$. It is now straightforward, using Condition \ref{bsigma}, to check that $\,\widetilde{\bm s} \,$ satisfies Condition \ref{sigma} in Section \ref{sec: timechange}, and that $\, \mathcal{I}(\widetilde{\bm s}) = \mathcal{Z}(\widetilde{\bm s}) = J^{c}\,$. By Theorem \ref{Thm: Walshnodrift}, there exists a  unique-in-distribution, non-explosive  \textsc{Walsh} diffusion $\, Y (\cdot) \,$ with values in $\, \R^{2}\,$ associated with the triple $\, ({\bm 0}, \widetilde{\bm s}, {\bm \nu}) \,$, given any initial distribution in $\, J   $. Moreover, by Remark \ref{rmk: Walshnodrift}, $\, Y (\cdot)\,$ becomes motionless once it hits $\, \mathcal{I}(\widetilde{\bm s}) =  J^{c}\,$, i.e., once it exits from the set $\, J\,$. Thus it is clear by definition that $\, Y(\cdot)\,$ is also a \textsc{Walsh} diffusion in $\, J ,$  with explosion time $\, S: = \inf\{ t: \, Y(t) \notin J \} \,$. 

\smallskip
On the other hand, assume that $\, Y(\cdot) \,$ is a \textsc{Walsh} diffusion with values in $\, J \,$ associated with the triple $\, ({\bm 0}, \widetilde{\bm s}, {\bm \nu}) \,$, up to an explosion time $\, S = \inf\{ t: \, Y(t) \notin J\} \,;$  note that we stipulate $\, Y (t) =Y(S) \,$ for $\, S \leq t < \infty\,$. Thus by setting $\,\widetilde{\bm s} \equiv 0 \,$ on $\, J^{c} \,$ as before, we see immediately that $\, Y (\cdot)\,$ is also a \textsc{Walsh} diffusion with values in $\, \R^{2}\,$ associated with the triple $\, ({\bm 0}, \widetilde{\bm s}, {\bm \nu}) \,$. By Theorem \ref{Thm: Walshnodrift}, its probability law is uniquely determined, for any given initial distribution.
\end{proof}

\section{Explosion Test for Walsh Diffusions with Angular Dependence}\label{sec4}

Throughout this section, we have for every $\, x \in I: = \{ (r, \theta): \, 0 \leq r < \ell(\theta) , \,\, 0 \le \theta < 2 \pi\}\,$  a \textsc{Walsh} diffusion $\, (X, W), \, (\Omega, \mathcal{F}, \mathbb{P}^{x}), \, \mathbb{F} = \{ \mathcal{F}(t) \}_{0 \leq t < \infty} \,$ with values in $\, I \,$, associated with the triple $\, ({\bm b}, {\bm  s}, {\bm \nu}) \,$ and up to an explosion time $\, S\,$,  with $\, X(0) = x \,$, $\, \mathbb{P}^{x}-$a.e. Here $\, \ell : [0, 2 \pi) \rightarrow (0, \infty] \,$ is measurable and bounded away from zero, and the functions $\, {\bm b}: \, \check{I} \rightarrow \R \,$, $\, {\bm s}: \, \check{I} \rightarrow \R \,$ are assumed to satisfy Condition \ref{bsigma}. 

For different initial conditions $\, x\,$, these \textsc{Walsh} diffusions (including the underlying probability space) are different; but we shall use $\, X (\cdot)\,$ to denote every one of them. We shall let the measures $\, \mathbb{P}^{x}\,$ distinguish them, since all the conclusions we will draw are about their probability distributions.

\smallskip
We   develop in this section   analogues of all the results in Section 5.5.C of \cite{KS}. The two main results are Theorem \ref{Thm: XS}, on the asymptotic behavior of $\, X (\cdot) ;$  and Theorem \ref{Thm: Sfinite} on the test for explosions in finite time. 

\subsection{Preliminaries; Explosion in  Finite Expected  Time}

We first note that if $\, X (\cdot)\,$ starts at the origin and $\, A\in \mathcal{B} ( [0, 2\pi)) \,$ satisfies $\, {\bm\nu}(A) =0\,$, then $\, X  (\cdot)  \,$ never visits any region in   the state-space whose rays correspond to angles in $\, A\,$, with positive probability.

\begin{prop}\label{prop: zeroangle}
For every $\, A\in \mathcal{B} ( [0, 2\pi)) \,$ with $\, {\bm\nu}(A) =0\,$, we have $\, R_{A}^{X}(\cdot) \equiv 0\,,$ $\mathbb{P}^{\bm 0}-$a.e. in 
\eqref{eq: RA}. 

In other words, the set $\,\{ t \ge 0 : \, X(t, \omega) \neq {\bm 0} \,\,\, \text{and} \,\,\, \text{arg}\big(X(t, \omega)\big) \in A \}\,$ is empty, for $ \mathbb{P}^{\bm 0}-$a.e. $\, \omega \in \Omega\,$.  
\end{prop}
\begin{proof}
From the proofs of Theorem \ref{Thm: Walshgen}, Proposition \ref{prop: scaling} and Theorem \ref{Thm: Walshnodrift}, we see that $ \, Y(\cdot): = \mathcal{P} (X(\cdot)) \,$ is a driftless \textsc{Walsh} diffusion, and that it is also a time-changed \textsc{Walsh} Brownian motion with angular measure $\, {\bm \nu} \,$. But a \textsc{Walsh} Brownian motion with angular measure $\, {\bm \nu} \,$ can be constructed as in the proof of Theorem 2.1 in \cite{IKPY}, by assigning every excursion of a reflected Brownian motion an angle via a sequence of I.I.D random variables distributed as $\, {\bm \nu} \,$. Therefore, if $Y(\cdot)$  starts at the origin, it almost surely never visits any rays with angles in a set $\, A\in \mathcal{B} ( [0, 2\pi)) \,$ with $\, {\bm\nu}(A) =0\,$, because the aforementioned I.I.D. random variables will   not be valued in $\, A\,$ with any positive probability. This property is inherited by the time-changed \textsc{Walsh} Brownian motion $\, Y (\cdot) \,,$ and then by the process $\, X(\cdot) = \mathcal{Q} (Y(\cdot))\,$. 
\end{proof}
Next, we note that $\,X (\cdot)\,$ has the strong \textsc{Markov} property. By Theorem \ref{Thm: Walshgen}, the probability 
\begin{equation}
\label{eq: transition}
\,\mathfrak{h}( x ; \,\Gamma) \, : = \,\mathbb{P}^{x} \big( X(\cdot)\in\Gamma\big) \,
\end{equation}
is uniquely determined,  for all $\, x \in I \,$ and $\,  \Gamma \in \mathcal{B } \big( C(\overline{I}) \big) \,$. Here $\, C(\overline{I})\,$ is the collection of all $\, \overline{I}$-valued  functions on $\, [0, \infty)\,$ which are continuous in the tree-topology and get absorbed upon hitting the boundary $  \partial I: = \{ (r, \theta):   r = \ell(\theta) , \,\, 0 \le \theta < 2 \pi \}  $; the   \textsc{Borel} subsets of this space are generated by its finite-dimensional cylinder sets. Since we constructed $\, X (\cdot)\,$ in the last section through scaling and time-change, it is clear that the mapping $\, x \mapsto \mathfrak{h}( x ;\, \Gamma)\,$ is measurable on $\, I  ,$  for every $\,\, \Gamma \in \mathcal{B } \big( C(\overline{I}) \big) \,$. 

The following result can be proved by connecting to local martingale problems through a combination of adaptations of Propositions 6.1 and 9.1 in \cite{IKPY}, that allow an explosion time; we will omit its proof.

\begin{prop}
\label{prop: SMarkov}
For every $\, x \in I\,$, the process $\, X (\cdot)\,$ is time-homogeneous and strongly Markovian, in the sense that for every stopping time $\, T \,$ of $\, \mathbb{F}\,$ and every set $ \, \Gamma \in \mathcal{B } \big( C(\overline{I}) \big) \,$  we have 
\[
\mathbb{P}^{x} \big(X (T+ \cdot) \in \Gamma\, \big\vert \,\mathcal{F}(T) \big) \, = \, \mathfrak{h}( X(T) ; \,\Gamma) \, , \qquad \, \mathbb{P}^{x}-\text{a.e. on}\,\,\, \{ T < S \} \, .
\]
\end{prop}

\smallskip
Now we recall the radial scale function $\, p : I \mapsto [0, \infty) \,$ in (\ref{eq: p}), and  observe from (\ref{eq: ||Y||}) that $\, p \,$ turns $\, X(\cdot)\,$ into a reflected local martingale, which is the radial part of the driftless \textsc{Walsh} diffusion $\, Y= \mathcal{P}(X)\,$. By analogy with  one-dimensional diffusions, we introduce the {\it speed measure} 
\begin{equation}
\label{eq: speed}
{\bm m}_{\theta} ({\rm d} r ): = \, \frac{2 \,{\rm d} r}{p_{\theta}^{\prime}(r) {\bm s}^{2}(r, \theta)} \, , \qquad (r, \theta) \in \check{I}
\end{equation}
as well as the   {\it \textsc{Feller}  function} 
\begin{equation}
\label{eq: v}
v(r, \theta) = v_{\theta}(r) : = \int_{0}^{r} p_{\theta}^{\prime}(y)  {\bm m}_{\theta} ([0, y] )  {\rm d} y \, = \, \int_{0}^{r} \big(  p_{\theta}(r) - p_{\theta}(y) \big) {\bm m}_{\theta} ({\rm d} y ) \, , \quad  (r, \theta) \in I\, .
\end{equation}

We have the following result regarding the functions $\, p \,$ and $\, v\,$. 

\begin{prop}\label{prop: v/p}
{\bf (i)} ~~The function $\, v: \, I \rightarrow [0, \infty) \,$ of \eqref{eq: v} is in the class $\,\mathfrak D_{I}\,$ (cf. Definition \ref{def: D2}) with $\, v_{\theta}^{\prime}(0+) \equiv 0 \,;$  and for every $\,\theta \in [0, 2\pi)\,$, we have 
\begin{equation}
\label{eq: vdifferential}
{\bm b} (r, \theta) v_{\theta}^{\prime}(r) + \frac{1}{2} {\bm s}^{2} (r, \theta) v_{\theta}^{\prime\prime}(r) \, = \, 1\, , \,\,\qquad \text{for a.e.} \, \, \,\, r \in (0, \ell(\theta) ) \, . \quad
\end{equation}
{\bf (ii)} ~ For every $\, \theta \in [0, 2\pi)\,$, the function $ \, r \longmapsto \big( v_\theta (r) / p_{\theta}(r) \big) \,$  
is strictly increasing on $\, (0, \ell (\theta) ) \,$ with $\,\big(  v_{\theta} / p_{\theta} \big) (0+) = 0 \,$. Thus $\,   \big( v_{\theta} / p_{\theta}\big) (\ell (\theta)-)\,$ is well-defined (but may be $\,\infty$).

\noindent
{\bf (iii)} ~The implication $\,\, p_{\theta}(\ell(\theta)-) = \infty \,\,\Rightarrow \,\, v_{\theta}(\ell(\theta)-) = \infty\,$  \,holds for every $\, \theta \in [0, 2\pi)\,$.
\end{prop}

\begin{proof}
The claim {\bf (i)} can be checked in a very direct manner. Moreover, we have by (\ref{eq: v}) that
\[
 \frac{v_{\theta}(r)}{p_{\theta}(r)} = \int_{0}^{r} \Big(  1 - \frac{p_{\theta}(y)}{p_{\theta}(r)} \Big) {\bm m}_{\theta} ({\rm d} y )\, ,
\]
and {\bf (ii)} is then immediate from this, and from the fact that $\, p_{\theta}(r)\,$ is positive and strictly increasing on $\, (0, \ell (\theta) ) \,$. Finally, {\bf (iii)} follows clearly from {\bf (ii)}.
\end{proof}

Now we give  a sufficient condition for $\, X (\cdot)\,$ to explode in   finite expected time.

\begin{prop}\label{prop: finiteexpectexplo}
We have $\, \mathbb{E}^{x} [ S ] < \infty\,$ for every $\, x \in I ,\,$  if 
\begin{equation}
\label{eq: finiteexpectexplo}
\,{\bm \nu} \big(\big\{ \theta : \, p_{\theta}(\ell (\theta)-) < \infty \big\}\big) > 0\qquad\text{and}\qquad \sup_{\theta \in [0, 2\pi)} \, \Big( \frac{v_{\theta}}{p_{\theta}} \Big) (\ell (\theta)-) \, <\, \infty\, .
\end{equation}
In particular, we have $\,\, \mathbb{E}^{x} [ S ] < \infty\,$ for every $\, x \in I ,$ \, if $\,\,\,\sup_{\theta \in [0, 2\pi)} v_{\theta}(\ell (\theta)-) \, < \, \infty\,$.
\end{prop}
\begin{proof}
Assume that (\ref{eq: finiteexpectexplo}) holds. Then we can define
\begin{equation}
\label{eq: c1}
C_{1}: = \frac{\int_{0}^{2 \pi} \big( \frac{v_{\theta}}{p_{\theta}}\big) (\ell (\theta)-) \, {\bm \nu} ({\rm d}\theta) }{\int_{0}^{2 \pi} \frac{1}{p_{\theta}(\ell (\theta)-)}\, {\bm \nu} ({\rm d}\theta)} \, , \qquad C_{2} (\theta): = \, - \frac{C_{1}}{p_{\theta}(\ell (\theta)-)} + \Big( \frac{v_{\theta}}{p_{\theta}} \Big) (\ell (\theta)-) \, ,
\end{equation}
\begin{equation}\label{eq: M}
\, M(r, \theta) \equiv  M_{\theta}(r) : = -v_{\theta}(r) + C_{2} (\theta)\, p_{\theta}(r)+C_{1}\, \qquad (r, \theta) \in I \, .
\end{equation}
Note that the expression for $\, C_{2}(\theta)\,$ in (\ref{eq: c1}) is meaningful even in the case $\, p_{\theta}(\ell (\theta)-)= \infty \,$. \newpage   

\smallskip
Now $\, M \,$ is a well-defined function on $\, I \,$, as $\, M(0, \theta) \equiv C_{1}\,$. Since $\, \theta \mapsto p_{\theta}(\ell(\theta)-)\,$ is bounded away from zero by Proposition \ref{prop: pq} (i), we see that $\, \theta \mapsto C_{2}(\theta)\,$ is bounded, and that $\, M \in \mathfrak{D}_{I}\,$, thanks to Proposition \ref{prop: v/p} (i). Moreover, by Propositions \ref{prop: v/p} and \ref{prop: pq}, it is easy to check 
\begin{equation}
\label{eq: MDeri}
\int_{0}^{2\pi} M_{\theta}^{\prime}(0+) \, {\bm \nu}({\rm d} \theta)  =  0 \qquad\,\, \,\text{and} \,\, \qquad\,\, {\bm b} (r, \theta) M_{\theta}^{\prime}(r) + \frac{1}{2} {\bm s}^{2} (r, \theta) M_{\theta}^{\prime\prime}(r) = -1\, . \quad
\end{equation}
Recalling Definition \ref{def: G}, we apply Theorem \ref{Thm: Gen} 
and obtain the $\mathbb{P}^{x}-$a.e. equality
\begin{equation}
\label{eq: M(X)}
M(X(T \wedge S_{n})) = M(x) -(T \wedge S_{n} ) + \int_{0}^{T \wedge S_{n}} {\bf 1}_{\{ X(t) \neq {\bm 0} \}}\, M^{\prime}(X(t))\, {\bm s} (X(t)) \,{\rm d}W(t) \, , \quad 0 \leq T < \infty  
\end{equation}
where $\, S_{n}\,$ is as in Definition \ref{def: WalshDiff} (iii). With 
\[
\tau_{n} : =   \inf \bigg\{ t: \int_{0}^{t \wedge S_{n}} {\bf 1}_{\{ X(u) \neq {\bm 0} \}}\, \big( M^{\prime}(X(u))\, {\bm s} (X(u)) \big)^{2} \,{\rm d}u \, \geq \, n \bigg\}   \wedge S_{n} \, , \quad
\]
taking expectations in (\ref{eq: M(X)}) yields
\begin{equation}
\label{eq: ESn}
\mathbb{E}^{x} \,\big[M(X(\tau_{n})) + \tau_{n} \big] \, = \, M(x) \, , \qquad \forall \, n \in \mathbb{N} \, .
\end{equation}
On the other hand, we have by Proposition \ref{prop: v/p}(ii) that
\[
M(r, \theta) \,=  \,C_{1} \,
\bigg( 1-\, \frac{p_{\theta}(r)}{p_{\theta}(\ell (\theta)-)} \bigg)\,  + \, p_{\theta}(r) \, \bigg[  \Big( \frac{v_{\theta}}{p_{\theta}} \Big) \big(\ell (\theta)-\big) - \Big(\frac{v_{\theta}}{p_{\theta}} \Big) (r)  \bigg]  \, \geq \, 0 \, , \quad\,\, \forall \, (r, \theta) \in I \, .
\]
Thus (\ref{eq: ESn}) implies $\,\, \mathbb{E}^{x} [ \tau_{n} ] \leq M(x) \, , \,\,\, \forall \, n \in \mathbb{N} \,$. Letting $\, n \rightarrow \infty\,$ we get $\, \, \mathbb{E}^{x} [ S ] \leq M(x) < \infty \,$. 

\smallskip
Finally, we note that $\,\,\sup_{\theta \in [0, 2\pi)} v_{\theta}(\ell (\theta)-) \, < \, \infty\,$ implies (\ref{eq: finiteexpectexplo}), thanks to Proposition \ref{prop: v/p} (iii) and Proposition \ref{prop: pq}(i). Proposition \ref{prop: finiteexpectexplo} is now proved. 
\end{proof}

\subsection{Asymptotic Behavior     
Near the Explosion Time}

Throughout this subsection and the next one, we use the notation $\, \Theta (t) : = \text{arg} (X(t)) \,$ whenever $\, X(t) \neq \bm 0\,$, and recall from (\ref{eq: RA})  the process  $$\, R_{A}^{X}(\cdot) : = \lVert X (\cdot) \rVert \cdot {\bf 1}_A \big(\Theta(\cdot) \big)\, , \quad \,\forall \, A \in \mathcal{B} ([0,2\pi))\, .$$ We recall also the functions $\,\{ \ell_{n}\}_{n=1}^{\infty}\,$ and the sets $\, \{ I_{n}\}_{n=1}^{\infty}\, $ at the beginning of Subsection 3.1.

The following main result of this subsection discusses the behavior of $\, X(t) \,$ as $\, t \,$ approaches $\, S\,$. 

\begin{thm}
\label{Thm: XS}
{\bf Starting At the Origin:} 
Let $\, x = \bm 0\,$ in the context specified at the beginning of this section. With $\, p \,$ defined as in (\ref{eq: p}), we distinguish two cases: 

\smallskip
\noindent
{\bf (i)} ~ ${\bm \nu} (\{ \theta : \, p_{\theta}(\ell (\theta)-) < \infty \}) > 0\, .~$  \\ Then the limit in the tree-topology $\, \lim_{t \uparrow S} X(t)\,$ exists $\,\mathbb{P}^{\bm 0}$-a.e. in the extended rays, and $\,\,  X(S): = \lim_{t \uparrow S} X(t)$ takes values in the set $\,\partial I : = \{ (r, \theta): \, r = \ell(\theta), \,\,\,\theta \in [0,2\pi) \} \,$.  Moreover, we have in this case 
\begin{equation}
\label{eq: XSdistri}
\mathbb{P}^{\bm 0} \big( \Theta(S) \in A \big) \,\, = \,\, \frac{\int_{A} \, \frac{1}{p_{\theta}(\ell (\theta)-)} \, {\bm\nu}({\rm d} \theta) }{\int_{0}^{2\pi} \frac{1}{p_{\theta}(\ell (\theta)-)} \, {\bm\nu}({\rm d} \theta)} \,\,\, , \qquad\,\,\, \forall \, A \in \mathcal{B} ([0,2\pi))\, .
\end{equation}
\noindent
{\bf (ii)} ~~ ${\bm \nu} (\{ \theta : \, p_{\theta}(\ell (\theta)-) < \infty \}) = 0 \,.~$ \\ Then $\,\mathbb{P}^{\bm 0}-$a.e., we have that the limit $\, \lim_{t \uparrow S} X(t)\,$ does not exist, and that 
\begin{equation}
\label{eq: recur}
{\bm \nu} \bigg( \, \overline{ \Big\{ \theta: \, \sup_{ 0 \leq t < S  } R^{X}_{\{\theta\}}(t) \, \geq \, \ell_{n}(\theta) \Big\} } \,\bigg) \, = \, 1 \, , \qquad \forall \, n \in \mathbb{N}  
\end{equation}
holds, where the closure is taken in $\, [0, 2\pi) .$  In particular,   $\, \mathbb{P}^{\bm 0} \big( S = \infty \big) =1\,.$ \\ Moreover, whenever $\, {\bm \nu} (\{ \theta\}) > 0  \,$ holds,  we have\newpage  
\begin{equation}
\label{eq: singleton}
\sup_{\, 0 \leq t < S  } R^{X}_{\{\theta\}}(t) \, = \,\ell (\theta) \, , \quad\mathbb{P}^{\bm 0}-\text{a.e.}
\end{equation}
\end{thm}

\begin{remark}\label{rmk: XS}
We stipulate $\, \frac{1}{\infty} = 0 \,$ in (\ref{eq: XSdistri}). Since $\, \theta \mapsto p_{\theta}(\ell (\theta)-)\,$ is bounded away from zero, we see that (\ref{eq: XSdistri}) makes good sense, provided   $\, {\bm \nu} (\{ \theta :   p_{\theta}(\ell (\theta)-) < \infty \}) > 0\,$ holds.  We make no claim  in {\bf (i)} regarding  the finiteness of $\, S\,$, and the result  holds there regardless of whether $\, S\,$ is finite or not. A full discussion regarding the finiteness of $\, S\,$ appears in Subsection 4.3.
\end{remark}
\begin{remark}
\label{rmk: XS2}
{\it If we replace ``$\ell_{n}(\theta)$" by ``$\ell(\theta)$" in the property (\ref{eq: recur}), then this new property no longer holds in general.}  Indeed, let $\,{\bm\nu}\,$ be the uniform  distribution on $\, [0, 2\pi)\,$ in {\bf (ii)}; then (\ref{eq: singleton}) holds for no $\, \theta \in [0, 2\pi)\,$, $\,\mathbb{P}^{\bm 0}-$a.e. This is because $\, X (\cdot) \,$ will be on different rays for any two of its excursions away from the origin.  
\end{remark}

\begin{proof}
We first note that the explosion time $\, S\,$ does not depend on the choice of the approximating sequence of functions $\,\{ \ell_{n}\}_{n=1}^{\infty}\,$, because $\, S  = \inf \{ t: \, X(t) \notin I \}\,$ always holds by (\ref{eq: explos}). Thus in the proof of {\bf (i)}, we will assume that 
 \begin{equation}
 \label{eq: pvlnfinite}
 \, p_{\theta}(\ell_{n} (\theta)) \leq n \,\qquad \text{and} \qquad\, v_{\theta}(\ell_{n} (\theta)) \leq n \, , \qquad \forall  \, \theta \in [0 , 2\pi)\, , \quad \forall \, n \in \mathbb{N}\, ;
 \end{equation}
for otherwise, we can define  
$$\,
\widetilde{\ell_{n}}(\theta) \,: =\, \sup \big\{ r: \, 0\leq r \leq \ell_{n}(\theta), \,\, p_{\theta}(r) \leq n ,\,\, v_{\theta}(r) \leq n \big\}\,
$$ 
and let the sequence $\,\{ \widetilde{\ell_{n}}\}_{n=1}^{\infty}\,$ play the role of $\,\{ \ell_{n}\}_{n=1}^{\infty}\,$. However, we will not assume (\ref{eq: pvlnfinite}) when proving {\bf (ii)}, because $\,\ell_{n}\,$ appears explicitly in the conclusion of {\bf (ii)}.

\medskip
\noindent
{\it Proof of {\bf (i)}.} {\it Step 1.} We shall prove {\bf (i)} in this step, albeit under the assumptions
\begin{equation}
\label{eq: vpfinite}
\sup_{\theta \in [0, 2\pi)}   p_{\theta}(\ell (\theta)-) \, < \, \infty\,  \qquad \text{and} \qquad \sup_{\theta \in [0, 2\pi)}  v_{\theta}(\ell (\theta)-) \, < \, \infty \, .\qquad
\end{equation}
With (\ref{eq: vpfinite}),  we have $\, \mathbb{E}^{\bm 0} [ S ] < \infty\,$ by Proposition \ref{prop: finiteexpectexplo}, thus $\, \mathbb{P}^{\bm 0} (S<\infty) = 1\,$. Thus, from (\ref{eq: explos}) we know that $\, X(S) = \lim_{t \uparrow S} X(t)\,$ exists under the tree-topology in $\, \{ (r, \theta): \, r = \ell(\theta) , \,\, 0 \le \theta < 2 \pi\} \,$, $\,\mathbb{P}^{\bm 0}$-a.e. It develops that $\, \Theta (S) \,$ is a well-defined random variable with values in $\, [0, 2\pi)\,$; we denote its distribution by   $\, \widetilde{\bm \nu} \,$,     a probability measure on  $\, ( [0, 2 \pi), \mathcal{B}([0, 2\pi)))\,$.

\smallskip
Let us define the {\it scale function associated with a set $\, A \in \mathcal{B} ([0,2\pi))\,$} by
\begin{equation}
\label{eq: pA}
p^{A}_{\theta}(r) = p^{A}(r, \theta):= \, p(r,\theta)\cdot \big( {\bm \nu}(A)\cdot  {\bf 1}_{A^{c}} (\theta) \, -\, {\bm \nu}(A^{c})\cdot {\bf 1}_{A} (\theta)\big) \, , \qquad (r, \theta) \in I \, .
\end{equation}
Clearly, we have $\, p^{A} \in \mathfrak{D}_{I}\,$ and $\,\int_{0}^{2\pi} (p^{A}_{\theta})^{\prime}(0+) \,{\bm\nu}({\rm d}\theta) \, = \, 0\,$. Now with the help of Proposition \ref{prop: pq} and Theorem \ref{Thm: Gen}, we can easily check that $\, p^{A}(X(\cdot \wedge S_{n})) \,$ is a local martingale -- and  actually a martingale, because (\ref{eq: vpfinite}) gives the boundedness of $\, p^{A}\,$ on $\, I \,$.  Then we may let $\, n \rightarrow \infty\,$ to obtain that $\, p^{A}(X(\cdot \wedge S)) \,$ is a bounded martingale. This gives
\[
0\, = \, p^{A} ({\bm 0})\, = \,\mathbb{E}^{\bm 0} \big[\,p^{A}\big( X(S) \big) \, \big]  \, = \,\mathbb{E}^{\bm 0} \big[\,p^{A} \big( \ell (\Theta (S)) , \Theta (S) \big) \, \big] \, =\, \int_{0}^{2\pi} p^{A} \big( \ell(\theta ) , \theta  \big)\, \widetilde{\bm \nu}({\rm d}\theta)
\]
\begin{equation}
\label{eq: pAX}
 \, = \, {\bm \nu}(A)\cdot \int_{A^{c}} p_{\theta} ( \ell (\theta )-  ) \, \widetilde{\bm \nu}({\rm d}\theta) - {\bm \nu}(A^{c})\cdot \int_{A} p_{\theta} ( \ell (\theta )-  ) \, \widetilde{\bm \nu}({\rm d}\theta) \, ; \qquad \qquad
\end{equation}
 here we have extended the function $\, p^{A} \,$ to $\, \overline{I}\,$ continuously, so that $\, p^{A}\big( X(S) \big)\,$ is well-defined.

From (\ref{eq: pAX}), we observe  that $\, \widetilde{\bm \nu}(A) = 0 \,$ holds whenever $\, {\bm \nu}(A)=0\,$. Thus the measure $\,\widetilde{\bm \nu}\,$ is absolutely continuous with respect to $\,\bm\nu\,,$ and we may assume that $\,\,\widetilde{\bm \nu}({\rm d}\theta) \, =\, \psi (\theta)\, {\bm \nu}({\rm d}\theta) \,\,$ for some function $\, \psi : [0, 2\pi) \rightarrow [0, \infty) \,$. Now for (\ref{eq: XSdistri}) to hold, we only need to show that 
\begin{equation}
\label{eq: psi}
\psi (\theta) \, = \, \frac{ \frac{1}{p_{\theta}(\ell (\theta)-)}  }{\int_{0}^{2\pi} \frac{1}{p_{\theta}(\ell (\theta)-)} \, {\bm\nu}({\rm d} \theta)} \,\,\, , \quad\qquad\,\,\, {\bm \nu}\text{-a.e.} \,\,\, \theta \in [0,2\pi) \, . \qquad
\end{equation}\newpage 
To this effect, we consider the sets 
\[
A_{1} : = \bigg\{ \theta: \, \psi (\theta) > \frac{ \frac{1}{p_{\theta}(\ell (\theta)-)}  }{\int_{0}^{2\pi} \frac{1}{p_{\theta}(\ell (\theta)-)} \, {\bm\nu}({\rm d} \theta)} \bigg\}  \,\qquad \text{and} \qquad \,  A_{2} : = \bigg\{ \theta: \, \psi (\theta) < \frac{ \frac{1}{p_{\theta}(\ell (\theta)-)}  }{\int_{0}^{2\pi} \frac{1}{p_{\theta}(\ell (\theta)-)} \, {\bm\nu}({\rm d} \theta)} \bigg\} \, .
\]
Letting $\, A= A_{1} \,$ in (\ref{eq: pAX}), it is easy to deduce that either $\, {\bm\nu}(A_{1}) =0 \,$ or $\, {\bm\nu}(A_{1}) =1\,$  must hold. But the latter cannot happen, for otherwise we would have $\, \widetilde{\bm \nu} ([0, 2\pi)) = \int_{0}^{2\pi} \psi (\theta)\, {\bm \nu}({\rm d}\theta) > 1 \,$. Thus $\, {\bm\nu}(A_{1}) =0 \,$ holds, and we deduce  $\, {\bm\nu}(A_{2}) =0 \,$ similarly. This way we get (\ref{eq: psi}), and Step 1 is now complete.    

\medskip
\noindent
{\it Step 2.} This step will complete the proof of {\bf (i)}. We first show the existence of $\, \lim_{t \uparrow S} X(t) \,$, $\,\mathbb{P}^{\bm 0}$-a.s. 

\smallskip
\noindent
{\it Case A:  $\, {\bm \nu} \,$ concentrates on   one angle $\, \theta_{0}\,$.} Then $\, p_{\theta_{0}}(\ell (\theta_{0})-) < \infty\,$, and  $\, X (\cdot) \,$  stays a.s.\,on the ray with angle $\, \theta_{0}\,$, by Proposition \ref{prop: zeroangle}. Thus the process $\, p(X(\cdot)) \,$ is bounded. But $\, p(X(\cdot   \wedge S_{n}))\,$ is a local submartingale (as a reflected local martingale), thus  a true submartingale, and so is $\, p(X(\cdot \, \wedge S)) \,$. We deduce that $\, \lim_{t \uparrow S} p(X(t)) \,$ exists $\,\mathbb{P}^{\bm 0}-$a.e. Since $\, X (\cdot) \,$ stays on the same ray, the existence of $\, \lim_{t \uparrow S} X(t) \,$ follows.

\smallskip
\noindent
{\it Case B:  $\, {\bm \nu} \,$ does not concentrate on one angle.} Since $\, {\bm \nu} (\{ \theta : \, p_{\theta}(\ell (\theta)-) < \infty \}) > 0\,$, we can choose an $\, M > 0 \,$ and an $\, A_{M} \subseteq [0, 2\pi)\,$, such that $\, p_{\theta}(\ell (\theta)-) \leq M\,$ for all $\,\theta \in A_{M} \,$, and that $\, 0< {\bm \nu}(A_{M}) < 1 \,$. Then the function $\, p^{A_{M}} \,$ is bounded from below.   Step 1 shows that $\, p^{A_{M}}(X(\cdot \wedge S_{n})) \,$ is a local martingale for every $\, n \in \mathbb{N}\,$, thus a supermartingale, and we may let $\, n\rightarrow \infty\,$  to obtain by \textsc{Fatou}'s lemma that $\, p^{A_{M}}(X(\cdot \wedge S)) \,$ is a bounded from below supermartingale. Therefore,  $\, \lim_{t \uparrow S} p^{A_{M}} (X(t)) \,$ exists $\,\mathbb{P}^{\bm 0}-$a.e. Now we set 
 $$\,
  \mathcal{P}^{A_{M}} (r, \theta):= \, \big( \big \vert p^{A_{M}} (r,\theta) \big \vert , \theta \big) \, \qquad \text{for} \qquad    (r,\theta)\in I\, , 
 $$ 
 and note $\, \mathcal{P}^{A_{M}} (I) = J^{A_{M}} \,$ where, thanks to $\, 0< {\bm \nu}(A_{M}) < 1 \,$,  the set 
 $$\, J^{A_{M}}\, : = \,\big\{ (r,\theta): 0\leq r< \vert p^{A_{M}}_{\theta}(\ell(\theta)-)\vert , \,\, 0 \le \theta < 2 \pi\big\}\,$$ is open in the tree-topology.  By the continuity of $\,X (\cdot)\,$ in the tree-topology, the existence of the limit $\, \lim_{t \uparrow S} p^{A_{M}} (X(t)) \,$ implies the existence of $\, \lim_{t \uparrow S} \mathcal{P}^{A_{M}} (X(t)) \,$ in $\, \overline{J^{A_{M}}}\,$, under the tree-topology.

\smallskip
By analogy with   Section 3.3, and thanks once again to $\, 0< {\bm \nu}(A_{M}) < 1 \,$, we can define the inverse mapping $\,\mathcal{Q}^{A_{M}}: J^{A_{M}}\rightarrow I \,$    of  $\, \mathcal{P}^{A_{M}}\,$, and both $\,\mathcal{Q}^{A_{M}}\,$ and $\, \mathcal{P}^{A_{M}}\,$ are continuous in the tree-topology. Moreover, we can extend $\,\mathcal{Q}^{A_{M}}\,$ to $\,  \overline{J^{A_{M}}}\,$ and $\, \mathcal{P}^{A_{M}}\,$ to $\, \overline{I}\,$ continuously. We see then, that  the existence of $\, \lim_{t \uparrow S} \mathcal{P}^{A_{M}} (X(t)) \,$ in $\, \overline{J^{A_{M}}}\,$ implies the existence of the limit $\, \lim_{t \uparrow S} X(t) \,$ in $\, \overline{I}\,$. 

\smallskip
Next, we turn to the proof of $\,  X(S)  \in \{ (r, \theta): \, r = \ell(\theta) , \,\, 0 \le \theta < 2 \pi\} \,$, as well as (\ref{eq: XSdistri}). Let us define 
\begin{equation}
\label{eq: n,m,A1}
\ell_{n,m}^{A}(\theta):=\, \ell_{n}(\theta) \cdot {\bf 1}_{A}(\theta) + \ell_{m}(\theta)\cdot {\bf 1}_{A^{c}}(\theta)\, , \qquad I_{n,m}^{A}: = \{ (r, \theta): \, 0 \leq r < \ell_{n,m}^{A}(\theta) , \,\, 0 \le \theta < 2 \pi\} \, ,
\end{equation}
\begin{equation}
\label{eq: n,m,A2}
S_{n,m}^{A} : = \, \inf \{ t\geq 0: \, \Vert X(t) \Vert \geq \ell_{n,m}^{A} (\Theta(t)) \} \, =\, \inf \{ t \geq 0: \, X(t) \notin I_{n,m}^{A} \} \,  \qquad \qquad
\end{equation}
for   $\, A \in \mathcal{B} ([0,2\pi))\,$ and $\, m ,\, n \in \mathbb{N}\,$ with $\, m \geq n\,$. By (\ref{eq: pvlnfinite}), we have $\, \sup_{\theta \in [0, 2\pi)}   p_{\theta}(\ell_{n,m}^{A} (\theta)) \, \leq \, m \,$ and $\,\, \sup_{\theta \in [0, 2\pi)}   v_{\theta}(\ell_{n,m}^{A} (\theta)) \, \leq \, m \,$. Thus Step 1 shows $\, \mathbb{P}^{\bm 0} (S_{n,m}^{A} <\infty) = 1\,$, and that 
\begin{equation}
\label{eq: XSnmA}
\mathbb{P}^{\bm 0} \big( \Theta(S_{n,m}^{A} ) \in A \big) \,\, = \,\, \frac{\int_{A} \, \frac{1}{p_{\theta}(\ell_{n}(\theta))} \, {\bm\nu}({\rm d} \theta) }{\int_{A} \, \frac{1}{p_{\theta}(\ell_{n}(\theta))} \, {\bm\nu}({\rm d} \theta) + \int_{A^{c}} \, \frac{1}{p_{\theta}(\ell_{m}(\theta))} \, {\bm\nu}({\rm d} \theta)} \,\,\, , \qquad\,\,\, \forall \, A \in \mathcal{B} ([0,2\pi))\, .
\end{equation}
Note that the events $\, \{ \Theta(S_{n,m}^{A} ) \in A \}\,$ are increasing in $\, m \,$. Setting $\, K_{n}^{A} : = \{ (r, \theta): \,  r \geq \ell_{n} (\theta), \, \,\, \theta \in A \}\,$, we have then
\[
\mathbb{P}^{\bm 0} \big( X \,\, \text{hits} \,\,\, K_{n}^{A} \, \big) \, \,\geq\,\,  \mathbb{P}^{\bm 0} \big( \Theta(S_{n,m}^{A} ) \in A \,\,\, \text{for some} \,\,\,  m \geq n \big)
~~~~~~~~~~~~~~~~~~ 
\]
\begin{equation}
\label{eq: XKnA}
~~~~~~~~~~~~~~~~~~~~~~~~~~~~~~~~
 \, = \, \lim_{m \rightarrow \infty} \mathbb{P}^{\bm 0} \big( \Theta(S_{n,m}^{A} ) \in A \big) \, = \,  \frac{\int_{A} \, \frac{1}{p_{\theta}(\ell_{n}(\theta))} \, {\bm\nu}({\rm d} \theta) }{\int_{A} \, \frac{1}{p_{\theta}(\ell_{n}(\theta))} \, {\bm\nu}({\rm d} \theta) + \int_{A^{c}} \, \frac{1}{p_{\theta}(\ell(\theta)-)} \, {\bm\nu}({\rm d} \theta)} \, .
\end{equation}
Since $\,  X(S):= \lim_{t\uparrow S} X(t) \,$ exists $\,\mathbb{P}^{\bm 0}$-a.e., we may let $\, n \rightarrow \infty\,$ in (\ref{eq: XKnA}) and obtain\newpage
\begin{equation}
\label{eq: XKA}
\mathbb{P}^{\bm 0}\big( X(S) \in  \{ (r, \theta): \, r = \ell(\theta), \,\, \theta \in A \} \big) \, =\, \mathbb{P}^{\bm 0}\big( X \,\, \text{hits} \,\,\, K_{n}^{A} \,\,\, \text{for every} \,\,\, n \in \N \big)\, \geq\, \frac{\int_{A} \, \frac{1}{p_{\theta}(\ell (\theta)-)} \, {\bm\nu}({\rm d} \theta) }{\int_{0}^{2\pi} \frac{1}{p_{\theta}(\ell (\theta)-)} \, {\bm\nu}({\rm d} \theta)} \, .
\end{equation} 
In particular, $\, \mathbb{P}^{\bm 0}\big( X(S) \in  \{ (r, \theta): \, r = \ell(\theta) , \,\, 0 \le \theta < 2 \pi\} \big) = 1 \,$. Replacing $\, A\,$ by $\, A^{c}\,$ in (\ref{eq: XKA}) and adding this back to (\ref{eq: XKA}), we find that the inequality sign in (\ref{eq: XKA}) can be replaced by an equality sign. Thus (\ref{eq: XSdistri}) follows, and the proof of Theorem \ref{Thm: XS}{\bf (i)} is now complete. 

\medskip
\noindent
{\it Proof of {\bf (ii)}.} Here we cannot assume (\ref{eq: pvlnfinite}), but   can use the result of {\bf (i)}. Because $\, p_{\theta}( \ell_{n,m}^{A}(\theta)) < \infty \,$ for every $\,\theta \in [0, 2\pi) \,$, we recover (\ref{eq: XSnmA}) by an application of (\ref{eq: XSdistri}). Thus we have
\[
\mathbb{P}^{\bm 0} \bigg( \sup_{ 0 \leq t < S  } R^{X}_{\{\theta\}}(t) \, \geq \, \ell_{n} (\theta)\,\, \,\text{for some} \,\,\, \theta \in A \, \bigg) \,\geq\, \mathbb{P}^{\bm 0} \big( \Theta(S_{n,m}^{A} ) \in A \,\,\, \text{for some} \,\,\,  m \geq n \big)
\]
\begin{equation}
\label{eq: XKnA1}
\, = \, \lim_{m \rightarrow \infty} \mathbb{P}^{\bm 0} \big( \Theta(S_{n,m}^{A} ) \in A \big)\, =\, \frac{\int_{A} \, \frac{1}{p_{\theta}(\ell_{n}(\theta))} \, {\bm\nu}({\rm d} \theta) }{\int_{A} \, \frac{1}{p_{\theta}(\ell_{n}(\theta))} \, {\bm\nu}({\rm d} \theta) + \int_{A^{c}} \, \frac{1}{p_{\theta}(\ell(\theta)-)} \, {\bm\nu}({\rm d} \theta)} \, = 1\,  
\end{equation}
for every $\, A \in \mathcal{B} ([0,2\pi))\,$ with $\,{\bm \nu}(A) > 0\,$, because $\,{\bm \nu} (\{ \theta : \, p_{\theta}(\ell (\theta)-) < \infty \}) = 0\,$. 

Now we can find an event $\, \Omega^{\ast} \in \cal F\,$ with $\, \mathbb{P}^{\bm 0} (\Omega^{\ast}) =1 \,$, such that for every $\, \omega \in \Omega^{\ast}\,$ and every $\, A = [a , b) \,$ with $\, a,\, b \in \mathbb{Q}\cap [0, 2\pi) \,$ and $\, {\bm \nu}(A) > 0\,$, we have $\, \sup_{ 0 \leq t < S  } R^{X}_{\{\theta\}}(t, \omega) \, \geq \, \ell_{n} (\theta)\,$ for some $\, \theta \in A \,$ and all $\, n \in \N\,,$ so (\ref{eq: recur}) is obtained. Moreover, if $\, {\bm \nu} (\{ \theta\}) > 0  \,$, we can take $\, A = \{ \theta\} \,$ in (\ref{eq: XKnA1}) and see that the inequality $\, \sup_{ 0 \leq t < S  } R^{X}_{\{\theta\}}(t) \, \geq \, \ell_{n} (\theta)\,$ holds $\,\mathbb{P}^{\bm 0}-$a.e., for every $\, n \in\mathbb{N}\,$; thus \eqref{eq: singleton} follows.   

\smallskip
Finally, we show that the nonexistence of $\, \lim_{t \uparrow S} X(t) \,$, and the property $\, S = \infty\,,$  follow directly, thanks to (\ref{eq: explos}). To this effect, we set $$\, A^{p} : = \big\{ \theta : \, p_{\theta}(\ell (\theta)-) < \infty \big\}\, ,
\qquad \, I^{p}: = \big\{ (r, \theta): \, r> 0 , \,\, \theta \in A^{p} \big\} \,$$ and $\, \Gamma^{0}:= \{ \omega_{2} \in C(\overline{I}): \, \omega_{2} (t) = {\bm 0} \,\,\, \text{for some} \,\,\, t \in [0,\infty) \} \,$. Recalling (\ref{eq: transition}), and using the theory of one-dimensional diffusion (e.g. Propositions 5.5.22, 5.5.32 in  \cite{KS}), we deduce  
\begin{equation}
\label{eq: hx0}
\mathfrak{h} ( x ; \, \Gamma^{0} ) = 1 \, , \qquad \forall \, \,\,\,x \in  I \setminus I^{p} \, .
\end{equation}
 With $\, T_{n}: = S_{n} \wedge n \,$, we have $\, T_{n} < S \,$, $\,\mathbb{P}^{\bm 0}-$a.e. Since $\, {\bm\nu}(A^{p}) = 0 \,$, Proposition \ref{prop: zeroangle} shows that $\, X(T_{n}) \in I \setminus I^{p}\,$, $\,\mathbb{P}^{\bm 0}-$a.e. Now we  apply Proposition \ref{prop: SMarkov} and obtain   $\, \mathbb{P}^{\bm 0} \big(X (T_{n} + \cdot) \in \Gamma^{0}  \big) = 1 \,$, $\, \forall \, n \in \mathbb{N}\,$. It follows that, $\,\mathbb{P}^{\bm 0}-$a.e., if $\, \lim_{t \uparrow S} X(t) \,$ exists, it must be $\,\bm 0 \,$. Comparing this fact with (\ref{eq: recur}), we see that $\, \lim_{t \uparrow S} X(t) \,$ does not exist, $\,\mathbb{P}^{\bm 0}-$a.e.  
\end{proof}

Theorem \ref{Thm: XS} takes the origin $\, \bm 0 \,$ as the starting point of $ X (\cdot) $. For a starting point $  x \in \check{I}  $, by the strong \textsc{Markov} property, we can treat $ X (\cdot)  $ as a one-dimensional diffusion  before it hits the origin, and use Theorem \ref{Thm: XS} afterwards. The following result can be   derived in a very direct manner, so we omit its proof.

\begin{cor}
\label{cor: XS}
{\bf Starting Away from the Origin:} 
In the context specified at the beginning of this section, let $\, x = (r_{0} , \theta_{0}) \in \check{I} \,$. We distinguish two cases: 

\smallskip
\noindent
{\bf (i)} ~ ${\bm \nu} (\{ \theta : \, p_{\theta}(\ell (\theta)-) < \infty \}) > 0\,$. 
\\ 
Then $\,\, \lim_{t \uparrow S} X(t)\,$ exists $\,\mathbb{P}^{x}-$a.e. in $\, \{ (\ell(\theta), \theta):   \,  \theta \in [0, 2 \pi)\}\, $, and for every $\, A \in \mathcal{B} ([0,2\pi))\,$ we have 
\begin{equation}
\label{eq: XSxdistri}
\mathbb{P}^{x} \big( \Theta(S) \in A \big) \, = \, \frac{\int_{A} \, \frac{1}{p_{\theta}(\ell (\theta)-)} \, {\bm\nu}({\rm d} \theta) }{\int_{0}^{2\pi} \frac{1}{p_{\theta}(\ell (\theta)-)} \, {\bm\nu}({\rm d} \theta)} \cdot \bigg( 1- \frac{p_{\theta_{0}} (r_{0}) }{p_{\theta_{0}}(\ell (\theta_{0})-)} \bigg) \, + \, {\bf 1}_{A}(\theta_{0})\cdot \frac{p_{\theta_{0}} (r_{0}) }{p_{\theta_{0}}(\ell (\theta_{0})-)}  \, .
\end{equation}

\smallskip
\noindent
{\bf (ii)} ~~ ${\bm \nu} (\{ \theta : \, p_{\theta}(\ell (\theta)-) < \infty \}) = 0\,$. \\ Then   we have\newpage $$\, \mathbb{P}^{x} \big(\mathcal{L}^{x}\big) = \frac{p_{\theta_{0}} (r_{0}) }{p_{\theta_{0}}(\ell (\theta_{0})-)} \,, \qquad \text{for} \qquad \mathcal L^{x}:=  \Big\{ \lim_{t \uparrow S} X(t) = ( \ell (\theta_{0}), \theta_{0} ) \Big \} .$$ On the other hand, $\,\mathbb{P}^{x}-$a.e. on $\, \big(\mathcal{L}^{x} \big)^{c} \,$,  we have that $\, \lim_{t \uparrow S} X(t)\,$ does not exist, that $\, S = \infty\,$, and that 
\begin{equation}
\label{eq: xrecur}
{\bm \nu} \bigg( \, \overline{\,\Big\{ \theta: \, \sup_{ 0 \leq t < S  } R^{X}_{\{\theta\}}(t) \, \geq \, \ell_{n}(\theta) \Big\}\, } \,\bigg) \, = \, 1 \, , \qquad \forall \, \,\, n \in \mathbb{N}\, . \quad 
\end{equation}
Moreover, whenever $\, {\bm \nu} (\{ \theta\}) > 0  \,$, we have $\, \,\,\sup_{\, 0 \leq t < S  } R^{X}_{\{\theta\}}(t) \, = \,\ell (\theta) \,$, $\,\mathbb{P}^{x}-$a.e. on $\, \big(\mathcal{L}^{x} \big)^{c} \,$. 
\end{cor}

\subsection{Test for Explosions in Finite Time}

This subsection provides   criteria for the finiteness of the explosion time. These involve the  scale and \textsc{Feller}  functions of (\ref{eq: p}), (\ref{eq: v}), and of course the measure ${\bm \nu}$. The proof of Theorem \ref{Thm: Sfinite} is in the Appendix, Section 6; whereas the proof of Corollary \ref{cor: Sfinite} is omitted, for the same reason as that of Corollary \ref{cor: XS}.

\begin{thm}
\label{Thm: Sfinite}
{\bf Starting At the Origin:} 
Let $\, x = \bm 0\,$ in the context specified at the beginning of this section. With the functions $ \, p \,$ and $\, v \,$ defined by (\ref{eq: p}) and (\ref{eq: v}) respectively, we distinguish three cases:

\smallskip
\noindent
{\bf (i)} ~ ${\bm \nu} (\{ \theta : \, v_{\theta}(\ell (\theta)-) < \infty \}) = 0 \,.$  \\Then we have $\, \mathbb{P}^{\bm 0} \big( S < \infty \big) = 0\,$.

\smallskip
\noindent
{\bf (ii)} ~${\bm \nu} (\{ \theta : \, v_{\theta}(\ell (\theta)-) < \infty \}) > 0\,$ and $\, {\bm \nu} (\{ \theta : \, v_{\theta}(\ell (\theta)-) = \infty , \, p_{\theta}(\ell (\theta)-) < \infty\}) = 0 \,.$ \\ Then we have $\, \mathbb{P}^{\bm 0} \big( S < \infty \big) = 1\,$.
 
\smallskip
\noindent
{\bf (iii)} ~${\bm \nu} (\{ \theta : \, v_{\theta}(\ell (\theta)-) < \infty \}) > 0\,$ and $\, {\bm \nu} (\{ \theta : \, v_{\theta}(\ell (\theta)-) = \infty , \, p_{\theta}(\ell (\theta)-) < \infty\}) > 0 \,.$ \\ Then we have $\, 0 < \mathbb{P}^{\bm 0} \big( S < \infty \big) < 1\,$.
\end{thm}

\begin{cor}
\label{cor: Sfinite}
{\bf Starting Away from the Origin:} 
In the context specified at the beginning of this section, let $\, x = (r_{0} , \theta_{0}) \in \check{I} \,$.   We distinguish three cases: 

\smallskip
\noindent
{\bf (i)} ~ ${\bm \nu} (\{ \theta : \, v_{\theta}(\ell (\theta)-) < \infty \}) = 0\,$. \\ Then we have $\, \mathbb{P}^{x} \big( S < \infty \big) = 0\,$ if $\, v_{\theta_{0}}(\ell (\theta_{0})-) = \infty\,$, and $\, 0 < \mathbb{P}^{x} \big( S < \infty \big) < 1\,$ otherwise. 

\smallskip
\noindent
{\bf (ii)} ~${\bm \nu} (\{ \theta : \, v_{\theta}(\ell (\theta)-) < \infty \}) > 0\,$ and $\, {\bm \nu} (\{ \theta : \, v_{\theta}(\ell (\theta)-) = \infty , \, p_{\theta}(\ell (\theta)-) < \infty\}) = 0\,$. \\ Then we have $\, \mathbb{P}^{x} \big( S < \infty \big) = 1\,$ if either $\, v_{\theta_{0}}(\ell (\theta_{0})-) < \infty\,$ or $\, p_{\theta_{0}}(\ell (\theta_{0})-) = \infty\,$ hold, \\ whereas we have  $\, 0 < \mathbb{P}^{x} \big( S < \infty \big) < 1\,$ otherwise.

\smallskip
\noindent
{\bf (iii)} ~${\bm \nu} (\{ \theta : \, v_{\theta}(\ell (\theta)-) < \infty \}) > 0\,$ and $\, {\bm \nu} (\{ \theta : \, v_{\theta}(\ell (\theta)-) = \infty , \, p_{\theta}(\ell (\theta)-) < \infty\}) > 0\,$. \\ Then we always have $\, 0 < \mathbb{P}^{x} \big( S < \infty \big) < 1\,$.
\end{cor}

\begin{remark}
\label{compare}  
{\it Comparison with one-dimensional results:} Theorems \ref{Thm: XS} and \ref{Thm: Sfinite} include the results of Propositions 5.5.22, 5.5.32  and Theorem 5.5.29 in \cite{KS},   since a scalar diffusion can be seen as a \textsc{Walsh} diffusion with $\, {\bm \nu}(\{ 0\}) = {\bm \nu}(\{ \pi\}) =  1/2 \,$. 
These more general results are actually {\it simpler to present, and more revealing:}     cases (b)-(d) of Proposition  5.5.22 in \cite{KS}, and all cases of Proposition 5.5.32 in \cite{KS}, are summarized by case {\bf (i)} of Theorem \ref{Thm: XS} and by case {\bf (ii)} of Theorem \ref{Thm: Sfinite}, respectively --- and each of them with just one, concise condition.
\end{remark}

\section{Optimal Control / Stopping of a \textsc{Walsh} Semimartingale on the Unit Disc}

We consider a \textsc{Walsh} semimartingale $\, X (\cdot) \,$ as in Definition \ref{def: Walshsemi},   i.e., a semimartingale on rays with the  property   (\ref{eq: LTA})   for a  fixed   measure $\,\bm\nu\,$. This process $\, X (\cdot)\,$ takes values in the closed unit disc $\, \overline{B}\,$  with
\begin{equation}
\label{eq: unit_disc}
B\,: = \, \big\{ (r, \theta): \, r \in [0 ,1 ), \, \theta \in [0, 2 \pi) \big\} \, ,
\end{equation}
and is driven by an \textsc{It\^o} process $\, U (\cdot)\,$ whose local drift and dispersion processes $\,  \beta (\cdot), \sigma(\cdot) \,$ are controlled. 

\newpage
More precisely, we assume now that, for every $\, \xi \in \check{B} = B \setminus \{ \bm 0 \} \,,$  there is a nonempty subset $\, \mathcal{K} (\xi)\,$ of $\, \R \times (0, \infty) \,$, serving as the ``control space" at $\, \xi \,$; i.e., the process $\,  (\beta(\cdot), \sigma(\cdot) )\,$ takes value in $\, \mathcal{K} (\xi)\,$ at time $\, t \in [0, \infty)\,$, whenever the current position is $\, X(t) = \xi \,$. We also set $\, \mathcal{K} (\xi) = \{ (0,0) \}\,$ whenever $\, \Vert \xi \Vert = 1 \,$, meaning that $\, X (\cdot) \,$ is absorbed upon reaching the boundary of $\, B \,$. We do {\it not} assume, however, that there is a control space at the origin;    we posit rather that, when at the origin, the process $\, X (\cdot)\,$ is ``immediately   dispatched  along some ray", i.e., that $\, X (\cdot)\,$ satisfies the non-stickiness requirement in  (\ref{conditionWalsh}).

\medskip
To make all this more precise, consider on a filtered probability space $\,(\Omega, \mathcal{F}, \mathbb{P}), \,$ $\, \mathbb{F} = \{ \mathcal{F}(t) \}_{0 \leq t < \infty}  \,$ an $\, \overline{B}-$valued \textsc{Walsh} semimartingale $\, X (\cdot) \,$ with angular measure $\,\bm\nu$
, such that $\, X(\cdot)\,$ satisfies (\ref{conditionWalsh})  and  
\begin{equation}
\label{eq: controlled||X||}
{\rm d} \Vert X (t) \Vert \, =\,  {\bf 1}_{\{ X(t) \neq {\bm 0} \}}\, \big( \, \beta (t) \, {\rm d} t \, +\, \sigma (t)\, {\mathrm d} W(t) \,\big) \, + \, {\rm d} L^{\Vert X \Vert} (t)\, , \quad\, X(0) = x \in B \, .
\end{equation}
Here $\, W (\cdot) \,$ is an $\,\mathbb{F}-$Brownian motion, and $\,  \beta(\cdot), \sigma(\cdot) \,$ are $\,\mathbb{F}-$progressively measurable processes, satisfying almost surely  the integrability and consistency conditions 
\begin{equation}
\label{eq: betasigma}
\int_{0}^{t} {\bf 1}_{\{ X(u) \neq {\bm 0} \}}\, \big( \vert \beta (u) \vert  \, +\, \sigma^{2} (u)\big) \, {\rm d} u \, < \infty,\,\,\,\forall\,t \in [0, \infty), \, \,\,\, ~ \text{and} \,\,~\,\, (\beta(t), \sigma(t) )  \in \mathcal{K} \big( X(t) \big) \,\,\, \text{when}~\, X(t)\neq {\bm 0} . 
\end{equation}
Given an initial position $\, x \in B\,$, we denote by $\, \mathcal{A} (x) \,$ the collection of all \textsc{Walsh} semimartingales $\, X (\cdot)\,$ which can be constructed as above, and are thus   ``available"   to the controller at initial position $\,  x \in B \,$.  

For every planar semimartingale $\, X  (\cdot) \in \mathcal{A} (x) \,$, we denote by $\, \mathcal{J}_{X}\,$   the class of all $\, \mathbb{F}^{X}-$stopping times, from which the controller can also choose a way to stop the controlled process $\, X (\cdot)\,$. We refer to \cite{KS1} and \cite{PS} for similar considerations  regarding the collection of all   available processes (the   ``gambling house" in the terminology of \textsc{Dubins \& Savage} \cite{DS}). We use   the convention $\, U (X(\infty)) = \limsup_{t \rightarrow \infty} U(X(t))\,$.  

\begin{prob}
\label{probe}
{\bf Control and Stopping of a \textsc{Walsh} Semimartingale:} Consider as our ``reward  function" a bounded, measurable   $\, U: \overline{B} \rightarrow \R\,,$   continuous in the tree-topology. We want to find, for each starting position $\, x \in B \,$, a process $\, X^{\ast} (\cdot)\in \mathcal{A} (x) \,$ and a stopping time $\, \tau_{\ast} \in  \mathcal{J}_{X^{\ast}}\,$ that attain the supremum    
\begin{equation}
\label{eq: value}
V(x) := \, \sup_{X \in \mathcal{A} (x) , \, \tau \in \mathcal{J}_{X}} \mathbb{E} \big[ U (X(\tau)) \big] \, . \quad
\end{equation}
\end{prob}

This is a {\it stochastic control problem with discretionary stopping}   in the spirit of \cite{DZ}, \cite{KO}, \cite{KS1},  for a \textsc{Walsh} semimartingale. We shall solve this problem fairly explicitly  under some mild additional  regularity assumptions and   in a manner inspired by  \cite{KS1}, which  treats a one-dimensional analogue. It is surprising, to us at least, that this problem should admit such a very explicit solution; this   is given in Theorem \ref{thm: ControlStop}, Subsection 5.3, with the help of the results developed in Sections 2-4.

\subsection{Optimal Stopping of a \textsc{Walsh} Diffusion on the Unit Disc}\label{subs: OS}

Let $\, X(\cdot) \,$ be a \textsc{Walsh} diffusion with values in the unit disc $\, B \,$ of \eqref{eq: unit_disc}, associated with some given triple $\, ({\bm b}, {\bm s}, {\bm \nu})\,,$  where the functions $\, {\bm b} : \check{B} \rightarrow \R \,$ and $\, {\bm s} : \check{B} \rightarrow \R \,$ satisfy Condition \ref{bsigma} with $\, \ell(\theta) \equiv 1\,$. We recall the radial scale function of  (\ref{eq: p}),   and assume    
\begin{equation}
\label{eq: pfinite}
p_{\theta} (1-) < \infty \, , \qquad \forall \, \,\,\theta \in [0, 2\pi)\, .
\end{equation}
Considering the same function $\, U \,$ as in Problem \ref{probe}, we define the value function of the {\it optimal stopping problem}  for $\, X (\cdot)  \,$ by 
\begin{equation}
\label{eq: valuestop}
Q(x) : = \, \sup_{\tau \in \mathcal{J}_{X}} \mathbb{E}^{x} \big[ U(X(\tau)) \big] \, , \quad x \in B \, .
\end{equation}
We are using here the superscript $\, x \,$ for the starting position, as in Section 4; we note that there is no superscript in (\ref{eq: value}), as the starting point $\, x \,$ is implied through the requirement $\, X(\cdot) \in \mathcal{A}(x)\,$. (\ref{eq: valuestop}) gives a pure   optimal stopping  problem    for the \textsc{Walsh} diffusion process $X(\cdot)$, without any   element of control.

In the standard theory of optimal stopping for one-dimensional diffusions  on a finite interval, the value function is given by the smallest $\, \mathcal{S}-$concave  majorant of the reward function, where $\, \mathcal{S}\,$ is the scale function of the one-dimensional diffusion under consideration. We recall  that {\it a function is said to be $\, \mathcal{S}-$concave, if and only if it is a concave function of $\, \mathcal{S} \,$.} This $\, \mathcal{S}-$concavity is the precise characterization of all excessive functions for a one-dimensional diffusion; those functions turn the diffusion into a (local) supermartingale. We refer to the works \cite{DY}, \cite{DK} and to the references cited there  for treatments of the optimal stopping problem in the  context of one-dimensional diffusions, and for some properties of $\, \mathcal{S}-$concave functions. 

\smallskip
For a given \textsc{Walsh} diffusion $\, X  (\cdot)  $, a natural guess from the change-of-variable formula of Theorem \ref{prop: itoformula}  is that an excessive function $\, g \,$ for $\, X  (\cdot)\,$ should have for every $\,\theta\,$ the $\, p_{\theta}-$concavity property along the ray of angle $\,\theta\,$, and satisfy the additional requirement  
\begin{equation}
\label{eq: concaveorigin}
\int_{0}^{2\pi} D^{+} g_\theta(0) \,  {\bm \nu}({\mathrm d} \theta) \, \leq \, 0 \,.
\end{equation}
This requirement ensures the supermartingale property of $\, g (X(\cdot))\,$ when $\, X (\cdot)\,$ passes through the origin. 

Condition (\ref{eq: concaveorigin}) was  considered also in \cite{FK}, where a characterization of all excessive functions for a \textsc{Walsh} Brownian motion was obtained. In the more general setting of a \textsc{Walsh} diffusion as considered here, we cannot obtain  such a characterization, due to the angular dependence   in the drift and dispersion characteristics that prevents the use of one-dimensional excursion theory. We can, however, use the above idea to describe  precisely the value function $\, Q\,$ of the pure optimal stopping problem in (\ref{eq: valuestop}), with the help of the \textsc{Freidlin-Sheu}-type change-of-variable formula in Theorem \ref{prop: itoformula}.

\begin{definition}  {\bf Concavity:} 
\label{def: Walshcon}
A function $\, g:  \, \overline{B}\rightarrow \R\,$ is said to be {\it $\, p-$concave with angular measure $\, \bm \nu\,$,} if

\smallskip
\noindent
{\bf (i)} \,\,\, for every $\, \theta \in [0, 2\pi) \,$, the function $\, r \mapsto g_{\theta}(r) \,$ is $\, p_{\theta}-$concave, i.e., $\, g_{\theta}(r) = \widetilde{g}_{\theta} \big( p_{\theta}(r) \big)\,$, $\, r \in [0,1]\,$  holds for some concave function $\, \widetilde{g}_{\theta}: [0, p_{\theta}(1-)] \rightarrow \R\,$, and 

\noindent
{\bf (ii)} \,\,   the condition  (\ref{eq: concaveorigin}) is satisfied.
\end{definition}
 
\begin{definition} {\bf Pencil of Least Concave Majorants:} 
\label{def: concavec}
For the reward function $\, U\,$of Problem \ref{probe},  and for every constant $\, c \geq U({\bm 0}) \,$, we define the function $\, U^{(c)}: \overline{B} \rightarrow \R\,$ via
\begin{equation}
\label{eq: Uc}
U^{(c)} (r, \theta) \equiv U^{(c)}_{\theta} (r) : =  \inf \big\{ \varphi (r) : \, \varphi (\cdot) \geq U_{\theta}(\cdot) ,  \,\,\, \varphi :  [0,1]\rightarrow \R \,\, \text{is $\, p_{\theta}-$concave,} \,\, \varphi (0) \geq c  \big\} .~~
\end{equation}
\end{definition}

The functions  introduced in Definition \ref{def: concavec} will be seen in Theorem \ref{thm: solstop} to provide the crucial link between the problem of finding the smallest $\, p-$concave majorant of $\, U\,$ with angular measure $\, \bm \nu\,,$ and the analogous problem along each ray. The following result gives some useful properties of the function $\, U^{(c)}\,$ in \eqref{eq: Uc}; its proof is in the Appendix, Section 6. Analogues of statement (ii) in Proposition \ref{prop: Uc} have been considered already; see Section III.7 of \cite{DY}, Section 4 of \cite{S}, and Section 3 of \cite{KS2}.    
\begin{prop}
\label{prop: Uc}
{\bf (i)} \,\, For every real constant $\, c \geq U({\bm 0}) \,$,  the function $\, U^{(c)}\,$ of (\ref{eq: Uc}) is continuous in the tree topology  and satisfies $\, U^{(c)} ({\bm 0}) = c\,,$ as well as  $\, U^{(c)} (1, \theta) = U(1, \theta) \,$ for all $\, \theta \in [0, 2\pi)\,$. 

\smallskip
\noindent
{\bf (ii)}\,\,
Whenever we have $\, U^{(c)}_{\theta} (r) > U_{\theta} (r) \,$   for some $\,\theta\,$ and for all $\, r \,$ in some  interval $\, (r_{1} , r_{2}) \subset [0,1]\,$, the mapping  $\, r\mapsto U^{(c)}_{\theta} (r)\,$ is an affine    transformation of $\, r\mapsto p_{\theta} (r) \,$ on $\, [r_{1} , r_{2}] \,$.  

\smallskip
\noindent
{\bf (iii)}\,\,
If $\, c > U({\bm 0}) \,$, then the function $\, U^{(c)} \big \vert_{B}\,$ belongs to the class $\, \mathfrak{C}_{B} \,$ (cf. Definition \ref{def: D3} and Remark \ref{rmk: DIextend}).

\smallskip
\noindent
{\bf (iv)}\,  The function $  \Phi : [U({\bm 0}) , \infty) \rightarrow \R \cup \{ \infty \}  $ below is well-defined, continuous, and strictly decreasing:
\begin{equation}
\label{eq: phi}
\Phi (c)\, : =  \int_{0}^{2 \pi} D^{+} U^{(c)}_{\theta} (0)\, {\bm \nu} ({\rm d} \theta).
\end{equation}
\end{prop}

We have the following crucial result, regarding the problem of optimal stopping in \eqref{eq: valuestop}. 

\begin{thm}
\label{thm: solstop}
{\bf Solving the Optimal Stopping Problem:} 
In the context specified at the beginning of this subsection, the value function $\, Q : \overline{B} \rightarrow \R \,$ of the optimal stopping problem  defined as in (\ref{eq: valuestop}) and with  $\, Q(x) : = U(x) \,$ for  $\,\Vert x \Vert = 1\,,\,$ is continuous in the tree-topology. \newpage

 \smallskip
\noindent 
(i) This function $Q$ is the smallest $\, p-$concave majorant of $\,U $  with angular measure $\,\bm \nu\,;$  in particular, $\, Q \,$ itself is $\, p-$concave with angular measure $\, \bm \nu\,$, and can be written as  the lower envelope   
\begin{equation}
\label{eq: Q}
Q(x) \, = \, \inf \big\{ g(x) : \, \,g(\cdot) \geq U(\cdot) , \,\,\, g: \, \overline{B}\rightarrow \R \,\,\, \text{is $\, p-$concave with angular measure $\,\bm \nu\,$} \big\} 
\end{equation}
of all such functions that dominate $U$. Moreover, the stopping time 
\begin{equation}
\label{eq: tauast}
\tau_{\star} := \, \inf \big \{ t \geq 0 : \, U\big( X(t) \big) = Q \big( X(t) \big) \big \}
\end{equation}
belongs to the class $\,\mathcal{J}_{X} \,$  and attains the supremum in (\ref{eq: valuestop}).

\smallskip
\noindent 
(ii) 
The function $\, Q\,$ can also be cast in the form
\begin{equation}
\label{eq: Qalter}
Q(x) \, = \, U^{(c_{0})} (x)\, , \qquad \text{with} \quad c_{0} : = \, \inf \big \{ c \geq U({\bm 0}): \, \Phi (c) \leq 0 \big \} \, ;
\end{equation}
 here $\, U^{(c_{0})}\,$ is as   in Definition \ref{def: concavec}, and $\, \Phi\,$ is given by (\ref{eq: phi}). Moreover, if $\, Q({\bm 0}) = c_{0} > U({\bm 0})\,$, then $\, Q \,$ ``has no concavity at the origin", in the sense that
\begin{equation}
\label{eq: slopeaver}
\int_{0}^{2\pi} D^{+} Q_\theta (0) \,  {\bm \nu}({\mathrm d} \theta) \, = \, 0 \, .
\end{equation}
\end{thm}

\begin{remark}
\label{rmk: optstop}
The property (\ref{eq: slopeaver}) is the counterpart at the origin  of the property in Proposition \ref{prop: Uc}(ii).    Taken  together, these two properties ensure that the process $\, Q  ( X(\cdot)  )\,$   ``is a martingale  before entering the  stopping region" $\,\{ x \in \overline{B}: \, U(x) = Q(x) \} \,;$ to wit, that $\, Q \big( X(\cdot \, \wedge \tau_\star) \big)\,$ is a martingale. On the other hand, the $p-$concavity with angular measure $\, \bm \nu\,$ of the function $Q,$ ensures that $\, Q  ( X(\cdot)  )\,$ is a supermartingale.
\end{remark}

\medskip
\noindent
{\it Proof of Theorem \ref{thm: solstop}:} 
We shall show first that the representations (\ref{eq: Q}) and (\ref{eq: Qalter}) are equivalent;  then   that (\ref{eq: Qalter}) holds, and  the stopping time of (\ref{eq: tauast}) attains the supremum in (\ref{eq: valuestop}).  The remaining  claims will follow   directly from (\ref{eq: Qalter}) and   Proposition \ref{prop: Uc}. 

\smallskip
\noindent
$\bullet~$ From Proposition \ref{prop: Uc}, it is clear that the function $\,   U^{(c_{0})} \,$ is $\, p-$concave with angular measure $\,\bm \nu\,$. On the other hand, taking any function $\, g: \, \overline{B}\rightarrow \R\,$ that is $\, p-$concave with angular measure $\, \bm\nu\,$ and dominates $\,U ,$  we have $\, g({\bm 0}) = U^{(g({\bm 0}))}({\bm 0})\,$ and $\, g(\cdot) \geq U^{(g({\bm 0}))}(\cdot)\,$, therefore
\[
\Phi (g({\bm 0})) = \int_{0}^{2 \pi} D^{+} U^{(g({\bm 0}))}_{\theta} (0)\, {\bm \nu} ({\rm d}   \theta) \leq  \int_{0}^{2 \pi} D^{+} g_{\theta} (0)\, {\bm \nu} ({\rm d} \theta) \leq 0 \, .
\]
It follows that $\, g({\bm 0}) \geq c_{0}\,,$ and consequently $\, g(\cdot) \geq U^{(g({\bm 0}))}(\cdot) \geq U^{(c_{0})}(\cdot)\,$. We have thus shown that (\ref{eq: Q}) and (\ref{eq: Qalter}) are equivalent.

\smallskip
\noindent
$\bullet~$
Next, we show that $\, U^{(c_{0})} (x) = Q(x) \,$. The main idea lies in the following claim.

\begin{claim}
\label{cl: optimal}
The process $\,U^{(c_{0})} \big( X(\cdot) \big) \,$ is a bounded supermartingale; moreover, with $$\, \widetilde{\tau}_{\star} \,: = \,\inf \big\{ t \geq 0 : \, U\big( X(t) \big) = U^{(c_{0})} \big( X(t) \big) \big\} \, ,$$    the stopped process $\,U^{(c_{0})} \big( X(\cdot \wedge \widetilde{\tau}_{\star} ) \big) \,$ is a bounded martingale.
\end{claim}
  
\begin{proof}
{\bf (A)} {\it We   consider     $\, c_{0} > U({\bm 0})\,$ first.} \,Then $\, \int_{0}^{2 \pi} D^{+} U^{(c_{0})}_{\theta} (0)\, {\bm \nu} ({\rm d}   \theta) = 0\,$ and $\, U^{(c_{0})} \in \mathfrak{C}_{B}\,$ hold, thanks to Proposition \ref{prop: Uc}\,(iii),\,(iv).  We recall the explosion time $\, S: = \inf \{ t: \, \Vert X(t) \Vert = 1\}\,,$ and consider the stopping times $\, S_{n}: = \inf \{ t: \, \Vert X(t) \Vert = 1-(1/n)\}\,$, $\, n \in \mathbb{N}\,$. Now (\ref{eq: ito})-(\ref{eq: VXg2}) of Theorem \ref{prop: itoformula}   give  
\[
U^{(c_{0})} \big(X(\cdot \wedge S_{n})\big) = U^{(c_{0})} \big( X(0) \big) \,+ \int_0^{\cdot \wedge S_{n}} \mathbf{ 1}_{\{ X(t) \neq \mathbf{ 0}\}}  D^{-} U^{(c_{0})}_{\Theta (t)} \big( \Vert X(t)\Vert  \big) \, \big[ {\bm b} \big(X(t)\big) \, {\rm d}t + {\bm s} \big(X(t)\big) \, {\rm d} W(t) \big] 
\]
\[
 + \, \sum_{\theta \in [0, 2\pi) } \int_{0}^{\cdot \wedge S_{n}} {\bf 1}_{ \{ X(t) \neq {\bm 0}, \,\, \Theta (t) = \theta \} } \int_{0}^{\infty}   L^{\Vert X \Vert} ({\rm d}t , r) \, D^{2} U^{(c_{0})}_\theta ({\rm d} r) \, + \,\Big( \int_{0}^{2\pi} D^{+} U^{(c_{0})}_\theta(0) \,  {\bm \nu}({\mathrm d} \theta) \Big)\, L^{\Vert X\Vert }(\cdot)
\]
\[
= \,\, U^{(c_{0})} \big( X(0) \big) \,+ \int_0^{\cdot \wedge S_{n}} \mathbf{ 1}_{\{ X(t) \neq \mathbf{ 0}\}}  D^{-} U^{(c_{0})}_{\Theta (t)}\big( \Vert X(t)\Vert  \big) \, {\bm s} \big(X(t)\big) \, {\rm d} W(t) \qquad\qquad\qquad\qquad
\]\newpage
\begin{equation}
\label{eq: Uc0Ito}
 + \, \sum_{\theta \in [0, 2\pi) } \int_{0}^{\cdot \wedge S_{n}} {\bf 1}_{ \{ X(t) \neq {\bm 0}, \,\, \Theta (t) = \theta \} } \int_{0}^{\infty}   L^{\Vert X \Vert} ({\rm d}t , r) \, \bigg[ D^{-} U^{(c_{0})}_{\theta}(r) \frac{{2 \bm b} (r,\theta)}{{\bm s}^{2}(r,\theta)}{\rm d}r + D^{2} U^{(c_{0})}_\theta ({\rm d} r) \bigg]\, .
\end{equation}
Let us  assume that the function $\, U^{(c_{0})}_{\theta}\,$ is of the form $\, U^{(c_{0})}_{\theta}(\cdot) = \widetilde{U}^{(c_{0})}_{\theta}\big(p_{\theta}(\cdot)\big)\,$, with $\, \widetilde{U}^{(c_{0})}_{\theta} : [0,p_{\theta}(1-)]\rightarrow \R\,$   concave. We have then
\[
 \bigg[ D^{-} U^{(c_{0})}_{\theta}(r) \frac{{2 \bm b} (r,\theta)}{{\bm s}^{2}(r,\theta)}{\rm d}r + D^{2} U^{(c_{0})}_\theta ({\rm d} r) \bigg] \, = \,   D^{-} \widetilde{U}^{(c_{0})}_{\theta}(p_{\theta}(r)) \, p^{\prime}_{\theta}(r) \frac{{2 \bm b} (r,\theta)}{{\bm s}^{2}(r,\theta)}{\rm d}r \,+ \, {\rm d} \big( D^{-} \widetilde{U}^{(c_{0})}_{\theta}(p_{\theta}(r)) \,p^{\prime}_{\theta}(r) \big)
\]
\[
\, =\, - D^{-} \widetilde{U}^{(c_{0})}_{\theta}(p_{\theta}(r)) \, p^{\prime \prime}_{\theta}(r) {\rm d}r \, + \, D^{-} \widetilde{U}^{(c_{0})}_{\theta}(p_{\theta}(r)) \, p^{\prime \prime}_{\theta}(r) {\rm d}r  + \, p^{\prime}_{\theta}(r) \, {\rm d} \big( D^{-} \widetilde{U}^{(c_{0})}_\theta (p_{\theta}(r))\big)\,
\]
\begin{equation}
\label{eq: UcXboundedvari}
 =\,  p^{\prime}_{\theta}(r)\, {\rm d} \big( D^{-} \widetilde{U}^{(c_{0})}_\theta (p_{\theta}(r))\big)\, .
\end{equation}
The last expression is nonpositive, since $\, \widetilde{U}^{(c_{0})}_{\theta}\,$ is concave; yet it vanishes near $\, r\,$ if $\, U^{(c_{0})}_{\theta} (r) > U_{\theta} (r)\,$, thanks to Proposition \ref{prop: Uc}(ii). On the other hand, if $\,U^{(c_{0})}_{\theta} (r) = U_{\theta}(r)\,$, then by the definition of $\,\widetilde{\tau}_{\star}\,$ and the nature of local times, the process $\, L^{\Vert X \Vert} (\cdot \wedge \widetilde{\tau}_{\star} ,r ) \,$ does not increase when $\, X(\cdot) \,$ is on the ray with angle $\,\theta\,$. 

Putting these observations together, we see that  the right-most side in (\ref{eq: Uc0Ito}) is a local supermartingale; and that if we stop this process at time $\, \widetilde{\tau}_{\star}\,$, we get a local martingale. As it is clear that the function $\, U^{(c_{0})}\,$ is  bounded, we   let $\, n \rightarrow \infty\,$ and obtain the claim.    

\smallskip
\noindent
{\bf (B)} \,{\it We consider next the case $\, c_{0} = U({\bm 0}) \,$.} Then $\, \Phi (c_{0}) \leq 0 $ holds, and therefore also does  $  \Phi (c) < 0 $ for any $  c> c_{0} $. Thus for any such $\, c> c_{0}\,$, we   apply Theorem \ref{prop: itoformula} as above and show that $\, U^{(c)} \big( X(\cdot) \big)\,$ is a bounded supermartingale. Since $$\, U^{(c_{0})}(\cdot) \leq U^{(c)} (\cdot) \leq U^{(c_{0})}(\cdot) + c - c_{0}\,$$ (clearly, $\, U^{(c_{0})}_{\theta} (\cdot) + (c-c_{0})\,$ is $\,p_{\theta}-$concave and dominates $\, U $), we   let $\, c \downarrow c_{0}\,$ and obtain that  the process $\, U^{(c_{0})} \big( X(\cdot) \big)\,$ is also a bounded supermartingale. 

\smallskip
On the other hand, since $\, c_{0} = U({\bm 0}) \,$, the process $\, X(\cdot   \wedge \widetilde{\tau}_{\star})  \,$ stops at the origin once it finds itself there; so it never changes the ray it is on, and $\, L^{\Vert X\Vert }(\cdot \,\wedge \widetilde{\tau}_{\star} ) \equiv 0\,$. Thus,   the one-dimensional generalized \textsc{It\^o} rule shows that $\, U^{(c)} \big( X(\cdot \, \wedge \widetilde{\tau}_{\star}) \big)\,$ is a bounded (local) martingale, following the same idea as above.
\end{proof}

From the Claim \ref{cl: optimal}, and for any stopping time $\,\tau \in \mathcal{J}_{X}\,$, we have 
$$
\, U^{(c_{0})} (x) \geq \mathbb{E}^{x} \big[ U^{(c_{0})} \big( X(\tau) \big) \big] \geq \mathbb{E}^{x} \big[ U \big( X(\tau) \big) \big] \, .
$$ 
Furthermore, $\, U^{(c_{0})} (x) = \mathbb{E}^{x} \big[ U^{(c_{0})} \big( X( \widetilde{\tau}_{\star}  ) \big) \big] = \mathbb{E}^{x} \big[ U \big( X(  \widetilde{\tau}_{\star}  ) \big) \big] \,$, where the last equality holds because $\, X(S) \in \partial B \,$ and $\, U^{(c_{0})} (\cdot)= U(\cdot)\,$ on $\,\partial B \,$. These facts come from (\ref{eq: pfinite}), Theorem \ref{Thm: XS}, and Proposition \ref{prop: Uc}(i).  We conclude that $\, U^{(c_{0})} (x) = Q(x) \,$, and that the stopping time $\, \widetilde{\tau}_{\star} \,$ ($= \tau_{\star}\,$) is optimal. 

The proof of Theorem  \ref{thm: solstop} is complete.   \qed

\subsection{Solution to the  Problem of Optimal Stochastic Control with Discretionary Stopping}

Let us return now to the context of Subsection 5.1, and deal with Problem \ref{probe} of stochastic control with discretionary stopping. We shall provide a characterization of the value function of this problem, as well as an explicit description of a control strategy and of a stopping time  that attain the supremum in (5.4). 

\begin{ass}
\label{ass: sigtonoise}
 There are two pairs $\, ({\bm b}_{0}, {\bm s}_{0}) \,$, $\,  ({\bm b}_{1}, {\bm s}_{1})\,$ of \textsc{Borel}-measurable functions on $\, \check{B}\,$, which  

\smallskip
\noindent
(i)\,\, satisfy Condition \ref{bsigma} with $\, \ell (\theta) \equiv 1\,$, and whose   corresponding radial scale functions 
  satisfy (\ref{eq: pfinite}); and

\smallskip
\noindent
(ii)\,\,  are such    that $\, ({\bm b}_{i}(x), {\bm s}_{i}(x)) \in \mathcal{K}(x)\,$ holds for all $\, x \in \check{B} \,$ and $\, i = 0,1\,$, and \newpage
\begin{equation}
\label{eq: bs01}
\frac{{\bm b}_{0}(x)}{{\bm s}^{2}_{0}(x)} \, = \, \inf \Big\{ \frac{\beta}{\sigma^{2}} : \, (\beta, \sigma) \in \mathcal{K}(x) \Big\}, \qquad \, \frac{{\bm b}_{1}(x)}{{\bm s}^{2}_{1}(x)} \, = \, \sup \Big\{ \frac{\beta}{\sigma^{2}} : \, (\beta, \sigma) \in \mathcal{K}(x) \Big\}. 
\end{equation}
\end{ass}
Results in this subsection will rely on the above Assumption \ref{ass: sigtonoise}, which is inspired by \cite{KS1}. Following principles of stochastic control and stopping (e.g. Theorems 3.6 and 4.5 in \cite{T}), we may write  informally, using the stochastic calculus in Section 2, the following \textsc{Hamilton-Jacobi-Bellman}-type variational inequalities for the value function $\,(r, \theta) \mapsto V (r, \theta) = V_{\theta}(r) \,$ of (\ref{eq: value}), namely
\begin{equation}
\label{eq: DPE}
\min \Big \{ - \sup_{(\beta, \sigma) \in \mathcal{K} (x) } \big\{ \beta \, D V_{\theta}(r) \, + \, \frac{1}{2} \, \sigma^{2} \, D^{2} V_{\theta}(r)\big\} \, ,\, V(x) - U(x)  \Big \} \, = \, 0\, , \quad x = (r, \theta) \in B  \,,
\end{equation}
and
\begin{equation}
\label{eq: DPEorigin}
\min \Big \{ - \int_{0}^{2 \pi} D^{+} V_{\theta} (0)\, {\bm \nu} ({\rm d}   \theta) \, ,\,\,\, V({\bm 0}) - U({\bm 0})  \Big \} \, = \, 0\,  .
\end{equation}
Equation (\ref{eq: DPE}) implies that, outside the stopping region (i.e., where $\, V(x) > U(x)$), we should have
\[
\sup_{(\beta, \sigma) \in \mathcal{K} (x) } \Big\{ \, \frac{\beta}{\sigma^{2}} \, D V_{\theta}(r)  +  \frac{1}{2} \,   D^{2} V_{\theta}(r)\Big\} \, = \, 0\, ,
\] 
as $\,\sigma\,$ is not allowed to be zero. This suggests   maximizing the {\it ``signal-to-noise ratio" $\, \frac{\beta}{\sigma^{2}}\,$} where $\, D V_{\theta}(r) > 0 \,$, and minimize it where $\, D V_{\theta}(r) < 0 \,$. Now to analyze the sign of $\, D V_{\theta}(r) \,$, we introduce (again inspired by \cite{KS1}), for every $\, \theta \in [0,2\pi) ,$ the maximum of the reward function $\, U \,$ on the corresponding ray, as well as the left-most and right-most    locations where this maximum is   attained, namely:
\begin{equation}
\label{eq: Umax}
U_{\theta}^{\ast}: = \, \max_{0\leq r \leq 1} U_{\theta}(r) \, ,
\end{equation}
\begin{equation}
\label{eq: Umaxloca}
\lambda_{\theta} : = \inf \{ r \in [0,1] : U_{\theta}(r) = U_{\theta}^{\ast} \} \, , \qquad  \varrho_{\theta} : = \sup \{ r \in [0,1] : U_{\theta}(r) = U_{\theta}^{\ast} \} \, .
\end{equation}
We note  however that,   in contrast to the one-dimensional problem, here the left endpoint (i.e., the origin) is not an  absorbing boundary, and thus the value $\, V({\bm 0})\,$ is not known in advance.

\smallskip
Thus, we shall treat every real number $\, c \geq U(\bm 0)\,$ as a ``candidate" for the value $\, V (\bm 0)\,$, and choose a pair of functions $\, ({\bm b}^{(c)}, {\bm s}^{(c)}) \,$ on $\, \check{B}\,$ that will generate a   \textsc{Walsh} diffusion which will be optimal for this problem. If indeed   $\, c=V (\bm 0)\,$, the function $\, V \,$ is then the value function of the optimal stopping problem for this \textsc{Walsh} diffusion 
 and reward function $\,U\,$. From Theorem \ref{thm: solstop} (ii) we know that, with $\, U^{(c, p^{(c)})}\,$ as in Definition \ref{def: Ucpc} below, we should have $\, V = U^{(c, p^{(c)})}\,$ with $\, c=V (\bm 0)\,$.

\smallskip
To implement this program 
we choose the pairs of functions $\, ({\bm b}^{(c)}, {\bm s}^{(c)}) \,$ as follows. 

\begin{definition} 
{\bf Candidate Optimal Control Strategies:} 
\label{def: bcsc}
For every real constant $\, c \geq U(\bm 0)\,$, we consider a pair $\, ({\bm b}^{(c)}, {\bm s}^{(c)}) \,$ of \textsc{Borel}-measurable functions on $\, \check{B}\,$, which satisfies:

\smallskip
\noindent
(i) \,\,  $\, ({\bm b}^{(c)}(x), {\bm s}^{(c)}(x)) \in \mathcal{K}(x)\,$ for all $\, x \in \check{B} \,$; 
  
\smallskip
\noindent
(ii) \,   
$\, \big( {\bm b}^{(c)}(r,\theta), {\bm s}^{(c)}(r,\theta)\big) = \big( {\bm b}_{0}(r,\theta), {\bm s}_{0}(r,\theta)\big) \,$ for all $\, (r, \theta) \in \check{B} \,$ with $\, U_{\theta}^{\ast} < c \,$; 

\smallskip
\noindent
(iii) \,   
$  \big( {\bm b}^{(c)}(r,\theta), {\bm s}^{(c)}(r,\theta)\big) = \big( {\bm b}_{0}(r,\theta), {\bm s}_{0}(r,\theta)\big)  $ for all $\, (r, \theta) \in \check{B} \,$ with $\, U_{\theta}^{\ast} \geq c \,$ and $ \, r \in (\varrho_{\theta},1)$; 

\smallskip
\noindent
(iv) \,    
$  \big( {\bm b}^{(c)}(r,\theta), {\bm s}^{(c)}(r,\theta)\big) = \big( {\bm b}_{1}(r,\theta), {\bm s}_{1}(r,\theta)\big)  \,$ for all $\, (r, \theta) \in \check{B} \,$ with $\, U_{\theta}^{\ast} > c \,$ and $ \, r \in (0, \lambda_{\theta})$;
\end{definition}



\begin{definition}
\label{def: Ucpc}
For every real constant  $\, c \geq U({\bm 0}) \,$, we define the function $\,U^{(c, p^{(c)})} : \overline{B} \rightarrow \R \,$ as $\, U^{(c, p^{(c)})} (r,\theta) \equiv U^{(c, p^{(c)})}_{\theta} (r) \,$, where 
\begin{equation}
\label{eq: Ucpc}
 U^{(c, p^{(c)})}_{\theta} (r): =  \inf \Big\{ \varphi(r): \, \varphi (\cdot) \geq U_{\theta}(\cdot) ,  \quad \varphi :  [0,1]\rightarrow \R\, \,\, \text{is $\, \,p^{(c)}_{\theta}-$concave with} \,\,\, \varphi (0) \geq c  \Big\} \,.
\end{equation}
Here $\, p^{(c)} \,$ is the radial scale function that corresponds,   via (\ref{eq: p}), to the above pair of  functions $\, ({\bm b}^{(c)}, {\bm s}^{(c)}) \,$.
\end{definition}


 \begin{remark}
 By Definition \ref{def: bcsc} (i) and Assumption \ref{ass: sigtonoise}, every pair of functions $\, ({\bm b}^{(c)}, {\bm s}^{(c)}) \,$ also satisfies condition (i) of  Assumption \ref{ass: sigtonoise}, so Theorem \ref{thm: solstop} applies to the \textsc{Walsh} diffusion it generates.\newpage
 
We also note that, in Definition \ref{def: bcsc}, we did not specify the values $\, ({\bm b}^{(c)}(r,\theta), {\bm s}^{(c)}(r,\theta)) \,$ in the case  $\, U_{\theta}^{\ast} = c \,$ and $\, r \in (0, \varrho_{\theta}]\,,$ or in the case $\, U_{\theta}^{\ast} > c \,$ and $\, r \in [\lambda_{\theta}, \varrho_{\theta}]\,$. 

In these cases, the values in question  need only be chosen suitably,  to make the resulting functions $\, ({\bm b}^{(c)}, {\bm s}^{(c)}) \,$ satisfy the property (i) of Definition \ref{def: bcsc}. For example, $\, ({\bm b}_{0}(r,\theta), {\bm s}_{0}(r,\theta)) \,$ and $\, ({\bm b}_{1}(r,\theta), {\bm s}_{1}(r,\theta)) \,$ are two choices. Fortunately,  this ambiguity does not carry over to the function $\,U^{(c, p^{(c)})}\,$, as shown below.
\end{remark}
 
\begin{prop}
\label{prop: Ucpc}
For every real constant $\, c \geq U({\bm 0})\,$, the function 
$\,U^{(c, p^{(c)})}\,$ is uniquely determined, regardless of the ambiguity in the choice of $\, ({\bm b}^{(c)}, {\bm s}^{(c)}) \,$ in Definition \ref{def: bcsc}.  

 \smallskip
Moreover,  for any given $\, \theta \in [0,2\pi) \,$, the following hold:

\smallskip
\noindent
{\bf (i)} \, If $\, c < U_{\theta}^{\ast} \,$,   we have \\ $\, ~~~~~~D^{-} U^{(c, p^{(c)})}_{\theta} (\cdot) \geq 0 \,$ ~on $\, (0, \lambda_{\theta}) \,$, $\, \,\,D^{-}U^{(c, p^{(c)})}_{\theta} (\cdot) \leq 0 \,$ ~on $\, (\varrho_{\theta}, 1) \,$, and  $\, \,\,U^{(c, p^{(c)})}_{\theta} (\cdot) = U_{\theta}^{\ast}\,$ ~on $\, [ \lambda_{\theta},  \varrho_{\theta}]\,$.

 \smallskip
\noindent
{\bf (ii)} \, If $\, c = U_{\theta}^{\ast} \,$,  we have \\ $\, ~~~~~~~~~~~~~~~~~~~~~~~~~~~~~~~~~~~~~~D^{-} U^{(c, p^{(c)})}_{\theta} (\cdot) \leq 0 \,$  ~~on $\, (\varrho_{\theta}, 1) \,$,~~~~~~~~ and ~~~~~~~~ $\, U^{(c, p^{(c)})}_{\theta} (\cdot) = U_{\theta}^{\ast}\,$ ~~on $\, [ 0,  \varrho_{\theta}]\,$.

\smallskip
\noindent
{\bf (iii)} \, If $\, c > U_{\theta}^{\ast} \,$,  we have \\ $\,~~~~~~~~~~~~~~~~~~~~~~~~~~~~~~~~~~~~~~~~~~~~~~~~~~~~~~~~~~~~~~~ D^{-} U^{(c, p^{(c)})}_{\theta} (\cdot) \leq 0 \,$ ~~on $\, (0, 1) \,$. 

\smallskip
\noindent
{\bf (iv)} \, With $\,\big( U^{(c)}, p \big)\,$ replaced by $\,\big(  U^{(c, p^{(c)})}, \,p^{(c)} \big)\,,$  the statements of Proposition \ref{prop: Uc} hold here as well.
\end{prop}

\begin{proof}
The proof of (i)-(iii) is elementary, using the definition of $\,U^{(c, p^{(c)})}\,$; see also the end of Section 3 of \cite{KS1}, where   similar properties   are considered.  

\smallskip
Next, we show the non-ambiguity in the definition of the function $\,U^{(c, p^{(c)})}\,$ in (\ref{eq: Ucpc}). Let $\, ({\bm b}^{(c,1)}, {\bm s}^{(c,1)}) \,$ and $\, ({\bm b}^{(c,2)}, {\bm s}^{(c,2)}) \,$ be two choices of $\, ({\bm b}^{(c)}, {\bm s}^{(c)}) \,$, and $\, p^{(c,1)}\,$ and $\, p^{(c,2)}\,$  the corresponding radial scale functions. Fix a ray with angle $\,\theta \,$. If $\, U^{\ast}_{\theta} < c\,$, there is no ambiguity in $\, ({\bm b}^{(c)}, {\bm s}^{(c)}) \,,$ and therefore in $\,U^{(c, p^{(c)})}\,,$ on this ray. If $\, U^{\ast}_{\theta} = c\,$, then $\,U^{(c, p^{(c)})}_{\theta} = U_{\theta}^{\ast}\,$ on $\, [0, \varrho_{\theta}]\,$, and it follows  that 
\begin{equation}
\label{eq: Ucrho1}
U^{(c, p^{(c)})}_{\theta} (r): =  \inf \Big\{ \varphi (r): \, \varphi(\cdot) \geq U_{\theta}(\cdot) ,  ~~ \varphi:  [\varrho_{\theta}, 1]\rightarrow \R \,\, \text{is $\, p^{(c)}_{\theta}-$concave}  \Big\} \, , ~~~~~ r \in [\varrho_{\theta},1]\, .
\end{equation}
But when restricted to $\, [\varrho_{\theta},1]\,$  we have $\, \big({\bm b}^{(c,1)}(\cdot \,,\theta),\, {\bm s}^{(c,1)}(\cdot \,,\theta)\big) = \big({\bm b}^{(c,2)}(\cdot \,,\theta), \,{\bm s}^{(c,2)}(\cdot \,,\theta)\big) \,,$ and therefore the functions $\, p^{(c,1)}_{\theta}(\cdot)\,$, $\, p^{(c,2)}_{\theta}(\cdot)\,$ are affine transformations of each other. Hence, the two choices $\, p^{(c)} = p^{(c,1)}\,$ and $\, p^{(c)} = p^{(c,2)}\,$ in (\ref{eq: Ucrho1}) lead to the same result. The case $\, U^{\ast}_{\theta} > c\,$ is dealt with similarly.

\smallskip
Finally, we address (iv).  It is easy to see that   Proposition \ref{prop: Uc} carries  over to the present context essentially unchanged, except for the claim that the mapping $\, c \mapsto \int_{0}^{2 \pi} D^{+} U^{(c, p^{(c)})}_{\theta} (0)\, {\bm \nu} ({\rm d} \theta)\,$ is continuous and strictly decreasing. For this claim  it is enough to show that the mapping $\, c \mapsto D^{+} U^{(c, p^{(c)})}_{\theta} (0)\,$ is continuous and strictly decreasing, given any $\,\theta\in[0,2\pi)\,$. We now observe that  we have the freedom to choose $$\, \big({\bm b}^{(U_{\theta}^{\ast})}(\cdot , \theta), {\bm s}^{(U_{\theta}^{\ast})}(\cdot , \theta)\big) \, = \, \big({\bm b}_{0}(\cdot , \theta), {\bm s}_{0} (\cdot , \theta)\big)\, ,$$ so that $\, ({\bm b}^{(c)}(\cdot, \theta), {\bm s}^{(c)}(\cdot, \theta)) \,$ is the same for all $\, c\in [U_{\theta}^{\ast}, \infty)\,$. Then the proof of Proposition \ref{prop: Uc}\,(iv) yields that the mapping $\, c \mapsto D^{+} U^{(c, p^{(c)})}_{\theta} (0)\,$ is continuous and strictly decreasing on $\, [U_{\theta}^{\ast}, \infty)\,$. The argument is similar for $\,c \in [U({\bm 0}), U_{\theta}^{\ast}]\,$.  
\end{proof}

The last task now, is to determine $\, V(\bm{0})\,$. Following (\ref{eq: DPEorigin}) and Proposition \ref{prop: Ucpc} (iv), we naturally conjecture  $$ V({\bm 0}) = \inf \Big\{ c \geq U({\bm 0 }): \, \int_{0}^{2 \pi} D^{+} U^{(c, p^{(c)})}_{\theta} (0)\, {\bm \nu} ({\rm d} \theta) \leq 0 \Big\}\, .$$

\smallskip
We can state now and prove the following fundamental result, regarding  the optimal control problem with discretionary stopping  for \textsc{Walsh} semimartingales, posed in the present section.

\begin{thm}
\label{thm: ControlStop}
 {\bf Solving the   Control Problem with Discretionary Stopping:}
With Assumption \ref{ass: sigtonoise} and the above notations, the value function $\, V   \,$ of the control problem with discretionary stopping in (\ref{eq: value}), is  given by
\begin{equation}
\label{eq: V}
V(x) \, = \, U^{(c_{\ast}, p^{(c_{\ast})})}(x)\, , \qquad c_{\ast}: = \, \inf \bigg\{ c \geq U({\bm 0 }): \, \int_{0}^{2 \pi} D^{+} U^{(c, p^{(c)})}_{\theta} (0)\, {\bm \nu} ({\rm d} \theta) \leq 0 \bigg\}\, ,
\end{equation}
and therefore satisfies $\, V ({\bm 0}) = c_{\ast} \,$.

The  supremum in (\ref{eq: value}) is attained by the \textsc{Walsh} diffusion $\, X^{\ast} (\cdot)\,$ associated with the triple $\, \big({\bm b}^{(c_{\ast})}, {\bm s}^{(c_{\ast})}, {\bm \nu}\big)\, $ as in Definition \ref{def: bcsc}, and the corresponding stopping time 
\begin{equation}
\label{eq: optimalstop}
\tau_{\ast}:= \, \inf \big \{ t \geq 0 : \, U\big( X^{\ast}(t) \big) = V \big( X^{\ast}(t) \big) \big \} .
\end{equation}
\end{thm}

 \begin{remark} 
 \label{rem: int}
 {\it On Interpretation:} 
In conjunction with Definition \ref{def: bcsc} this result states that,  before entering the stopping region $\, \{ x \in \overline{B}: \, U(x) = V(x) \}\,$, it is optimal to control the state process $\, X (\cdot) \,$ thus:

\smallskip
\noindent
{\bf (i)} \, Along any ray of angle $\, \theta \,$ with $\, U_{\theta}^{\ast} >   V ({\bm 0})  :$  maximize the ``signal-to-noise" ratio $\,  \beta / \sigma^{2}  \,$ on the interval $\, (0, \lambda_{\theta})\,;$ minimize the ``signal-to-noise"  ratio $\,  \beta / \sigma^{2}  \,$ on the interval $\, (\rho_{\theta}, 1) \,; $ and follow  on the interval  $\, [\lambda_{\theta}, \varrho_{\theta}] \,$ any strategy that will bring the process $\, X (\cdot)  \,$ to one of the   endpoints of the interval. 

\smallskip
\noindent
{\bf (ii)} \, Along any ray of angle $\, \theta \,$ with $\, U_{\theta}^{\ast}  = V ({\bm 0})  :$ minimize the ``signal-to-noise" ratio $\,  \beta / \sigma^{2}  \,$ on 
$\, (\varrho_{\theta}, 1) \,,$ and follow on  the interval 
$\, (0, \rho_{\theta}] \,$ any strategy that will bring the process $\, X (\cdot)  \,$ to one of its endpoints. 

\smallskip
\noindent
{\bf (iii)} \, Along any ray of angle $\, \theta \,$ with $\, U_{\theta}^{\ast} <  V ({\bm 0})  :$  minimize the ``signal-to-noise" ratio $\,  \beta / \sigma^{2}  \,$.

\smallskip
Since the function $\, V \,$ is obtained via (\ref{eq: V}), the above strategy can indeed be implemented.
 \end{remark}

\begin{proof}
{\bf (A)} We first show that $\, U^{(c_{\ast}, \,p^{(c_{\ast})})} (x) \geq V(x)\,$. Let us fix a starting point $\, x \in B\,$, pick up an arbitrary  process $\, X (\cdot) \in \mathcal{A}(x)\,$, a stopping time $\,\tau \in \mathcal{J}_{X}\,$, and recall the dynamics of (\ref{eq: controlled||X||}). We claim that we have 
\begin{equation}
\label{eq: claim}
U^{(c_{\ast}, \,p^{(c_{\ast})})} (x)  \,   \geq  \,   \mathbb{E} \Big[ U^{(c_{\ast}, \,p^{(c_{\ast})})} \big(X(\tau )\big)\Big]\, .
\end{equation}
This   implies $ U^{(c_{\ast}, p^{(c_{\ast})})} (x)     \geq     \mathbb{E} \big[ U  \big(X(\tau )\big)\big]  $ for all $ X (\cdot) \in \mathcal{A}(x) ,$   $\,\tau \in \mathcal{J}_{X},$ thus also $\, U^{(c_{\ast}, p^{(c_{\ast})})} (x) \geq V(x)\,$. 

\smallskip
\noindent
$\bullet~$ Now we  establish the claim  (\ref{eq: claim}). Assume first that $\, c_{\ast} > U({\bm 0})\,$.  Proposition \ref{prop: Ucpc} (iv) gives then $\, \int_{0}^{2 \pi} D^{+} U^{(c_{\ast}, p^{(c_{\ast})})}_{\theta} (0)\, {\bm \nu} ({\rm d}   \theta) = 0\,$ and $\, U^{(c_{\ast}, p^{(c_{\ast})})} \in \mathfrak{C}_{B}\,$. In the same manner  as in the derivation  of (\ref{eq: Uc0Ito}), (\ref{eq: UcXboundedvari}), and recalling the stopping times $\, S \,$, $\, S_{n}\,$ given there, we obtain here
\[
U^{(c_{\ast}, p^{(c_{\ast})})} \big(X(\cdot \wedge S_{n})\big) = U^{(c_{\ast}, p^{(c_{\ast})})} (x ) \,+ \int_0^{\cdot \wedge S_{n}} \mathbf{ 1}_{\{ X(t) \neq \mathbf{ 0}\}} \, D^{-} U^{(c_{\ast}, p^{(c_{\ast})})}_{\Theta (t)} \big( \Vert X(t)\Vert  \big) \, \big[ \beta(t) \, {\rm d}t + \sigma (t) \, {\rm d} W(t) \big] 
\]
\[
 + \, \sum_{\theta \in [0, 2\pi) } \int_{0}^{\cdot \wedge S_{n}} {\bf 1}_{ \{ X(t) \neq {\bm 0}, \,\, \Theta (t) = \theta \} } \int_{0}^{\infty}   L^{\Vert X \Vert} ({\rm d}t , r) \, D^{2} U^{(c_{\ast}, p^{(c_{\ast})})}_\theta ({\rm d} r) 
\]
\[
\leq \, U^{(c_{\ast}, p^{(c_{\ast})})} (x ) \,+ \int_0^{\cdot \wedge S_{n}} \mathbf{ 1}_{\{ X(t) \neq \mathbf{ 0}\}}  D^{-} U^{(c_{\ast}, p^{(c_{\ast})})}_{\Theta (t)} \big( \Vert X(t)\Vert  \big) \,\Big[  \frac{{\bm b}^{(c_{\ast})}\big(X(t)\big)}{({\bm s}^{(c_{\ast})})^{2}\big(X(t)\big)} \cdot \sigma^{2} (t)\, {\rm d}t + \sigma (t) \, {\rm d} W(t) \Big] 
\]
\[
 + \, \sum_{\theta \in [0, 2\pi) } \int_{0}^{\cdot \wedge S_{n}} {\bf 1}_{ \{ X(t) \neq {\bm 0}, \,\, \Theta (t) = \theta \} } \int_{0}^{\infty}   L^{\Vert X \Vert} ({\rm d}t , r) \, D^{2} U^{(c_{\ast}, p^{(c_{\ast})})}_\theta ({\rm d} r) 
\]
\[
= \, U^{(c_{\ast}, p^{(c_{\ast})})} (x ) \,+ \int_0^{\cdot \wedge S_{n}} \mathbf{ 1}_{\{ X(t) \neq \mathbf{ 0}\}}  D^{-} U^{(c_{\ast}, p^{(c_{\ast})})}_{\Theta (t)} \big( \Vert X(t)\Vert  \big) \, \sigma (t) \, {\rm d} W(t) 
\]
\begin{equation}
\label{eq: UcpcIto}
 + \, \sum_{\theta \in [0, 2\pi) } \int_{0}^{\cdot \wedge S_{n}} {\bf 1}_{ \{ X(t) \neq {\bm 0}, \,\, \Theta (t) = \theta \} } \int_{0}^{\infty}   L^{\Vert X \Vert} ({\rm d}t , r) \, (p^{(c_{\ast})}_{\theta})^{\prime} (r)\, {\rm d} \big( D^{-} \widetilde{U}^{(c_{\ast}, p^{(c_{\ast})})}_\theta (p_{\theta}(r)) \big)\, ,
\end{equation}
\noindent
where $\, \widetilde{U}^{(c_{\ast}, p^{(c_{\ast})})}_{\theta} : [0,p_{\theta}(1-)] \to \R\,$ is concave, and such that $\, U^{(c_{\ast}, p^{(c_{\ast})})}_{\theta}(\cdot) = \widetilde{U}^{(c_{\ast}, p^{(c_{\ast})})}_{\theta}\big(p_{\theta}(\cdot)\big)\,$. 

We have used Definition \ref{def: bcsc} and Proposition \ref{prop: Ucpc} (i)-(iii) for the above inequality; namely, we observe \newpage 
\[
D^{-} U^{(c_{\ast}, p^{(c_{\ast})})}_{\Theta (t)} \big( \Vert X(t)\Vert  \big) \, > (<) \,  \,0 \,\,\, \,\,\,\Longrightarrow \,\,\, \,\,\,\frac{{\bm b}^{(c_{\ast})}\big(X(t)\big)}{\big({\bm s}^{(c_{\ast})}\big)^{2}\big(X(t)\big)} \,\geq (\leq) \, \,\frac{\beta (t)}{\sigma^{2} (t)} \, .
\]
The claim for the case $\, c_{\ast} > U({\bm 0})\,$ now follows by localizing (\ref{eq: UcpcIto}) and taking expectations, with the help of the concavity of the function $\, \widetilde{U}^{(c_{\ast}, p^{(c_{\ast})})}_{\theta}\,$ and the boundedness of the function $\, U^{(c_{\ast}, p^{(c_{\ast})})}\,$.

Next, we consider the case $\, c_{\ast} = U({\bm 0})\,$. Then we have $\, \int_{0}^{2 \pi} D^{+} U^{(c, p^{(c)})}_{\theta} (0)\, {\bm \nu} ({\rm d}   \theta) < 0\,$ and $\, U^{(c, p^{(c)})} \in \mathfrak{C}_{B}\,$, for any $\, c>c_{\ast}\,$.  Thus, similarly as above, we see that $$\, U^{(c, p^{(c)})} (x) \, \geq   \,  \mathbb{E} \big[ U^{(c, p^{(c)})} \big(X(\tau )\big)\big]\, .$$ On the strength of the following paragraph, we may let $\, c \downarrow c_{\ast}\,$ and obtain the claim in this case. 

\smallskip
Fix $\,\theta\in [0,2\pi)\,$. By making $\,  ({\bm b}^{(c)}(\cdot\,, \theta), {\bm s}^{(c)}(\cdot\,, \theta))\,$ the same for all $\,c  \geq U_{\theta}^{\ast}\,$ (cf. the proof of Proposition \ref{prop: Ucpc} (iv)), we note that there exists an $\,\varepsilon (\theta) > 0\,$ such that $\, p^{(c)}_{\theta}(\cdot)\,$ is the same for $\, c \in [c_{\ast}, c_{\ast}+\varepsilon(\theta)]\,$. Thus $\, U^{(c, p^{(c)})} (\cdot\,, \theta) \leq  U^{(c_{\ast}, p^{(c_{\ast})})}(\cdot\,, \theta) + c - c_{\ast}\,$ for $\, c \in [c_{\ast}, c_{\ast}+\varepsilon(\theta)]\,$.  

\smallskip
\noindent
{\bf (B)} We need to argue 
$\, U^{(c_{\ast}, p^{(c_{\ast})})} (x) \leq V(x)\,$ as well. But this follows from the fact that, by Theorem \ref{thm: solstop}, $\, U^{(c_{\ast}, p^{(c_{\ast})})} \,$ is the value function of the optimal stopping problem for the same reward function $\, U\,$ and the \textsc{Walsh} diffusion $\, X^{\ast} (\cdot) \,$ associated with the triple $\, ({\bm b}^{(c_{\ast})}, {\bm s}^{(c_{\ast})}, {\bm \nu})\,$. 

We conclude that $\, U^{(c_{\ast}, p^{(c_{\ast})})} (x) = V(x)\,$; the other claims of the theorem follow then directly.  
\end{proof}

The importance -- and advantage -- of the   purely probabilistic   approach we have developed, is that it obviates the need to give rigorous meaning to the fully nonlinear variational inequalities (\ref{eq: DPE}), (\ref{eq: DPEorigin});  it constructs, rather, the value function and the optimal control and stopping strategies of the problem {\it from first principles} and using educated guesses. We regard the fact, that  such a problem can be shown to admit a very explicit solution, as   testament to the power of the stochastic calculus developed in the present paper.    




\section{Appendix: Proofs of Selected Results} 

{\bf PROOF OF THEOREM \ref{prop: itoformula}:} {\it Step 1:} In this first step we extend Proposition \ref{prop: D} and Lemma \ref{lm: gprime} to functions in the class $\,\mathfrak{C}\,$. Except for Lemma \ref{lm: gprime}(ii), it is straightforward to state and prove the extension. For the extension of Lemma \ref{lm: gprime}(ii)  we shall show that,  whenever $\, g \in \mathfrak{C} \,$, the process
\begin{equation}
\label{eq: processvari}
\sum_{\theta \in [0, 2\pi) } \int_{0}^{T} {\bf 1}_{ \{ X(t) \neq {\bm 0}, \,\, \Theta (t) = \theta \} } \int_{0}^{\infty}   L^{\Vert X \Vert} ({\rm d}t , r)  \,  D^{2} g_\theta ({\rm d} r) \, , \qquad 0 \leq T < \infty \, 
\end{equation}
is well-defined, adapted, continuous and of finite variation on compact intervals. Following the idea and notation  in the proof of Lemma \ref{lm: finite} and using (iii) of Definition \ref{def: D3}, we derive
\[
\sum_{\theta \in [0, 2\pi) } \int_{0}^{T} {\bf 1}_{ \{ X(t) \neq {\bm 0}, \,\, \Theta (t) = \theta \} } \int_{0}^{\infty}   L^{\Vert X \Vert} ({\rm d}t , r)  \,  \big| D^{2} g_\theta ({\rm d} r) \big|
\]
\[
\leq \, \sum_{\theta \in [0, 2\pi) } \int_{0}^{T} {\bf 1}_{ \{ 0< \Vert X(t)\Vert \leq \eta, \,\, \Theta (t) = \theta \} } \int_{0}^{\infty}   L^{\Vert X \Vert} ({\rm d}t , r)  \,  \mu ({\rm d} r) \, + \,\sum_{\{ \ell: \,\, \tau^\eta_{2 \ell + 1} < T\}} \int_{0}^{\infty} L^{\Vert X \Vert} (T , r)  \, \big| D^{2} g_{\Theta( \tau^\eta_{2 \ell + 1})} ({\rm d} r) \big| \, .
\] 
The second  term in the above expression represents a continuous process of finite variation on compact intervals; indeed, the process 
$\,
 \int_{0}^{\infty} L^{\Vert X \Vert} (\cdot \,, r)  \, \big| D^{2} g_{\theta} ({\rm d} r) \big|\,$ has these properties   for every fixed $\, \theta \in [0,2\pi) ,$  and the set $\, \{ \ell: \,\, \tau^\eta_{2 \ell + 1} < T\}\,$ is almost surely finite. On the other hand, the first  term can be written as 
\[
\int_{0}^{T} {\bf 1}_{ \{ 0< \Vert X(t)\Vert \leq \eta \} } \int_{0}^{\infty}   L^{\Vert X \Vert} ({\rm d}t , r)  \,  \mu ({\rm d} r) \, = \,  \int_{0}^{\eta}   L^{\Vert X \Vert} (T , r)  \,  \mu ({\rm d} r) \, 
\]
via interchanging first the summation and the integration, then the two integrals; this is justified by the finiteness of the last expression above. It is now easy to see that the process given by (\ref{eq: processvari}) is well-defined, continuous and of finite variation on compact intervals.
 
For adaptedness, it is standard to show, by the \textsc{Borel}-measurability of $\, g \,$ and the joint measurability of $\, (t, r, \omega) \mapsto L^{\Vert X \Vert} (t , r, \omega) \,$,  that for any $\, T \in [0,\infty)\,$ the mapping
\[
(t, \theta ,  \omega) \, \longmapsto  \, \int_{0}^{\infty} L^{\Vert X \Vert} (t , r, \omega) \, D^{2} g_\theta ({\rm d} r) 
\]
is $\, \mathcal{B} \big( [0,\infty) \big) \otimes \mathcal{B} \big( [0,2\pi) \big) \otimes \mathcal{F} (T)-$measurable when restricted to $\, [0,\infty) \times [0,T] \times \Omega \,$. Let $\, (s_{k,T} , t_{k,T}) \,,$  $\, k \in \mathbb{N}\,$ be an enumeration of all excursion intervals of the path $\, \Vert X(t) \Vert\,$, $\, 0 < t < T\,$ away from $\, 0\,$, such that all random variables  $\, \big\{s_{k,T}  \big\}_{  k \in \mathbb{N}}\,$, $\, \big\{  t_{k,T}\big\}_{  k \in \mathbb{N}}\,$  are   $\,\mathcal{F} (T)-$measurable. Let $\, \Theta (t) = \theta_{k,T}\,$ for all $\, t\in (s_{k,T} , t_{k,T})\,$, and thus $\, \theta_{k,T}\,$ is also $\,\mathcal{F} (T)-$measurable for every $k \in \N$. Since (\ref{eq: processvari}) may be rewritten as
\[
\sum_{k=1}^{\infty}  \int_{0}^{\infty} \Big( L^{\Vert X \Vert} (t_{k,T} , r, \omega) - L^{\Vert X \Vert} (s_{k,T} , r, \omega) \Big) \, D^{2} g_{\theta_{k,T}} ({\rm d} r) \, ,\qquad 0 \leq T < \infty \, ,
\]
it is thus adapted to the filtration $\, \mathbb{F}\,$. Step 1 is now complete.

\smallskip
\noindent
{\it Step 2:}   With Proposition \ref{prop: D} and Lemma \ref{lm: gprime} having been  extended, we can follow exactly the same arguments as in the proof of Theorem \ref{Thm: Gen}, to prove (\ref{eq: ito}) and (\ref{eq: VXg2}); we note here that  Theorem 3.7.1(v) in \cite{KS} should be used here for the generalized \textsc{It\^o}'s rule. Finally, we  observe that any function $\, f \,$ as in Lemma \ref{lm: finite} is the second derivative (in the sense of Definition \ref{def: G}) of some function in $\, \mathfrak{D}\,$, hence also in $\,\mathfrak{C}\,$. Thus,  both Theorem \ref{Thm: Gen} and the just obtained (\ref{eq: ito}) apply;  comparing the results, we obtain (\ref{eq: density}). \qed


\bigskip
\noindent
{\bf PROOF OF PROPOSITION \ref{nonexplo}:}
Let $\, E_{t} : = \big \{ \int_{0}^{t \wedge S} {\bf 1}_{\{ X(u) \neq {\bm 0}\} } \, {\bm s}^{2}(X(u)) {\rm d} u \, = \infty \big \}\,$. Following the idea of the solutions to Problem 3.4.11 and Problem 5.5.3 in \cite{KS}, we have $\, \underline{\lim}_{n\rightarrow\infty} \Vert X(t \wedge S_{n})\Vert = 0 \,$  and $\,\overline{\lim}_{n\rightarrow\infty} \Vert X(t \wedge S_{n})\Vert = \infty \,,$ a.e. on $\, E_{t}\,$. Thus $\, \mathbb{P}(E_{t}) = 0\,$ by the continuity  of $\, X(\cdot)\,$  in the tree-topology.

\smallskip
Therefore,   $\, \int_{0}^{t \wedge S} {\bf 1}_{\{ X(t) \neq {\bm 0}\} } \, {\bm s}^{2}(X(u)) {\rm d} u < \infty\,$ holds a.e., and we obtain the existence in $\, \R^{2} \,$ of the limit $\, \lim_{n\rightarrow\infty} X(t \wedge S_{n}) \,$ in the tree-topology, in the same spirit as in the second-to-last paragraph in the proof of Proposition \ref{prop: DDS}. Thus $\, X(t \wedge S)\,$ is valued in $\, \R^{2} \,,$ a.e., for every $\, t\geq 0\,$, and consequently $\, S= \infty \,$ a.e. \qed

\bigskip
\noindent
{\bf PROOF OF THEOREM \ref{Thm: Walshnodrift}:} Omitting from the notation  the underlying probability space, we begin with a standard one-dimensional Brownian motion $\, \{ \widetilde{B}(s), \widetilde{\mathcal{G}}(s) ; 0\leq s < \infty \} \,$ and an independent two-dimensional random variable $\, \xi\,$ with distribution $\, \bm\mu\,$. Let $\, \{ Z(\cdot), \mathcal{G}(\cdot) \}\,$ be a \textsc{Walsh} Brownian motion starting at $\, Z(0) = \xi \,$  and driven by the Brownian motion $\, B(\cdot) = \Vert \xi \Vert + \widetilde{B}(\cdot)\,,$   with angular measure $\,\bm\nu\,$. This \textsc{Walsh} Brownian motion can be constructed as in the proof of Theorem 2.1 in \cite{IKPY} (even though in that proof the process starts at a nonrandom point, the same method applies to a random initial condition). Let 
\[
T(s): = \int_{0}^{s+} \frac{{\bf 1}_{\{ Z(u) \neq {\bm 0}\} }\, {\rm d} u}{{\bm s}^{2}\big(Z(u)\big)} \, , \,\,\, ~0 \leq s < \infty , \qquad A(t): = \inf \big \{ s \geq 0 : T(s) > t \big \} \, , \,\,\, 0 \leq t < \infty \, .
\]

\begin{lm}
\label{lm: infinite}
We have $\, T(\infty) = \infty\,$, a.s.
\end{lm}

\noindent
{\it Proof of Lemma \ref{lm: infinite}:} Consider the stopping times $\,\{\tau_{k}^{\varepsilon}\}_{k \in \mathbb N_{-1}}\,$ as in (\ref{eq: tau rec}), with $\, X \,$ replaced by $\, Z \,$. Since $\, Z(\cdot) \,$ is time-homogeneous strongly-Markovian (as a \textsc{Walsh} Brownian motion) and $\, Z(\tau_{2 m}^{\varepsilon}) \equiv {\bm 0} ,\,\, \forall \, m \in \mathbb{N}_{0} \,$, we deduce that the random variables 
$$\,
 \widehat{T}_{m} \,: = \, \int_{\tau_{2 m}^{\varepsilon}}^{\tau_{2 m +2 }^{\varepsilon}}\frac{{\bf 1}_{\{ Z(u) \neq {\bm 0}\} }\,{\rm d} u}{{\bm s}^{2}(Z(u))}\,  , \qquad \,\, m \in \mathbb{N}_{0} 
$$ are I.I.D and strictly positive. Therefore, we have $\, T(\infty) \geq \sum_{m \in \mathbb{N}_{0}} \widehat{T}_{m} = \infty \,$, a.e. \qed\newpage

\smallskip
We also note that $\, T(\cdot)\,$ is strictly increasing when it is finite, because $\, Z (\cdot)\,$ spends zero amount of time at the origin $\, \bm 0\,$. Now it is easy to see that the analogue of relationships (5.10)-(5.14) at the beginning of Section 5.5.A in \cite{KS}, as well as the discussions between them, all hold   here as well. Define
\begin{equation}
\label{eq: R}
R: = \inf \big \{ s \geq 0 : \, Z(s) \in \mathcal{I}({\bm s})   \big \} .
\end{equation}

\begin{lm}\label{lm: R=Ainfty}
We have $\, R = A(\infty)\,$, a.s. 
\end{lm}
\noindent
{\it Proof of Lemma \ref{lm: R=Ainfty}:} The proof of $\, R \leq A(\infty)\,$ follows   as in the proof of Lemma 5.5.2 in \cite{KS}, with the help of Condition \ref{sigma}, Lemma \ref{lm: finite}, and the tree-metric.

\smallskip
As for the reverse inequality $\, A(\infty) \leq R\,$, it suffices to prove it on the event $\, \{ R \leq n \} \,$ for every $\, n \in \mathbb{N}\,$. We define the standard Brownian motion $\, B_{n} (\cdot) : = B \big( (R \wedge n) + \cdot \big) - B(R \wedge n)  ,$ and the stopping time $\, \tau : = \{ s \geq 0 : B_{n} (s) \leq - \eta \}\,$. Then on the event $\, \{ R \leq n \}\,$, we have for any $\, 0 < s < \tau \,$ the comparison 
\[
\int_{0}^{R+s} \frac{{\bf 1}_{\{ Z(u) \neq {\bm 0}\} }\,{\rm d} u }{{\bm s}^{2} \big(Z(u)\big)} \,\geq\, \int_{R}^{R+s} \frac{{\bf 1}_{\{ Z(u) \neq {\bm 0}\} }\,{\rm d} u }{{\bm s}^{2} \big(Z(u)\big)} \, = \, \int_{0}^{s} \frac{{\rm d} u}{{\bm s}^{2} \big( \Vert Z(R) \Vert + B_{n}(u) , \text{arg}\big(Z(R)\big)\big)}\, .
\]
The last equality comes here from the fact $\, Z(\cdot) \neq {\bm 0} \,$ holds on the interval $\, [R, R+ \tau) \,$, which is because $\, Z(R) \in \mathcal{I}({\bm s}) \subseteq \{(r, \theta):  r \geq \eta , \,\, 0 \le \theta < 2 \pi\} $. It follows from Lemma 3.6.26 in \cite{KS} that the last integral above is infinite, thus $\, T(R) = \infty\,$ holds on $\, \{ R \leq n \} \,,$ and therefore $\, A(\infty) \leq R\,$ holds on $\, \{ R \leq n \} \,$. \qed

\smallskip
\noindent
$\bullet~$
Now we adapt the proof of Theorem 5.5.4 in \cite{KS}; i.e., we shall show that, under the Condition \ref{sigma},  a  \textsc{Walsh} Diffusion with state-space  $\, I \,$ associated with the triple $\, ({\bf 0}, {\bm s}, {\bm \nu}) \,$ exists, if and only if $\,\,  \mathcal{I}({\bm s}) \subseteq \mathcal{Z}({\bm s}) $. 

\noindent
(i) Let us first assume $\,  \mathcal{I}({\bm s}) \subseteq \mathcal{Z}({\bm s})\,$ and define
\begin{equation}
\label{eq: Xtimechange}
X(t): = Z(A(t)) , \qquad U(t): = B(A(t)), \qquad \mathcal{F}(t):= \mathcal{G}(A(t)) , \quad \qquad 0\leq t< \infty . 
\end{equation}
It follows that $\, X(T(u)) = Z(u)\,$ for $\, u < A(\infty)\,$. Thus,  for every $\, t \in [0, \infty)\,$ we have 
\begin{equation}
\label{eq: Xzerolebe}
\int_{0}^{t} \, {\bf 1}_{\{  X(v)  = {\bm 0} \} }\, {\rm d} v \, = \, \int_{0}^{A(t)} \, {\bf 1}_{\{  X(T(u))  = {\bm 0} \} }\, {\rm d} T(u) \, = \,  \int_{0}^{A(t)} \, {\bf 1}_{\{  Z(u)  = {\bm 0} \} }\, {\rm d} T(u) \, = \, 0 \, , 
\end{equation}
verifying (\ref{conditionWalsh}). Moreover, with (\ref{eq: Xzerolebe}) and all the previous preparations, we can proceed   as in the proof of Theorem 5.5.4 in \cite{KS}, and obtain that the process $\, U(\cdot) - \Vert \xi \Vert \,$ is a scalar local martingale with $\,\langle U \rangle (\cdot) = A(\cdot) \,$, as well as the representation
\begin{equation}
\label{eq: Aquadratic}
A(t) = \, \int_{0}^{t}{\bf 1}_{\{ X(v) \neq {\bm 0}\} }\, {\bm s}^{2}(X(v)) \,{\rm d} v \, , \qquad\qquad 0\leq t < \infty \,.
\end{equation}
Then there exists a Brownian motion $\, W(\cdot)\,$ on a possibly extended probability space, with the property  $\, U(t) = \Vert \xi \Vert + \int_{0}^{t} {\bf 1}_{\{ X(v) \neq {\bm 0}\} }\,{\bm s}(X(v)) {\rm d} W(v) \,$, $\, \,\,0\leq t < \infty\,$. 

\smallskip
Let us note that $\, \Vert Z(\cdot) \Vert\,$ is the \textsc{Skorokhod} reflection of $\, B(\cdot)\,$; thus the same relationship is true for $\, \Vert X(\cdot) \Vert\,$ and $\, U(\cdot)\,$  by (\ref{eq: Xtimechange}), and so  (\ref{eq: ||X||2}) gives 
\begin{equation}
\label{eq: ||X||nodrift}
\Vert X(\cdot) \Vert = \Vert \xi \Vert + \int_{0}^{\,\cdot} {\bf 1}_{\{ \Vert X(t) \Vert >0 \} } \,  {\bm s}\big( X(t)\big) {\rm d} W(t)  + L^{\Vert X \Vert}(\cdot)\, .
\end{equation}
Finally, the ``partition of local time" property and the continuity in the tree-topology  for $\, X  ( \cdot) \,$ are both inherited from $\, Z(\cdot)\,$, as the proof of Proposition \ref{prop: DDS} illustrates. We have thus verified that the just constructed $\, X   ( \cdot)\,$ is a \textsc{Walsh} diffusion as described in the Theorem.

 \smallskip
\noindent
(ii) Conversely, let us assume the existence of the \textsc{Walsh} diffusion $\,X   ( \cdot)\,$ described in Theorem \ref{Thm: Walshnodrift}, with any given initial condition. Consider such a \textsc{Walsh} diffusion $\,X   ( \cdot)\,$ with $\, X(0) = x \in \mathcal{Z}({\bm s})^{c}\,$ and the underlying Brownian motion $\, W(\cdot)\,$. We introduce the scalar local martingale 
\begin{equation}
\label{eq: U}
U(\cdot) : = \Vert X(\cdot) \Vert - L^{\Vert X \Vert}(\cdot)  = \Vert X(0) \Vert + \int_{0}^{\,\cdot}{\bf 1}_{\{ \Vert X(t) \Vert >0 \} }\, {\bm s}(X(t)) {\rm d} W(t) \, . \quad
\end{equation}\newpage
Then $\, \Vert X(\cdot) \Vert\,$ is the \textsc{Skorokhod} reflection of $\, U(\cdot)\,$, and therefore $\, X(\cdot)\,$ is a \textsc{Walsh} semimartingale driven by  $  U(\cdot)$. By Proposition \ref{prop: DDS}, there exists a \textsc{Walsh} Brownian motion $\, Z(\cdot)\,$ on a possibly extended probability space, such that $\, X(\cdot) = Z(\langle U \rangle (\cdot))\,$. We can follow the proof of Theorem 5.5.4 in \cite{KS} with $\, T(s) : = \inf\{ t \geq 0 : \langle U \rangle (t) > s \}\,$, first to derive that 
\[
\int_{0}^{s \wedge \langle U \rangle (\infty)} \, \frac{{\bf 1}_{\{ Z(u) \neq {\bm 0}\}} \, {\rm d} u }{{\bm s}^{2}(Z(u))} \, = \, \int_{0}^{T(s)}   \frac{{\bf 1}_{\{ X(v) \neq {\bm 0} \} } \, {\rm d} \langle U \rangle (v)}{{\bm s}^{2}(X(v))}\, = \, \int_{0}^{T(s)} {\bf 1}_{\{ X(v) \neq {\bm 0} , \,\,\, {\bm  s}(X(v)) \neq 0\} } {\rm d} v \, \leq T(s)  \quad
\]
holds for all $\, 0 \leq s < \infty\,$, then to argue   $\, \mathbb{P} \big( T(s) < \infty, \,\, \langle U \rangle (\infty) > 0 \big) \, > 0 \,$ for sufficiently small $\, s> 0 \,$, and finally to show that    $\, x \in \mathcal{I}({\bm s})\,$ cannot hold. It follows that $\,  \mathcal{I}({\bm s}) \subseteq \mathcal{Z}({\bm s})\,$. 

\smallskip
\noindent
$\bullet ~$
Next, we assume the validity  of Condition \ref{sigma} and $\,  \mathcal{I}({\bm s}) \subseteq \mathcal{Z}({\bm s})\,$, and show that uniqueness in distribution is then equivalent to the condition $\,  \mathcal{I}({\bm s}) \supseteq \mathcal{Z}({\bm s})\,$. 

\noindent
(i) First, we suppose that the inclusion $\,  \mathcal{I}({\bm s}) \supseteq \mathcal{Z}({\bm s})\,$ does not hold.  
By picking a starting point $\, x \in \mathcal{Z}({\bm s}) \setminus \mathcal{I}({\bm s}) \,$, we see that uniqueness in distribution is violated for the \textsc{Walsh} diffusion described in Theorem \ref{Thm: Walshnodrift} and starting at $\, x \,$, in the spirit of Remark 5.5.6 in \cite{KS}. 

\noindent
(ii) Conversely, let us assume in addition that $\,  \mathcal{I}({\bm s}) \supseteq \mathcal{Z}({\bm s})\,$ holds. Let $\, X (\cdot) \,$ be a \textsc{Walsh} diffusion described in Theorem \ref{Thm: Walshnodrift} and with an arbitrarily given initial distribution $\, \bm \mu \,$. With $\, U(\cdot)\,$ as in (\ref{eq: U}), we can adapt the proof of Theorem 5.5.7 in  \cite{KS} in a manner similar to what we did before, and obtain   the existence of a \textsc{Walsh} Brownian motion $\, Z(\cdot)\,$  such that $\, X(\cdot) = Z(\langle U \rangle (\cdot)) \,$ and    
\begin{equation}
\label{eq: <U>}
\langle U \rangle (t) \, = \, \inf \Big \{ s\geq 0: \, \int_{0}^{s+} \frac{{\bf 1}_{\{ Z(u) \neq {\bm 0} \} }{\rm d} u}{{\bm  s}^{2}(Z(u))} > t \Big\} \, , \qquad 0 \leq t < \infty \, .
\end{equation}
It develops that the process $\, X(\cdot)\,$ can be expressed as a measurable functional of the \textsc{Walsh} Brownian motion $\, Z(\cdot)\,$, with initial distribution $\, \bm \mu \,$ and angular measure $\, \bm \nu \,$. Since this $\, Z(\cdot)\,$ has a uniquely determined probability distribution, thanks to Proposition 7.2 in \cite{IKPY} (again, this can be generalized from a nonrandom starting point to a random initial condition), we deduce the uniqueness of $\, X(\cdot)\,$ in distribution.   \qed

\bigskip
\noindent
{\bf ON THE PROOF OF THEOREM \ref{Thm: Sfinite}:} 
We need some   preparation before proving Theorem \ref{Thm: Sfinite}. By analogy with  Section 5.5.C in \cite{KS}, we define a sequence $\, \{ u_{n} \}_{n=0}^{\infty}\,$ of functions on $\, I \,$ via $\, u_{0} \equiv 1\,$ and 
\begin{equation}
\label{eq: un}
u_{n}(r, \theta) := \, \int_{0}^{r} p_{\theta}^{\prime}(y) \int_{0}^{y} u_{n-1}(z, \theta) {\bm m}_{\theta}\,({\rm d} z) \, , \qquad\quad (r, \theta) \in I \, , \quad n \in \mathbb{N}, 
\end{equation}
recursively. Note that $\, u_{1} \equiv v \,$. We have the following analogue of Lemma 5.5.26 in \cite{KS}.

\begin{lm}\label{lm: u}
Under Condition \ref{bsigma}, the series 
\begin{equation}
\label{eq: u}
u (r, \theta) : =  \, \sum_{n=0}^{\infty} \, u_{n}(r, \theta) \, , \qquad \quad (r, \theta) \in I
\end{equation}
converges on $\, I \,$ and defines a function in the class $\, \mathfrak{D}_{I}\,$. Furthermore, for every $\, \theta \in [0, 2\pi) \,$, the mapping $\, r \mapsto u_{\theta}(r): = u (r, \theta) \,$ is strictly increasing on $\, [0, \ell(\theta)) \,$, and satisfies $\, u_{\theta} (0) = 1 \,$, $\, u_{\theta}^{\prime} (0+) = 0 \,$, as well as 
\begin{equation}
\label{eq: utheta}
{\bm b} (r, \theta) u_{\theta}^{\prime}(r) + \frac{1}{2} {\bm s}^{2} (r, \theta) u_{\theta}^{\prime\prime}(r) \, = \, u_{\theta}(r) \, , \qquad \text{a.e.} \,\, \, r \in (0, \ell(\theta) )\, . \quad 
\end{equation}
 
Moreover, we have $\,\, 1+ v(x) \leq u(x) \leq e^{v(x)}\,$, $\, \forall \, x \in I \,$.
\end{lm}
\begin{proof}
Apart from (iii) of Definition \ref{def: D2} for the claim that $\, u \in \mathfrak{D}_{I}\,$, Lemma \ref{lm: u} can be proved in the same way as in the proof of Lemma 5.5.26 in \cite{KS}. And (iii) of Definition \ref{def: D2} for $\, u\,$ can be seen through Condition  \ref{bsigma}, (\ref{eq: utheta}), the fact $\, u_{\theta}(r)\leq e^{v_{\theta}(r)}\,$,  as well as the fact $\, u^{\prime}_{\theta}(r)  \leq  v^{\prime}_{\theta}(r) \cdot e^{v_{\theta}(r)}  \,$ derived from the proof of Lemma 5.5.26 in \cite{KS}.
\end{proof}

 \smallskip
 \noindent
 {\it Proof of Theorem \ref{Thm: Sfinite}:}  
 Thanks to Lemma \ref{lm: u}, we can apply   
 Theorem \ref{Thm: Gen}  to $\, u\,$  and obtain that the process $\, \{ e^{- t \wedge S_{n}} u (X(t \wedge S_{n})) ; \, 0 \leq t < \infty \} \,$ is a local martingale  for every $\, n \in \N\,$. But  this process is also nonnegative, thus   a supermartingale. Then we may let $\, n \rightarrow \infty\,$ to obtain that $\, \{ e^{- t \wedge S} u (X(t \wedge S)) ; \, 0 \leq t < \infty \} \,$ is a nonnegative supermartingale, thus
\begin{equation}
\label{eq: uXS}
\, \lim_{t \uparrow S} e^{- t } u (X(t)) \quad\,\, \text{exists and is finite,} \quad \mathbb{P}^{\bm 0}-\text{a.e.} \qquad\quad
\end{equation}

\noindent
{\it Proof of {\bf (i)}.} By (\ref{eq: explos}),  $ \,\lim_{t \uparrow S} X(t)$ exists in $\{ (r, \theta): \, r = \ell(\theta) , \,\, 0 \le \theta < 2 \pi\} $, $\, \mathbb{P}^{\bm 0}-$a.e. on $\, \{ S < \infty\}$. Since $\, {\bm \nu} (\{ \theta : \, v_{\theta}(\ell (\theta)-) < \infty \}) = 0$, Proposition \ref{prop: zeroangle} implies that $\, \lim_{t \uparrow S} v(X(t)) =  \infty $,  $\, \mathbb{P}^{\bm 0}-$a.e. on $\, \{ S < \infty\}$. Thus $\, \lim_{t \uparrow S} u(X(t)) = \infty $,  $\, \mathbb{P}^{\bm 0}-$a.e.  on $\, \{ S < \infty\}$, by Lemma \ref{lm: u}.  It follows that $\, \lim_{t \uparrow S} e^{- t } u (X(t)) = \infty \,$ holds $\, \mathbb{P}^{\bm 0}-$a.e. on $\, \{ S < \infty\}$. Comparing this with (\ref{eq: uXS}), we deduce $\, \mathbb{P}^{\bm 0} \big( S < \infty \big) = 0$.

\medskip
\noindent
{\it Proof of {\bf (ii)}.} ~With $$\, A^{p}:= \{ \theta: \, p_{\theta}(\ell (\theta)-) < \infty \}\,\qquad \text{and} \qquad \, A^{v}:= \{ \theta: \, v_{\theta}(\ell (\theta)-) < \infty \}\, ,$$ we have $\, A^{v} \subseteq  A^{p}\,$ by Proposition \ref{prop: v/p}(iii) and  $\, {\bm \nu} (A^{v}) > 0 \,$   by assumption, thus $\, {\bm \nu} (A^{p}) > 0 \,$. By Theorem \ref{Thm: XS}, the limit $\, \lim_{t \uparrow S} X(t)\,$ exists $\,\mathbb{P}^{\bm 0}-$a.e. in $\, \{ (r, \theta): \, r = \ell(\theta) , \,\, 0 \le \theta < 2 \pi\} \,$, and $\, \mathbb{P}^{\bm 0} \big( \Theta(S) \in A^{p} \big) =1 \,$. We also have the assumption $\, {\bm \nu} (A^{p} \setminus A^{v}) = 0 \,$, thus $\, \mathbb{P}^{\bm 0} \big( \Theta(S) \in A^{p} \setminus A^{v} \big) = 0 \,$ and therefore $\, \mathbb{P}^{\bm 0} \big( \Theta(S) \in A^{v} \big) =1 \,$. 
For every $\, n \in \mathbb{N}\,$, let us define  \begin{equation}
\label{eq: lvn}
\ell^{v}_{n} (\theta): = \,\sup \{ r: \, 0 \leq r < \ell(\theta), \,\, v_{\theta}(r) \leq n \}\, , \quad I_{n}^{v}: = \{ (r, \theta): \, 0 \leq r < \ell^{v}_{n} (\theta) , \,\, 0 \le \theta < 2 \pi\} \, , \,\,
\end{equation}
\begin{equation}
\label{eq: Svn}
S_{n}^{v} : = \, \inf \{ t\geq 0: \, \Vert X(t) \Vert \geq \ell^{v}_{n} (\Theta(t)) \} \, =\, \inf \{ t \geq 0: \, X(t) \notin I^{v}_{n} \} \, .  \qquad \qquad
\end{equation}
 By Proposition \ref{prop: finiteexpectexplo}, we have $\,\mathbb{E}^{\bm 0} [ S_{n}^{v} ] < \infty\,$, thus $\, \mathbb{P}^{\bm 0} \big( S_{n}^{v} < \infty \big) =1 \,$, $\, \forall\, n \in \mathbb{N}\,$.
 
\smallskip
Therefore, there is an event $\, \Omega^{\star} \in \cal F\,$ with $\, \mathbb{P}^{\bm 0} (\Omega^{\star}) =1 \,$, such that for every $\, \omega \in \Omega^{\star}\,$, we have that: $\, \lim_{t \uparrow S(\omega)} X(t, \omega) \,$ exists in $\, \{ (r, \theta): \, r = \ell(\theta) , \,\, 0 \le \theta < 2 \pi\} \,$; that $\, \Theta(S (\omega), \omega)\in A^{v}\,$; and that $\, S_{n}^{v}(\omega) < \infty \,$ for every $\, n \in \mathbb{N}\,$. We fix now an $\, \omega \in \Omega^{\star}\,$. Since  $\, \Theta(S (\omega), \omega)\in A^{v}\,$, the limit  $\, \lim_{t \uparrow S(\omega)} v (X(t, \omega)) \,$ exists and is finite. Thus we can choose $\, n(\omega) \in \mathbb{N}\,$, such that $\, n(\omega) > \sup_{t \in [0, S(\omega))} v (X(t, \omega))\,$.

\begin{claim}\label{cl: Svn=S}
We have $\, S_{n(\omega)}^{v}(\omega) = S(\omega)\,$, thus $\, S(\omega) < \infty \,$.
\end{claim}
\begin{proof}
Since $  S_{n(\omega)}^{v}(\omega) < \infty $, we have $  X( S_{n(\omega)}^{v}(\omega) , \, \omega)  \in \, \{ (r, \theta): \, r = \ell_{n(\omega)}^{v}(\theta) , \,\, 0 \le \theta < 2 \pi\}$. With $\, A^{v}_{n}: = \{ \theta: \, v_{\theta}(\ell (\theta)-) \leq n \} \,$ for every $\, n \in \mathbb{N}\,$, we claim that $\, \Theta ( S_{n(\omega)}^{v}(\omega) , \, \omega) \,\in \, A^{v}_{n(\omega)} \,$. 

Indeed, whenever $\, \theta \notin A^{v}_{n(\omega)}\,$, we have $\, v_{\theta}(\ell (\theta)-) > n(\omega) \,$ and therefore $\, v_{\theta}(\ell _{n(\omega)}^{v} (\theta) ) = n(\omega) \,$. But $\, v (X ( S_{n(\omega)}^{v}(\omega) , \, \omega)) \leq \sup_{t \in [0, S(\omega))} v (X(t, \omega)) < n(\omega)\,$, so we must have $\, \Theta ( S_{n(\omega)}^{v}(\omega) , \, \omega) \,\in \, A^{v}_{n(\omega)} \,$.

\smallskip
We also observe that, whenever $\, \theta \in A^{v}_{n(\omega)}\,$, we have $\, v_{\theta}(\ell (\theta)-) \leq n(\omega) \,$ and therefore $\, \ell _{n(\omega)}^{v} (\theta) = \ell (\theta) \,$. Thus the fact $\, \Theta ( S_{n(\omega)}^{v}(\omega) , \, \omega) \,\in \, A^{v}_{n(\omega)} \,$ implies that $\, X( S_{n(\omega)}^{v}(\omega) , \, \omega) \, \in \, \{ (r, \theta): \, r = \ell(\theta) , \,\, 0 \le \theta < 2 \pi\}\,$. We have then $\, S_{n(\omega)}^{v}(\omega) = S(\omega)\,$, and $\, S(\omega) < \infty \,$ follows.
\end{proof}
Since Claim \ref{cl: Svn=S} holds   for every $\, \omega \in \Omega^{\star}\,$, the proof of {\bf (ii)} is complete.

\medskip
\noindent
{\it Proof of {\bf (iii)}.} Since $\, {\bm \nu} (\{ \theta : \, v_{\theta}(\ell (\theta)-) < \infty \}) > 0\,$, we can choose an integer $\, N \in \mathbb{N}\,$, such that $\, A^{v}_{N} = \{ \theta: \, v_{\theta}(\ell (\theta)-) \leq N \} \,$ satisfies $\,{\bm \nu}(A^{v}_{N} ) > 0\,$.  Recalling (\ref{eq: lvn}) and (\ref{eq: Svn}), we have by Proposition \ref{prop: v/p}(iii) that $\, p_{\theta} (\ell_{N}^{v} (\theta)) < \infty \,$ for all $\,\theta \in [0, 2\pi)\,$. Then an application of Theorem \ref{Thm: XS} yields $\, \mathbb{P}^{\bm 0} \big( \Theta(S_{N}^{v}) \in A^{v}_{N} \big) > 0 \,$.

\smallskip
 We have also $\, S_{N}^{v} = S \,$, $\, \mathbb{P}^{\bm 0}$-a.e. on $\, \{ \Theta(S_{N}^{v}) \in A^{v}_{N} \}\,$, in light of the last paragraph of the proof of Claim \ref{cl: Svn=S}. Thus $\,\mathbb{P}^{\bm 0} \big( S_{N}^{v} =S \big) > 0 \,$. But $\, \mathbb{P}^{\bm 0} \big( S_{N}^{v} < \infty \big) =1 \,$, so $\, \mathbb{P}^{\bm 0} \big( S < \infty \big)> 0 \,$ follows. 
 
\smallskip
It remains only to show that $\, \mathbb{P}^{\bm 0} \big( S < \infty \big) <1 \,$ holds under the assumptions of {\bf (iii)}. We have $$\, {\bm\nu}(A^{p}\setminus A^{v} )  =  {\bm\nu} \big( \big\{ \theta: \, v_{\theta}(\ell (\theta)-) = \infty , \, p_{\theta}(\ell (\theta)-) < \infty \big\} \big)> 0  \,$$ by assumption. Another application of 
Theorem \ref{Thm: XS} yields that $\, \lim_{t \uparrow S} X(t)\,$ exists $\,\mathbb{P}^{\bm 0}-$a.e. in the set $\, \{ (r, \theta): \, r = \ell(\theta) , \,\, 0 \le \theta < 2 \pi\} \,$, and that $\, \mathbb{P}^{\bm 0} \big( \Theta(S) \in A^{p}\setminus A^{v} \big) > 0 \,$. But since $\, \lim_{t \uparrow S} v(X(t)) = \infty\,$ and therefore $\, \lim_{t \uparrow S} u(X(t)) = \infty \,$ on $\, \{ \Theta(S) \in  A^{p}\setminus A^{v} \}\,$, we may recall (\ref{eq: uXS}) to obtain $\, S = \infty \,$, $\,\mathbb{P}^{\bm 0}-$a.e. on $\, \{ \Theta(S) \in  A^{p}\setminus A^{v} \}\,$. It follows that $\, \mathbb{P}^{\bm 0} \big( S= \infty \big)> 0 \,$, thus $\, \mathbb{P}^{\bm 0} \big( S < \infty \big) <1 \,$. \qed


\bigskip
\noindent
{\bf PROOF OF PROPOSITION \ref{prop: Uc}:} {\bf (i).} {\it Step 1.} We shall show in this step that for every $\, \theta \in [0,2\pi)\,$, $\, U^{(c)}_{\theta}(\cdot)\,$ is continuous on $\, [0,1] \,$ with $\, U^{(c)}_{\theta}(0) = c \,$ and $\, U^{(c)}_{\theta}(1) = U_{\theta}(1) \,$. It is easy to show that, $\, U^{(c)}_{\theta}(\cdot)\,$ itself is also $\, p_{\theta}-$concave and therefore continuous on $\, (0,1) \,$ with finite limits at the two endpoints, such that $\, \lim_{r \downarrow 0} U^{(c)}_{\theta}(r) \geq U^{(c)}_{\theta}(0) \geq c \,$. Thus to finish this step, it suffices to show $\,\lim_{r \downarrow 0} U^{(c)}_{\theta}(r)\leq c\,$ (the situation at $\,r \uparrow 1\,$ can be treated in the same way, thanks to condition (\ref{eq: pfinite})). 

We  need only construct, for every $\, c^{\prime} > c \,,$ a continuous and $\, p_{\theta}-$concave function $\, \varphi\,$ on $\, [0,1]\,$ with $\, \varphi(\cdot) \geq U_{\theta}(\cdot)\,$ and $\,\varphi (0) = c^{\prime} \,$. Let $\, M: = \sup_{x \in \overline{B}} \vert U(x) \vert < \infty\,$. If $\, c^{\prime} \geq M\,$, we take $\, \varphi \equiv c^{\prime}\,$. If $\, c^{\prime} < M\,$, by the continuity of $\, U_{\theta} (\cdot) \,$, we choose $\, r^{\prime} > 0\,$ such that $\, U_{\theta} (\cdot) \leq c^{\prime}\,$ on $\, [0, r^{\prime} ]\,$, and take 
$$\, \varphi (r) = c^{\prime} + \Big(M-c^{\prime} \Big) \,\frac{p_{\theta}(r)}{p_{\theta}(r^{\prime})}\, , \quad r \in [0, r^{\prime}] \, ; \qquad \quad \varphi (r) = M \, , \quad r \in [r^{\prime} , 1]\, .$$ 
\smallskip
\noindent
{\it Step 2.} By Step 1, the only remaining issue in proving (i), is the continuity in the tree-topology at the origin. By $\, p_{\theta}-$concavity we have
\begin{equation}
\label{eq: Ulowerbound}
U^{(c)}_{\theta} (r) \,\geq\, U^{(c)}_{\theta} (1) \, \frac{p_{\theta}(r)}{p_{\theta}(1-)} + c \Big(1- \frac{p_{\theta}(r)}{p_{\theta}(1-)}\Big) \,\geq\, c - \big(c+M\big) \, \frac{p_{\theta}(r)}{p_{\theta}(1-)}\, , \quad  r \in (0,1).
\end{equation}
Since $\, p\,$ is continuous in the tree-topology and $\, p_{\theta}(1-) \,$ is bounded away from zero, we see that $$\, \lim_{r \downarrow 0} \inf_{\tilde{r} \leq r, \, \theta \in [0,2\pi)} U^{(c)}_{\theta} (\tilde{r}) \geq c\, .$$

On the other hand,  since $\, U \,$ is continuous in the tree-topology, given $\,\varepsilon > 0\,$ we can choose $\,\delta >0\,$  such that $\, U (r,\theta) \leq U({\bm 0}) + \varepsilon \,$ for all $\, r \leq 2 \delta\,$ and $\,\theta \in [0,2\pi)\,$. Fixing $\, r \leq \delta\,$ and   $\, \theta\in[0,2\pi)\,$, we distinguish two cases:

\smallskip
\noindent
{\it Case 1.} $\, U^{(c)}_{\theta} (r) = U (r,\theta) \,$. \\ Then $\, U^{(c)}_{\theta} (r) \leq U({\bm 0}) + \varepsilon \leq c+\varepsilon \,$.

\smallskip
\noindent
{\it Case 2.} {\it The point 
$  r $  belongs to some connected component $\, (r_{1}, r_{2})\,$ of the set} $\, \{ \rho \in (0,1): \, U^{(c)}_{\theta} (\rho) > U (\rho,\theta) \}\,$.  By (ii) of this proposition (whose proof will not use the continuity of $\, U^{(c)}\,$ at the origin under the tree-topology), $\, U^{(c)}_{\theta}\,$ is a linear function of $\,p_{\theta}\,$ on $\, [r_{1}, r_{2}]\,$.  

If $\, r_{2} \leq 2 \delta\,$, then $\, U^{(c)}_{\theta} (r_{i}) = U_{\theta} (r_{i}) \leq c+\varepsilon\,$ for $\,i =1,2\,$, and it follows that $\, U^{(c)}_{\theta} (r) \leq c+\varepsilon\,$. If on the other hand $\, r_{2} > 2 \delta\,$, then the slope of the just mentioned linear function $\, U^{(c)}_{\theta}\,$ does not exceed $\, \frac{ \max (M, c) + M}{p_{\theta}(2\delta) - p_{\theta}(\delta) }\,$, because $\, r_{1}< r \leq \delta\,$. Therefore,
$$ U^{(c)}_{\theta} (r) \leq U^{(c)}_{\theta} (r_{1}) + \frac{\max (M, c) + M}{p_{\theta}(2\delta) - p_{\theta}(\delta) } \big(p_{\theta}(r) - p_{\theta}(r_{1})\big) \leq c+\varepsilon + \frac{\max (M, c) + M}{p_{\theta}(2\delta) - p_{\theta}(\delta) }p_{\theta}(r) . $$

By Condition \ref{bsigma}, the mapping $\, \theta \mapsto p_{\theta}(2\delta) - p_{\theta}(\delta)\,$ is bounded away from zero when $\, 2\delta \leq \eta \,$. Thus we obtain $\, \lim_{\,r \downarrow 0} \sup_{\,\tilde{r} \leq r, \, \theta \in [0,2\pi)} U^{(c)}_{\theta} (\tilde{r}) \leq c+\varepsilon\,$ from the above two cases. It is now clear that $\, U^{(c)}\,$ is continuous at the origin in the tree-topology. 

 \medskip
\noindent
{\bf (ii).} Without loss of generality, we may assume $\, p_{\theta} (r) \equiv r\,$. By way of contradiction, we assume that there exist some $\,\theta \in [0,2\pi)\,$ and $\, (r_{1} , r_{2}) \subset [0,1]\,$, such that $\, U^{(c)}_{\theta} (r) > U_{\theta} (r) \,$ holds for $\, r \in (r_{1} , r_{2})\,$, yet $\, r\mapsto U^{(c)}_{\theta} (r)\,$ is {\it not} linear on $\, [r_{1} , r_{2}] \,$. We shall then construct a concave function $\, \varphi \,$ on $\, [0,1]\,$ that dominates $\, U_{\theta}\,$ and satisfies $\, \varphi (0) = c\,;$   yet does not dominate $\,U^{(c)}_{\theta}\,$, thus contradicting (\ref{eq: Uc}). \newpage  

Since $\, U^{(c)}_{\theta} (\cdot)\,$ is concave and not linear on $\, [r_{1} , r_{2}] \,$,  we have $\, D^{-} U^{(c)}_{\theta} (r_{2}) < D^{+} U^{(c)}_{\theta} (r_{1})\,$. Choose $\, [r_{3}, r_{4}] \subset (r_{1}, r_{2})\,$ with $\, D^{-} U^{(c)}_{\theta} (r_{4}) < D^{+} U^{(c)}_{\theta} (r_{3})\,$, and we have 
$$
\, \min_{r \in [r_{3}, r_{4}]} \big(U^{(c)}_{\theta} (r) - U_{\theta}(r) \big) > 0 \, .
$$ 
Thus by dividing $\, [r_{3}, r_{4}]\,$ into small enough subintervals, we can find $\, [r_{5}, r_{6}] \subset [r_{3}, r_{4}]\,$, such that 
\begin{equation}
\label{eq: r5r6}
\, D^{-} U^{(c)}_{\theta} (r_{6}) < D^{+} U^{(c)}_{\theta} (r_{5})\, \qquad \text{and} \qquad  \, \max_{r \in [r_{5}, r_{6}]} U_{\theta}(r) < \min_{r \in [r_{5}, r_{6}]}U^{(c)}_{\theta} (r)\, .
\end{equation}
Let $\, \varphi :[0,1] \rightarrow \R\,$ be a linear function on $\, [r_{5},r_{6}]\,$ which equals $\, U^{(c)}_{\theta}  \,$ on $\, [0,1]\setminus (r_{5}, r_{6})\,$. Then $\, \varphi\,$ is concave and satisfies $\, \varphi (0) = U^{(c)}_{\theta}(0) = c\,$; also,  the two inequalities of (\ref{eq: r5r6}) imply $\, \varphi (r) < U^{(c)}_{\theta} (r)\,$ for $\, r \in (r_{5}, r_{6})\,,$ and $\, \varphi (\cdot)\geq U_{\theta}(\cdot)\,$, respectively. Thus $\, \varphi\,$ is our desired function that leads to the contradiction.   

\medskip
\noindent
{\bf (iii).} Property (i) in Definition \ref{def: D3} is obvious for $\,U^{(c)}\,$. For the \textsc{Borel}-measurability, we may write (in the spirit of the proposition in Section 3 of \cite{KS1})
\begin{equation}
\label{eq: Ucalter}
U^{(c)}_{\theta} (r) \, = \, \inf \big\{ a_{n} (\theta) p_{\theta}(r) + b_{n} (\theta)\,, \,\, \, n \in \mathbb{N} \big\} ,
\end{equation}
where $\, \{ a_{n} \}_{ n \in \mathbb{N} }\,$, $\,\{ b_{n} \}_{ n \in \mathbb{N} }\,$ are two sequences of measurable functions on $\, [0,2\pi)\,$, such that for every $\,\theta\in[0,2\pi)\,$, the set $\, \{ (a_{n} (\theta), b_{n} (\theta)), \,\, n \in \mathbb{N} \}\,$ is the collection of all rational pairs $\, (a, b) \,$ with $\, b \geq c\,$, and for which $\, a p_{\theta} (\cdot) + b \,$ dominates $\, U_{\theta}(\cdot)\,$. This is due to the continuity in $\, r \,$ and the measurability of both $\, (r,\theta) \mapsto p_{\theta}(r) \,$ and $\, (r,\theta) \mapsto U_{\theta}(r) \,$. The representation (\ref{eq: Ucalter}) yields the \textsc{Borel}-measurability of $\,U^{(c)}\,$.   

 \smallskip
Now let us assume $\, c> U({\bm 0})\,$. Since both functions $\,U^{(c)}\,$ and $\, U\,$ are continuous in the tree-topology, we can find an $\,\eta > 0\,$, such that $\,U^{(c)}_{\theta}(\cdot) > U_{\theta}(\cdot)\,$ on $\, [0,\eta)\,$, for all $\,\theta\in[0,2\pi)\,$. Hence, we may write $\,U^{(c)}_{\theta}(r) = a_{\theta}\, p_{\theta}(r) + c\,$ for $\, r \in [0,\eta]\,$, and thus 
\[
a_{\theta} = \frac{U^{(c)}_{\theta}(\eta) - c }{p_{\theta}(\eta)} ~~\qquad \text{but}\qquad 
-\frac{c+M}{p_{\theta}(\eta)} \,\leq \,\frac{U^{(c)}_{\theta}(\eta) - c }{p_{\theta}(\eta)} \, \leq \, \max \Big(0, \frac{M-c}{p_{\theta}(\eta)}\Big)\, .
\]
As the function $\, \theta \mapsto p_{\theta}(\eta)\,$ is bounded away from zero, we see that $\,\theta \mapsto a_{\theta}\,$ is bounded. Thus property (ii) in Definition \ref{def: D3} holds for $\,U^{(c)}\,$. Property (iii) also follows, using in addition that $\, p \in \mathfrak{C}_{B}\,$.

\medskip
\noindent
{\bf (iv).} The inequality (\ref{eq: Ulowerbound}) shows that the function  $\,\theta \mapsto D^{+} U^{(c)}_{\theta}(0)\,$ is bounded from below, so the function $\, \Phi\,$ is well-defined by \eqref{eq: phi} and takes values in $\,\R \cup \{ \infty\}\,$. In fact, from the just proved property (iii), we see that $\, \Phi\,$ takes the value $\,\infty\,$ only possibly at $\, U({\bm 0})\,$.

For the other two claimed properties for $\, \Phi\,$, it suffices to show that the mapping $\, c \mapsto D^{+} U^{(c)}_{\theta}(0)\,$ is continuous and strictly decreasing for every $\,\theta \in [0,2\pi)\,$. Fix $\,\theta \in [0,2\pi)\,$ and consider $\, c_{2} > c_{1} \geq U({\bm 0})\,$. With 
$$\, 
r_{2,\theta} \,: = \, \inf \big \{ r \in [0,1]: \, U^{(c_{2})}_{\theta}(r) = U_{\theta}(r) \big \} > 0 \, ,
$$ 
the function  $\,U^{(c_{2})}_{\theta}(\cdot)\, $ is a linear transformation of $\, p_{\theta}(\cdot) \,$ on $\, [0, r_{2,\theta}]\,,$ and   $\, U^{(c_{2})}_{\theta}(r_{2,\theta}) = U_{\theta}(r_{2,\theta}) \,$. Hence   
\[
D^{+} U^{(c_{2})}_{\theta}(0) = \frac{U_{\theta}(r_{2,\theta}) - c_{2}}{p_{\theta}(r_{2,\theta})}~~ \qquad \text{and} \qquad ~~
D^{+} U^{(c_{1})}_{\theta}(0) \geq \frac{U^{(c_{1})}_{\theta}(r_{2,\theta}) - c_{1}}{p_{\theta}(r_{2,\theta})} > \frac{U_{\theta}(r_{2,\theta}) - c_{2}}{p_{\theta}(r_{2,\theta})}\, ,
\]
thanks to $\, p_{\theta}-$concavity; we have also used the fact $\, p^{\prime}_{\theta} (0+) \equiv 1\,$.  We have thus obtained the strict decrease of the mapping  $\, c \mapsto D^{+} U^{(c)}_{\theta}(0)\,$. Therefore, we may let $\, c_{2} \downarrow c_{1} \,$ then $\,r\downarrow 0\,$ in the observation
\[
D^{+} U^{(c_{2})}_{\theta}(0) \geq \frac{U^{(c_{2})}_{\theta}(r) - c_{2}}{p_{\theta}(r)} \geq \frac{U^{(c_{1})}_{\theta}(r) - c_{2}}{p_{\theta}(r)}\, ,
\]
and obtain the right-continuity of $\, c \mapsto D^{+} U^{(c)}_{\theta}(0)\,$. \newpage

To show   left-continuity, we assume   $\, c_{2} > c_{1} > U({\bm 0})\,$ and set $\, r_{1,\theta} : = \inf \big\{ r \in [0,1]: \, U^{(c_{1})}_{\theta}(r) = U_{\theta}(r) \big\} > 0 \,$. It follows that $\,U^{(c)}_{\theta}(\cdot)\,$ is a linear transformation of $\, p_{\theta}(\cdot) \,$ on $\, [0, r_{1,\theta}]\,$ whenever $\, c \geq c_{1}\,$. Thus for $\, r \in (0, r_{1,\theta}]\,$, we have
\[
 D^{+} U^{(c_{2})}_{\theta}(0) = \frac{U^{(c_{2})}_{\theta}(r) - c_{2}}{p_{\theta}(r)}\,, \qquad \text{and} \qquad 
D^{+} U^{(c)}_{\theta}(0) = \frac{U^{(c)}_{\theta}(r) - c}{p_{\theta}(r)} \leq \frac{U^{(c_{2})}_{\theta}(r) - c}{p_{\theta}(r)} \, ,\quad c \in [c_{1}, c_{2}] \, .
\]
Letting $\, c \uparrow c_{2} \,$, we obtain the left-continuity of $\, c \mapsto D^{+} U^{(c)}_{\theta}(0)\,$. \qed

\end{document}